\documentclass[a4paper,11pt]{amsart}
\usepackage[latin1]{inputenc} 
\usepackage[T1]{fontenc}   
\usepackage[francais, english]{babel}  
\usepackage{amssymb}
\usepackage{amsmath}
\usepackage{tikz}
\usepackage{graphicx}
\usepackage{caption}
\usepackage{subcaption}

\newtheorem{cond}{Condition}[section]
\newtheorem{lemma}[cond]{Lemma}
\newtheorem{prop}[cond]{Proposition}
\newtheorem{thm}[cond]{Theorem}

\newtheorem{cor}[cond]{Corollary}
\theoremstyle{definition}
\newtheorem{defn}[cond]{Definition}
\theoremstyle{remark}
\newtheorem{rem}[cond]{\bf Remark}

\newtheorem{ex}[cond]{\bf Example}

\newtheorem{conj}[cond]{\bf Conjecture}
\newcounter{numl}
\newcommand{\labelnuml}{\textup{(\roman{numl})}}
\newenvironment{numlist}{\begin{list}{\labelnuml}%
{\usecounter{numl}\setlength{\leftmargin}{0pt}%
\setlength{\itemindent}{2\parindent}%
\setlength{\itemsep}{\smallskipamount}\def
\makelabel ##1{\hss \llap {\upshape ##1}}}}{\end{list}}
\newenvironment{bulletlist}{\begin{list}{\labelitemi}%
{\setlength{\leftmargin}{\parindent}\def
\makelabel ##1{\hss \llap {\upshape ##1}}}}{\end{list}}

\DeclareSymbolFont{script}{U}{eus}{m}{n}
\DeclareSymbolFontAlphabet{\mathscr}{script}
\DeclareMathSymbol{\Wedge}{0}{script}{"5E}
\DeclareMathAlphabet{\mathrmsl}{OT1}{cmr}{m}{sl}

\newcommand{\R}{{\mathbb R}}
\newcommand{\C}{{\mathbb C}}
\newcommand{\Z}{{\mathbb Z}}

\newcommand{\T}{{\mathbb T}}
\newcommand{\F}{{\mathbb F}}

\newcommand{\cC}{{\mathcal C}}
\newcommand{\cF}{{\mathcal F}}
\newcommand{\cL}{{\mathcal L}}

\newcommand{\cM}{{\mathcal M}}
\newcommand{\cO}{{\mathcal O}}

\newcommand{\cS}{{\mathcal S}}
\newcommand{\tor}{{\mathfrak t}}

\newcommand{\vzero}{u^0_{\alpha,\beta}}

\newcommand{\Proj}{\mathbb P}
\newcommand{\Hzero}{\mathbf{H}^0_{\alpha,\beta}}
\newcommand{\Gzero}{\mathbf{G}^0_{\alpha,\beta}}

\newcommand{\al}{\alpha}
\newcommand{\be}{\beta}

\newcommand{\1}{0}
\newcommand{\2}{\infty}
\newcommand{\Gv}{\mathbf{G}^u}
\newcommand{\Hv}{\mathbf{H}^u}
\newcommand{\bL}{\boldsymbol{L}}

\begin{document}

\title{Extremal K\"ahler Poincar\'e type metrics on toric varieties}

\author[V. Apostolov]{Vestislav Apostolov}
\address{Vestislav Apostolov \\ D{\'e}partement de Math{\'e}matiques\\
UQAM\\ C.P. 8888 \\ Succursale Centre-ville \\ Montr{\'e}al (Qu{\'e}bec) \\
H3C 3P8 \\ Canada}
\email{apostolov.vestislav@uqam.ca}

\author[H. Auvray]{Hugues Auvray}
\address{Hugues Auvray \\ Laboratoire de Math\'{e}matiques d'Orsay \\ Universit\'{e} Paris-Sud \\
CNRS, Universit\'{e} Paris-Saclay \\ D\'epartement de Math\'ematiques \\
B\^atiment 425 \\ 91405 Orsay \\ France.}
\email{hugues.auvray@math.u-psud.fr}

\author[L. M. Sektnan]{Lars Martin Sektnan}
\address{Lars Martin Sektnan\\ D{\'e}partement de Math{\'e}matiques\\
UQAM\\ C.P. 8888 \\ Succursale Centre-ville \\ Montr{\'e}al (Qu{\'e}bec) \\
H3C 3P8 \\ Canada}
\email{lars.sektnan@cirget.ca}


\thanks{V.A. was supported in part by an NSERC discovery grant.  He is grateful to l'\'Ecole Normale Sup\'erieure and l'\'Ecole polytechnique in Paris,  and 
the  Institute of Mathematics and Informatics of the Bulgarian Academy of Sciences in Sofia, for their hospitality during the preparation of parts of this work. L.S. is thankful to the CIRGET who supports his postdoctoral position. V.A. and H.A. would like to thank O. Biquard and P. Gauduchon for stimulating discussions about this project.}

\begin{abstract} We develop a general theory for the existence of extremal K\"ahler metrics  of Poincar\'e type in the sense of \cite{Auvray0},  defined on the complement of a toric divisor of a polarized toric variety.  In the case when the divisor is smooth, we obtain  a list of necessary conditions which must  be satisfied  for such a metric to exist. Using the explicit methods of \cite{ACG1,ACG2} together with the computational approach of \cite{lars}, we show that on a Hirzebruch complex surface the necessary conditions are also sufficient. In particular, on such a complex surface the complement of  the  infinity section admits  an extremal K\"ahler metric of Poincar\'e type whereas the complement of a fibre  admits a complete ambitoric extremal K\"ahler metric which is not of Poincar\'e type.
\end{abstract}

\maketitle

\section{Introduction}  
In this article, we are interested in the study of (non-compact) complete extremal K\"ahler metrics, 
defined on the complement of a simple normal crossing divisor  $Z$ in an $n$-dimensional K\"ahler manifold $X$. 
Such metrics naturally appear (see e.g. \cite{Do2})  in attempts to apply continuity methods, or to study global properties of geometric  flows, 
aiming at producing extremal K\"ahler metrics on $X$ in the framework of the general problem of finding canonical K\"ahler metrics formulated by Calabi~\cite{calabi}.  

The main conjecture regarding the Calabi problem  is the Yau--Tian--Donaldson conjecture 
which relates the existence of an extremal K\"ahler metric in the first Chern class $c_1(L)$ of an ample line bundle $L$ on $X$ 
to an algebro-geometric notion of stability of the polarized projective variety $(X, L)$. 
In this context, a key point is to understand what happens when  an extremal K\"ahler metric does not exist in $c_1(L)$. 
For toric varieties, Donaldson conjectured \cite[Conj. 7.2.3.]{Do2}  
that there should be a splitting of the corresponding  Delzant polytope into sub-polytopes 
which are semi-stable when attaching a $0$ measure to the facets that are not from the original polytope; 
furthermore, in the case when a semi-stable polytope in the splitting is stable, 
it is conjectured  to admit  a symplectic potential inducing a  (unique)  complete extremal K\"ahler metric 
on the complement of the divisors corresponding to the facets with $0$ measure. 
Such extremal toric K\"ahler metrics have  a finite volume, and  we shall refer to them  as {\it Donaldson metrics}.

The main motivation for this paper is to study, in the toric case,  
the  precise link between the extremal Donaldson metrics and the class  of  complete K\"ahler metrics of finite volume on $X\setminus Z$, 
called of {\it Poincar\'e type}, 
early used for instance in \cite{CoGr, TY}, and studied by the second named author in \cite{Auvray0}.

\begin{defn}\label{d:Poincare-type}  
Let $Z \subset X$ be a simple normal crossing divisor in a compact complex $n$-dimensional K\"ahler manifold $(X, \omega_0)$. 
A  K\"ahler metric $\omega$ on $X\setminus Z$ is said to be  of {\it Poincar\'e type of class $[\omega_0]$} if 
\begin{enumerate}
\item[$\bullet$] On any open subset  $U\subset X$ with holomorphic coordinates $(z_1, \ldots, z_n)$ such that $Z\cap U$ is given by $z_1\cdots z_k =0$, 
$\omega$ is quasi-isometric to the $(1,1)$-form
$$\omega_{\rm mod} =\sqrt{-1} \Big(\sum_{j=1}^k \frac{1}{|z_j|^2(\log|z_j|)^2} dz_j \wedge d\bar z_j +  \sum_{j=k+1}^n dz_j \wedge d\bar z_j \Big)$$
near $Z$, and
\item[$\bullet$] $\omega = \omega_0 + dd^c \varphi$ where $\varphi$ is a smooth function on $X\setminus Z$ 
and $\varphi = O\big(\sum_{j=1}^k \log(-\log |z_j|)\big)$ in the coordinates $(z_1, \ldots, z_n)$ as above, 
with  $d\varphi$ having  bounded derivatives of any order with respect to the model metric $\omega_{\rm mod}$ above.
\end{enumerate}
\end{defn}

General theory for  {\it extremal} Poincar\'e type metrics  on $(X\setminus Z, [\omega_0])$  has been developed in \cite{Auvray0, Auvray, Auvray1, Auvray2}. 
In particular,  a differential-geometric obstruction for the existence of a {\it constant scalar curvature K\"ahler} (CSCK) metric of Poincar\'e type on $X\subset Z$, 
reminiscent to the usual Futaki invariant, is introduced in \cite{Auvray2}.  
Furthermore,   in the special case when the K\"ahler class $[\omega_0]= c_1(L)$ is associated to a polarization $L$ on $X$, 
an algebro-geometric notion of  (relative) {\it K-stability} of $(X, Z, L)$ is  formulated by G. Sz\'ekelyhidi in \cite{Sz},  
who introduced a suitable version of the Donaldson--Futaki invariant of a test configuration associated to the triple $(X, Z, L)$. 
Sz\'ekelyhidi also defined a numerical constraint, which we shall refer to in this paper as {\it  Sz\'ekelyhidi's numerical constraint} (see Definition~\ref{d:constraint}), 
which is related to the deformation to the normal cone of $Z\subset X$, 
and is designed to guarantee the existence of a Poincar\'e type metric (and not a complete extremal K\"ahler metric with different asymptotics near $Z$).  
It was later shown in \cite{Auvray} that  Sz\'ekelyhidi's numerical constraint is  a necessary condition 
for the existence of a  CSCK Poincar\'e type  metric  on $X\setminus Z$ in the class $c_1(L)$.  
The case when $Z$ is smooth and admits a  K\"ahler metric of {\it non-positive}  constant scalar curvature has been also studied in \cite{S-Sun, Sun-Sun, J-Sun}, 
where it is conjectured that $(X, Z, L)$ is then $K$-semistable and admits a complete K\"ahler  metric of {\it negative} constant scalar curvature.

Thus motivated, in Section~\ref{s:toric} we turn  to the case when $(X, L)$ is a  smooth toric variety, and $Z$ a divisor invariant under the torus action. 
Compared to the theory in \cite{Sun-Sun}, we are dealing with the case when each component of $Z$ is a toric variety, 
and therefore can only admit a {\it positive} CSCK metric. 
In terms of the corresponding momentum polytope $\Delta$, $Z$  is  the  preimage by the moment map of  the union $F= \cup_{i} F_i$ of  facets $F_i$ of $\Delta$. 
In this setting (and taking $F$ to be a single facet), we show that Sz\'ekelyhidi's numerical constraint  takes a particularly simple form (Lemma~\ref{l:toric-numerical}), and matches the necessary numerical condition for the existence of an extremal K\"ahler metric of Poincar\'e type  on $X\setminus Z$ found in  \cite{Auvray2}.  
Furthermore, building on the recent results in \cite{Auvray1} and a conjecture from \cite{Do2} (see Conjecture~\ref{conjecture2} below),  
we formulate a precise conjectural picture concerning the existence of an extremal Poincar\'e type toric K\"ahler metric on $(X\setminus Z, [\omega ])$,  see Conjecture~\ref{conjecture1}.  
It states that such an extremal K\"ahler metric exists if 
(i) the pair $(\Delta, F)$ is stable, (ii) for each facet $F_i \subset F$, the pair $(F_i, F_i\cap (\cup_{j\neq i} F_j))$ is stable, 
and (iii) when restricted to $F_i$, the extremal affine function of $(\Delta, F)$ differs  from the extremal affine function of $(F_i, F_i\cap (\cup_{j\neq i} F_j))$ 
by a negative constant.  
This is a much stronger statement than the original conjecture made in \cite{Sz} (see Conjecture~\ref{c:numerical}), 
but in the remainder of the paper we show that  it is sharper too.

In Section~\ref{s:abreu-guillemin}, we develop the Abreu--Guillemin formalism of toric K\"ahler metrics of Poincar\'e type, 
thus leading to a natural class of symplectic potentials (see Definition~\ref{toric-poincare} and Theorem~\ref{thm:criterion}) 
which give rise to Poincar\'e type metrics in the sense of Definition~\ref{d:Poincare-type}. 
While these conditions are sufficient, they are not  necessary in general 
(but are conjecturally sharp \textit{when adding the extremality condition}).  
We show that within this class  of Poincar\'e type metrics on $X\setminus Z$, 
the extremal ones are unique and the conditions (i), (ii) and (iii) of our Conjecture~\ref{conjecture1}  are necessary too.

In the last Section~\ref{s:examples}, we turn to explicit examples by using the methods of \cite{ACG1,ACG2}. 
These results together with \cite{Dixon} confirm the Donaldson conjecture concerning the existence of a  complete extremal K\"ahler metric 
in the case when  $(\Delta,F)$ is a stable quadrilateral with some of its facets with measure $0$, also allowing us to find the metric explicitly.  
Investigating the stability of such pairs is,  on its own,  a  problem of formidable complexity but using the method from \cite{lars},  
we obtain a  complete picture on the Hirzebruch complex surfaces.

\begin{figure}[h!]
\captionsetup[subfigure]{labelformat=parens}
\centering
\begin{subfigure}[b]{0.4\textwidth}
\includegraphics[scale=0.13,angle=0]{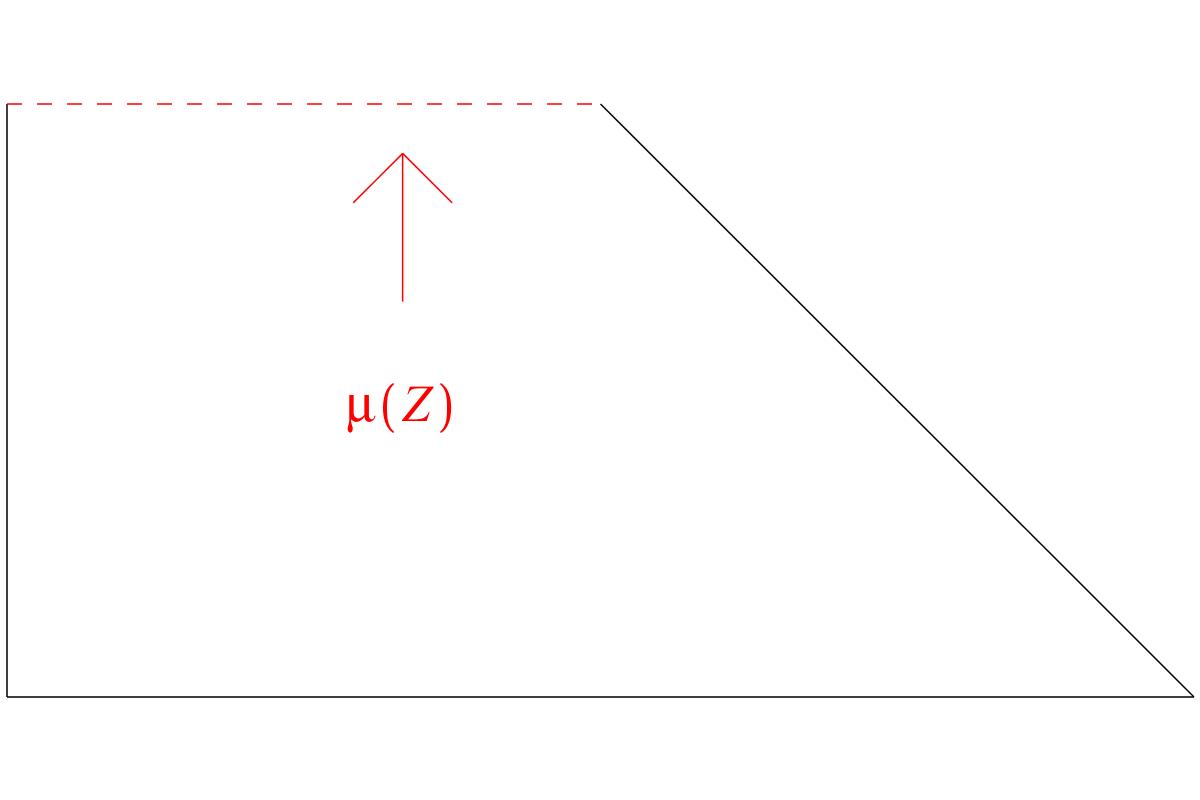}
\caption{$Z = S_0$}
\end{subfigure}
~
\begin{subfigure}[b]{0.4\textwidth}
\includegraphics[scale=0.13,angle=0]{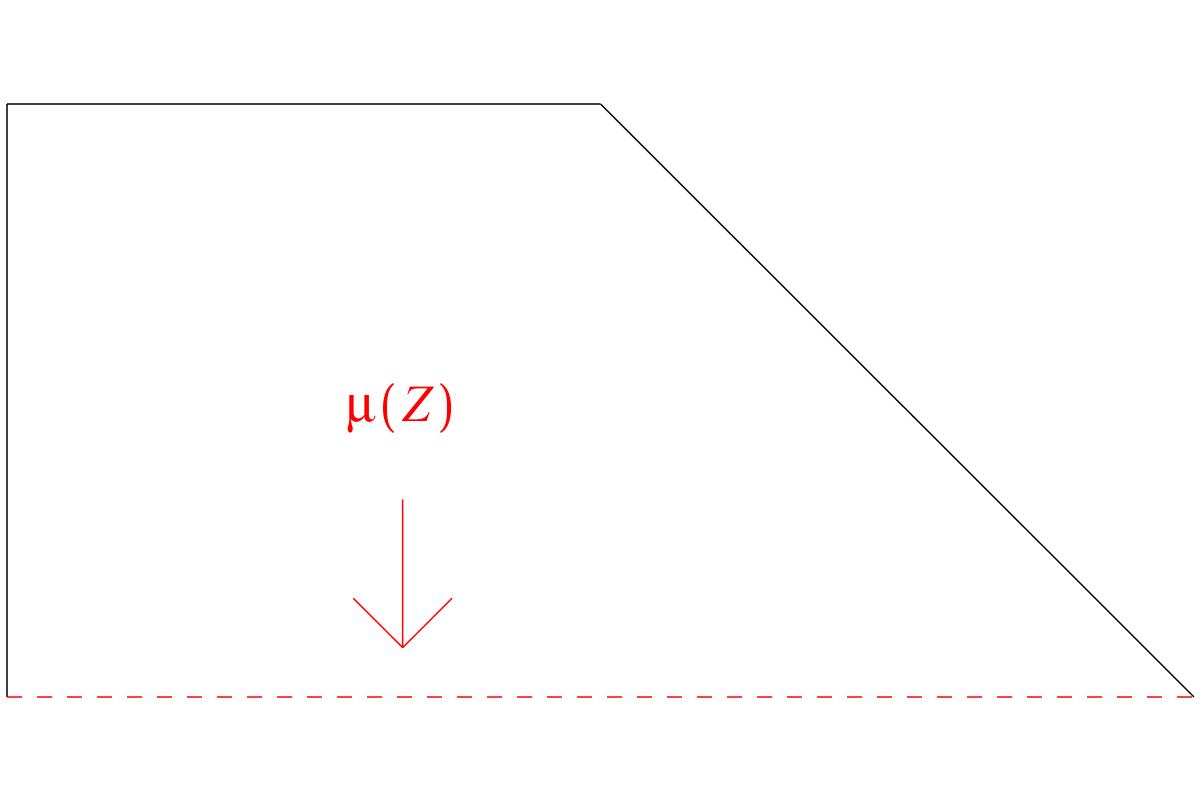}
\caption{$Z = S_{\infty}$}
\end{subfigure}

\begin{subfigure}[b]{0.4\textwidth}
\includegraphics[scale=0.13,angle=0]{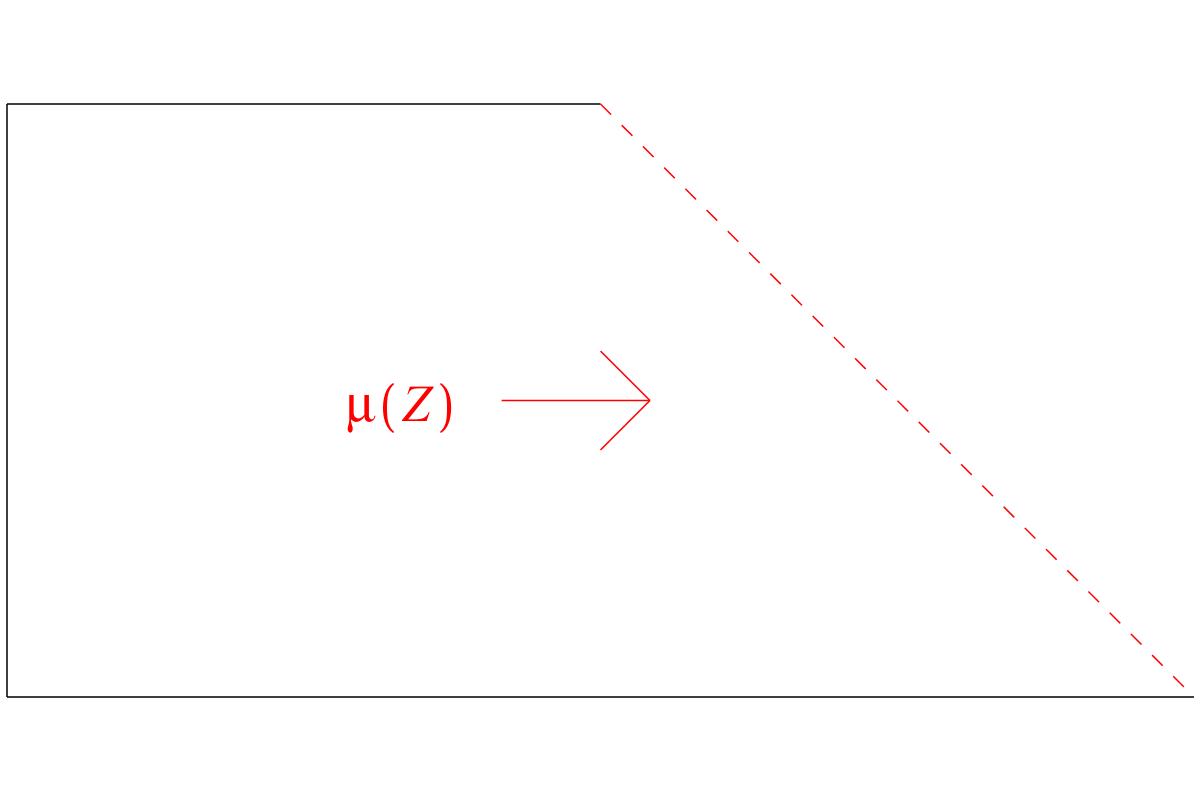}
\caption{$Z = F_1$}
\end{subfigure}
~
\begin{subfigure}[b]{0.4\textwidth}
\includegraphics[scale=0.13,angle=0]{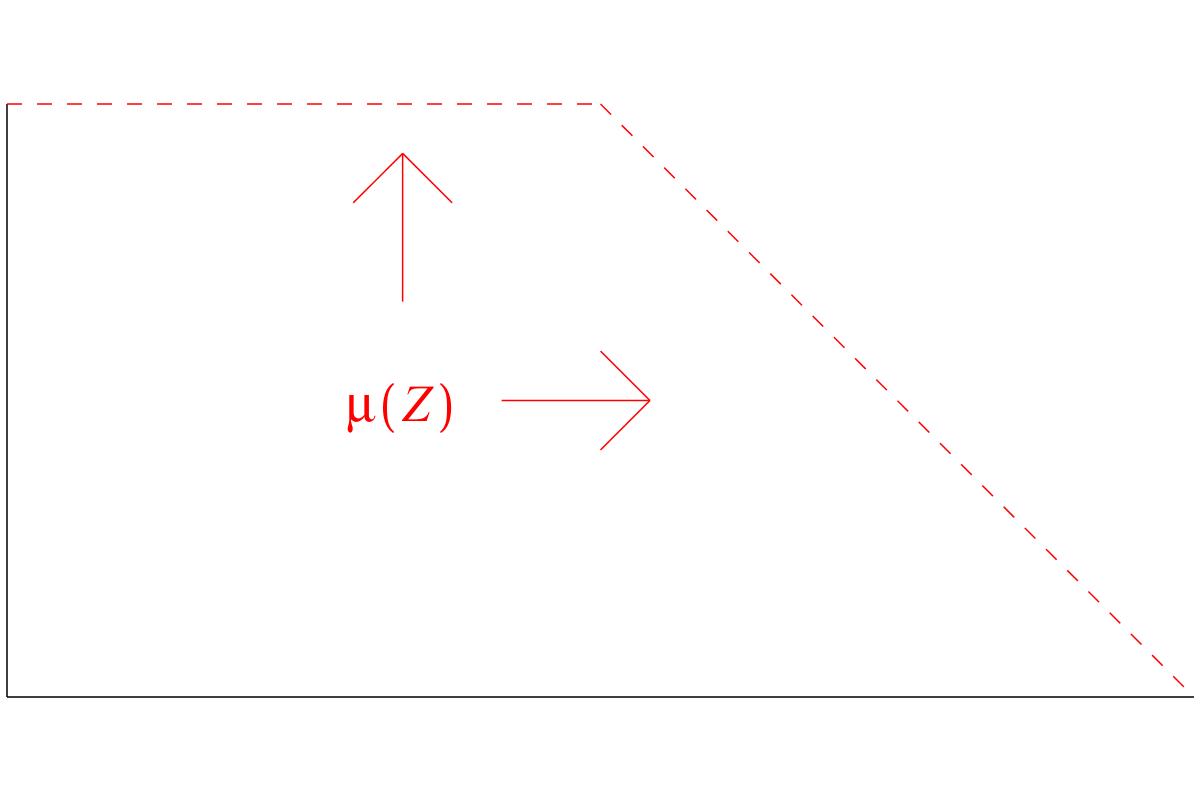}
\caption{$Z = S_0 \cup F_1$}
\end{subfigure}

\begin{subfigure}[b]{0.4\textwidth}
\includegraphics[scale=0.13,angle=0]{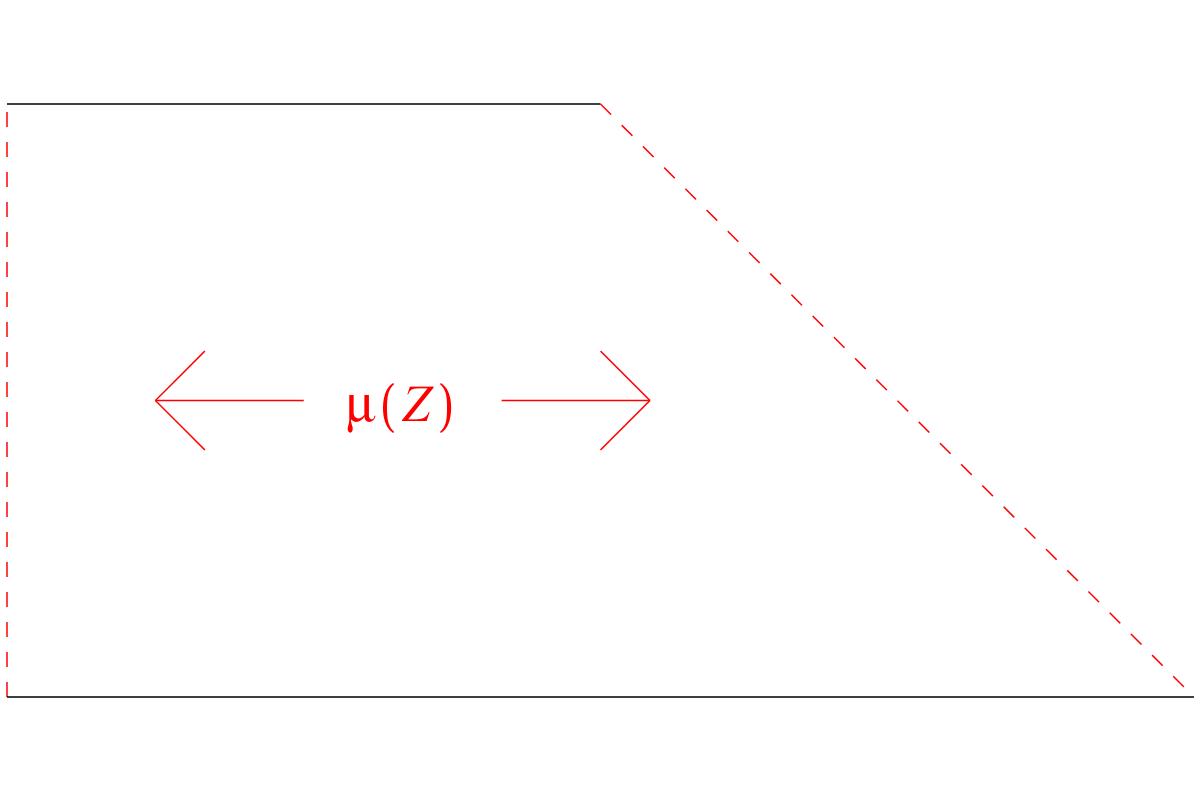}
\caption{$Z = F_1 \cup F_2$}
\end{subfigure}
~
\begin{subfigure}[b]{0.4\textwidth}
\includegraphics[scale=0.13,angle=0]{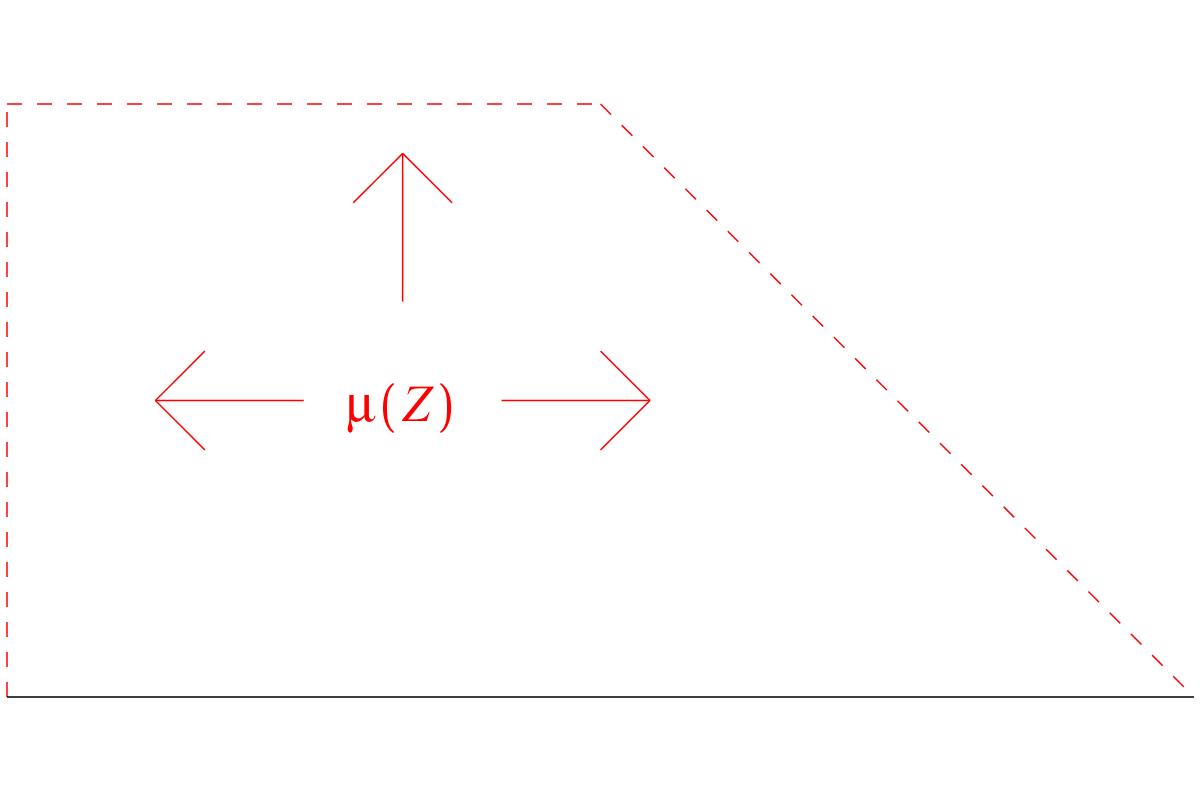}
\caption{$Z = F_1 \cup F_2 \cup S_0 $}
\end{subfigure}
\caption{The rows illustrate case (a), (b) and (c) of Theorem~\ref{main-example}.}
\end{figure}

\begin{thm}\label{main-example} 
Let $X= {\mathbb P}(\cO \oplus \cO(m)) \to \C P^1, m\ge 1$, be  the $m$-th Hirzebruch  surface, 
considered as a  toric complex surface under the action of a $2$-dimensinal torus $\T$,  and $[\omega_0]$ be a K\"ahler class on $X$. Then, 

\begin{enumerate}
\item[{\rm (a)}] If $Z \subset X$ is the divisor consisting of either the zero section $S_0$, 
                 or the  infinity section $S_{\infty}$, 
                 or the union of both, 
                 then $X\setminus Z$ admits a $\T$-invariant  extremal Poincar\'e type K\"ahler metric in $[\omega_0],$ which  is a Donaldson metric of $(X, Z, [\omega_0])$.
\item[{\rm (b)}] If $Z\subset X$ is the divisor consisting of either a single fibre $F_1$ of $X$ fixed by the $\T$-action, 
                 or the union of $F_1$ with $S_0$, 
                 or the union of $F_1$ with $S_{\infty}$,  
                 then 
                 $X\setminus Z$ admits a complete $\T$-invariant Donaldson extremal K\"ahler metric in $[\omega_0],$ which is not of Poincar\'e type.
\item[{\rm (c)}] If $Z$ consists of the union of two fibres fixed by the torus action, or contains three curves fixed by the action, 
                 then $(X, Z, [\omega_0])$ is  unstable, i.e.  there are no Donaldson metrics on $X\setminus Z$.
\end{enumerate}
In particular, Conjecture~\ref{conjecture1} holds true.
\end{thm}

Similar existence results hold for the toric surfaces $X= \C P^2$ and $\C P^1 \times \C P^1$, see Corollary~\ref{c:product} and Theorem~\ref{thm:triangle} below. 

We end this introduction by noticing  that part (b) of Theorem~\ref{main-example} above implies that while for the $X,Z$ and $[\omega_0]$ considered here, 
the relative stability of $(X,Z,[\omega_0])$ does imply the existence of a complete  extremal  K\"ahler metric on $X \setminus Z$, 
this metric is not in general of Poincar\'e type, even though the Sz\'ekelyhidi numerical constraint  is satisfied. 
Thus,  in general, one will need more conditions to guarantee that the extremal metric obtained for a relatively stable triple $(X, Z, [\omega_0])$ is of Poincar\'e type. 
Conjecture~\ref{conjecture1} is designed to incorporate this extra requirement in the toric setting.

\section{The relative K-stability of a pair}\label{s:Donaldson-Futaki}  

\subsection{Donaldson--Futaki invariant of a pair} We follow \cite[\S 3.1.2]{Sz}. 
Let $(X,L)$ be a smooth polarized variety of complex dimension $n$ and $Z\subset X$ a smooth divisor.  
Suppose $\alpha$ is a $\C^{\times}$-action on $(X,L)$ which preserves $Z$. We consider the embedding
$$H^0(X, L^k\otimes {\mathcal O}(-Z)) \subset H^0(X, L^k)$$
via a section of ${\mathcal O}(Z)$ which vanishes along $Z$. Since $L$ is ample,
\begin{equation*}
H^1(X, L^k\otimes{\mathcal O}(-Z))=0  \ \ {\rm for} \ k\gg 0,
\end{equation*}
and we have an exact sequence
\begin{equation}\label{exact}
0 \longrightarrow H^0(X, L^k\otimes {\mathcal O}(-Z)) 
  \longrightarrow H^0(X, L^k) 
  \longrightarrow  H^0(Z, L_{|_Z}^k) 
  \longrightarrow 0.
\end{equation}
Let $d_k, d'_k, d^Z_k$ be the dimensions of $H^0(X,L^k), H^0(X, L^k\otimes \cO(-Z)), H^0(Z, L_{|_Z}^k))$, respectively, 
and let ${\tilde d}_k$ be the average of $d_k$ and $d'_k$. By Riemann--Roch and \eqref{exact}, 
\begin{equation}\label{dimension}
\begin{split}
d_k&= c_0k^n + c_1k^{n-1} + O(k^{n-2}); \ \ d^Z_k = \alpha_0k^{n-1} + \alpha_1 k^{n-2} + O(k^{n-3});\\
{\tilde d}_k& = \frac{d_k + d'_k}{2}=  d_k - \frac{d^Z_k}{2}= c_0 k^{n} + \Big(c_1- \frac{\alpha_0}{2}\Big)k^{n-1} + O(k^{n-2}).
\end{split}
 \end{equation}
Similarly, letting $w_k, w'_k, w^Z_k$ be the respective weights of the action of $\alpha$ on $H^0(X,L^k), H^0(X, L^k\otimes \cO(-Z)), H^0(Z, L_{|_Z}^k))$, 
and ${\tilde w}_k$ be the average of $w_k$ and $w'_k$, by the equivariant Riemann--Roch and \eqref{exact} we have
\begin{equation}\label{weight}
\begin{split}
w_k (\alpha) &= a_0k^{n+1} + a_1k^{n} + O(k^{n-1}); \ \ w^Z_k (\alpha) = \beta_0k^{n} + \beta_1 k^{n-1} + O(k^{n-2});\\
{\tilde w}_k(\alpha) & = \frac{w_k (\alpha)+ w'_k(\alpha)}{2}=  w_k(\alpha) - \frac{w^Z_k(\alpha)}{2}= a_0 k^{n+1} + \Big(a_1- \frac{\beta_0}{2}\Big)k^{n} + O(k^{n-1}).
\end{split}
 \end{equation}

\begin{defn}\label{d:DF} 
The {\it Donaldson--Futaki invariant} ${\widetilde {\mathcal F}}_{X,Z,L}(\alpha)$ of $\alpha$ with respect to $(X,Z,L)$  
is defined as $-4$ times the residue at $k=0$ of the Laurent series of ${\tilde w}_k/(k {\tilde d}_k)$ with respect to $k$, i.e.
\begin{equation}\label{modified-futaki}
\begin{split}
\frac{1}{4} {\widetilde {\mathcal F}}_{X,Z,L}(\alpha) &= \frac{c_0(a_1- \frac{\beta_0}{2}) -a_0(c_1- \frac{\alpha_0}{2})}{c_0^2}\\
&= \frac{1}{4}{\mathcal F}_{X,L}(\alpha) +  \frac{1}{2}\Big(\frac{a_0\alpha_0 -c_0\beta_0}{c_0^2}\Big),
\end{split}
\end{equation}
where ${\mathcal F}_{X,L}(\alpha) = 4(\frac{c_0a_1- a_0c_1}{c_0^2})$ is the convention for the Donaldson--Futaki invariant in \cite{Do2}, 
so that it coincides with the differential-geometric formula in \cite{Auvray2} for the usual normalized Futaki invariant of $\alpha$, 
expressed in terms of the $\bL^2$
-product  of a normalized Killing potential for the $\C^{\times}$-action 
and the scalar curvature with respect to an $S^1$-invariant K\"ahler metric on $X$ in  $ 2\pi c_1(L)$, divided by the volume. 
\end{defn}

Following \cite{Sz}, one can also define a relative version of ${\widetilde {\mathcal F}}_{X,Z,L}(\alpha)$ 
with respect to another $\C^{\times}$-action $\beta$ in the group ${\rm Aut}(X, L, Z)$ of automorphisms of  $(X,L)$,  preserving $Z$. 
Recall that the inner product $\langle \alpha, \beta \rangle$ is defined 
to be the coefficient of $k^{n+2}$ of the expansion of ${\rm Tr} (A_k B_k) - w_k(\alpha)w_k(\beta)/d_k$,
where $A_k$ and $B_k$ are generators of the actions of $\alpha$ and $\beta$ on $H^0(X,L^k)$. 
This definition is consistent with the $\bL^2$
-norm of normalized Killing potentials (the so-called Futaki--Mabuchi bilinear form~\cite{FM}). 

\begin{defn} The $\beta$-relative Donaldson--Futaki invariant (of $\alpha$, with respect to $(X,Z,L)$) is 
\begin{equation}\label{relative-modified-futaki}
{\widetilde {\mathcal F}}_{X,Z,L}^{\beta}(\alpha) 
          = {\widetilde {\mathcal F}}_{X,Z,L}(\alpha)
             - \frac{\langle \alpha, \beta\rangle}{\langle \beta, \beta \rangle} {\widetilde {\mathcal F}}_{X,Z,L}(\beta).
\end{equation}
\end{defn}
The above definitions make sense for any rational multiples of $\alpha$ and $\beta$ (by linearity). 
We then consider a maximal complex torus $\T^c =(\C^{\times})^{\ell}$ in ${\rm Aut}(X, L, Z)$ 
and define the  {\it extremal} $\C^{\times}$-action $\chi$ of $(X,L,Z)$ as the unique $\C^{\times}$ subgroup of $\T^c$ 
such that ${\widetilde {\mathcal F}}_{X,Z,L}^{\chi}(\alpha)=0$.

\subsection{Test configurations and $K$-stability of a pair} The ingredients of the previous section yield Sz\'ekelyhidi's extension \cite{Sz} of $K$-stability to pairs.

\begin{defn}\label{d:relative-DF} 
 The triple $(X, Z, L)$ is called {\it $K$-stable} if 
 for any test-configuration $({\mathcal X}, {\cL})$ of $(X,L)$ 
 with a flat $\C^{\times}$-invariant Cartier divisor ${\mathcal Z} \subset {\mathcal X}$ which restricts to $Z$ on the non-zero fibres, 
 the modified Donaldson--Futaki invariant  of the central fibre satisfies
\begin{equation}\label{stability}
{\widetilde {\mathcal F}}_{X_0,Z_0,L_0}(\alpha)  \ge 0
\end{equation}
with equality if and only if the test configuration is {\it trivial in codimension 2} 
(see \cite{stoppa} for a precise definition of triviality). Similarly, one can define {\it relative K-stability} of $(X,L,Z)$  by requiring 
\begin{equation}\label{relative-stability}
{\widetilde {\mathcal F}}^{\chi}_{X_0,Z_0,L_0}(\alpha)  \ge 0,
\end{equation}
with equality if and only if the test configuration is { trivial in codimension 2}. 
(Recall $\chi$ is the extremal $\C^{\times}$-action defined algebraically in the previous section.)
\end{defn}

Investigating a ruled complex surface $X={\mathbb P}(\cO \oplus \cL) \to \Sigma$ with $Z$ being the infinity section, 
Sz\'ekelyhidi~\cite{Sz} noticed that  for some polarizations $L$, 
there are  complete finite volume extremal K\"ahler metrics on $X\setminus Z$  in $c_1(L),$  which are not of Poincar\'e type, but have instead the asymptotics of 
\begin{align*} \frac{|dz|^2}{|z|^2 \big(- \log(|z|) \big)^{\frac{3}{2}}} + \ {\rm smooth},
\end{align*}
where $z$ is a (local) defining holomorphic function of $Z$.
In order to exclude this behaviour,  
Sz\'ekelyhidi furthermore proposes to use the notion of {\it slope stability} introduced by Ross--Thomas~\cite{RT} for the triple $(X, L, Z)$ as follows. Recall
that for any $(X,L,Z)$ as above,  and any rational number $c\in (0, \epsilon(Z))$ (where $\epsilon(Z)$ is the Seshadri constant of $Z$ with respect to $(X,L)$), 
one can associate a  test configuration $({\mathcal X}, {\mathcal L}_c, {\mathcal Z})$,  called the {\it degeneration to the normal cone of $Z$}: 
${\mathcal X}$ is the blow-up of $X\times \C$ along $Z\times \{0\}$,  $\cL_c=\pi^*(L)\otimes \cO(-cP)$  where $P$ is the exceptional divisor 
(naturally identified with the projective bundle $P=\mathbb{P}(\cO \oplus \nu_Z) \to Z$ where $\nu_Z$ is the normal bundle of $Z\subset X$), 
and  $\pi : {\mathcal X} \to \C$ is the projection. 
Note that  the central fibre $X_0$ of $\pi$ is isomorphic to $X$ glued to $P$ along the infinity section $\mathbb{P}(\nu_Z)\cong Z$. 
However, considering the zero section $Z_0 \subset P$, one gets the proper transform ${\mathcal Z}$ of $Z\times \C \subset X\times \C$ on the blow-up  ${\mathcal X}$, 
so that $\pi : {\mathcal Z} \to \C$ is a trivial family. 
It follows from \cite{RT} that  $({\mathcal X}, {\mathcal L}_c, {\mathcal Z})$ defines  a test configuration for the triple $(X,L,Z)$.  
This motivates: 
\begin{defn} In the notation above, we let 
$$F(c):={\widetilde {\mathcal F}}_{X_0,{\cL_c}_{|_{{\mathcal X}_0},Z_0}}(\alpha_c), 
\ \ \ F_{\chi}(c):={\widetilde {\mathcal F}}^{\chi}_{X_0, Z_0, {\cL_c}_{|_{{\mathcal X}_0}}}(\alpha_c)$$ 
be the corresponding modified Donaldson--Futaki invariant and relative modified Donaldson--Futaki invariant  associated to 
the degeneration  of the normal cone to $Z\subset (X, L)$.
\end{defn}
Then, Sz\'ekelyhidi conjectures:
\begin{conj}[\bf Sz\'ekelyhidi~\cite{Sz}]\label{c:numerical} 
 The triple $(X, Z,L)$ admits a {\it constant scalar curvature} 
 (resp. an {\it extremal}) K\"ahler metric of Poincar\'e type 
 if and only if $(X,Z,L)$ is $K$-stable (resp. relative $K$-stable)  
 and, additionally,  $F''(0)>0$ (resp. $F_{\chi}''(0)>0)$.
\end{conj}
\begin{defn}\label{d:constraint} We shall refer to the conditions $F''(0)>0$ (resp. $F_{\chi}''(0)>0$) as the  {\it Sz\'ekelyhidi numerical constraint} 
(resp. {\it relative Sz\'ekelyhidi numerical constraint}).
\end{defn}
The following observation is made in \cite{Sz}:
\begin{lemma}\label{characterization} $F(c):={\widetilde {\mathcal F}}_{X_0,{\cL_c}_{|_{{\mathcal X}_0},Z_0}}(\alpha_c)$ is a polynomial of degree $\le (n+1)$ in $c$ 
satisfying $F(0)=F'(0)=0$. It is positive for $c\in (0, \epsilon(Z))$ if $(X,L,Z)$ is $K$-stable. 
Furthermore, the Sz\'ekelyhidi numerical constraint $F''(0)>0$ is equivalent to 
\begin{equation}\label{contraint}
\alpha_1c_0 > \alpha_0(c_1 - \frac{\alpha_0}{2}),
\end{equation}
where $\alpha_i$, $c_i$  are defined by \eqref{dimension}.
\end{lemma}
\begin{proof} The (usual) Donaldson--Futaki invariant of $({\mathcal X}, {\mathcal L}_c)$ is computed in \cite{RT}: 
\begin{equation}\label{slope-futaki}
\frac{1}{4}{\mathcal F}(\alpha_c) = \frac{1}{c_0^2}\Big[-c_0\int_0^c \alpha_1(x)(x-c)dx +c \frac{c_0\alpha_0}{2} + c_1\int_{0}^c \alpha_0(x)(x-c)dx\Big],
\end{equation}
where $c_0,c_1$ are the coefficients of $k^n$ and $k^{n-1}$ of  $d_k$ as defined in the previous section (with respect to $(X,L,Z)$), 
see \eqref{dimension}, and
\begin{equation} \begin{split}
 \alpha_0(x) &= \frac{1}{(n-1)!} \int_Z (c_1(L) + x c_1({\mathcal O}(Z))^{n-1}; \\
\alpha_1(x) &= \frac{1}{2(n-2)!}\int_Z c_1(TX)\wedge (c_1(L)+ x c_1({\mathcal O}(Z))^{n-2}.
\end{split}
\end{equation}
By Riemann--Roch, $\alpha_i(0)$ is the constant $\alpha_i$ appearing in the previous section (first line of \eqref{dimension}). 

The main ingredient in order to carry out the above calculation in the modified case 
is the weight space decomposition for the induced $\C^{\times}$-action $\alpha_c$ on the space $H^0(X_0, {\cL_c}_{|_{{\mathcal X}_0}})$
(see \cite[\S4.2]{RT}):
\begin{equation}\label{splitting}
H^0(X_0, {\cL_c}_{|_{{\mathcal X}_0}})= H^0(X, L^{k}\otimes \cO(kc Z)) \oplus \bigoplus_{i=0}^{ck-1} t^{ck-i}H^0(Z, L^k_{|_Z}\otimes (\nu^*_Z)^i),
\end{equation}
where the weight of $\alpha_c$ on the first factor is $0$ and $-(ck-i)$ on the components of the second direct sum. 
Note that the factor $t^{ck}H^0(Z, L^k_{|_Z})$ in the above decomposition corresponds to $H^0(Z_0, \cL_{|_{Z_0}})$ in \eqref{exact}. It follows that
\begin{equation}
d_k^{Z_0} = d_k^{Z} = \alpha_0 k^{n-1}  + O(k^{n-2}); \ \  w_k^{Z_0} = ckd_k^Z= c\alpha_0 k^n + O(k^{n-1})
\end{equation}
while the coefficients $a_0$ and $a_1$ of the weights induced on $H^0(X_0, {\cL_c}_{|_{{\mathcal X}_0}})$ are given by (see \cite[Eqn.(4.6)]{RT}): 
\begin{equation}
a_0= \int_{0}^{c}(x-c)\alpha_0(x)dx; \ \ a_1 = -c\frac{\alpha_0}{2c_0} + \int_0^c(x-c)dx. 
\end{equation}
We therefore compute the modified Donaldson--Futaki invariant given by \eqref{modified-futaki}: 
\begin{equation}\label{modified-futaki-slope}
\frac{1}{4}{\widetilde {\mathcal F}}_{X_0, Z_0, {\cL_c}_{|_{{\mathcal X}_0}}}(\alpha_c) 
      = \frac{1}{c_0^2}\Big[(c_1-\frac{\alpha_0}{2})\int_{0}^c(x-c)\alpha_0(x)dx - c_0\int_0^c(x-c)\alpha_1(x)dx\Big].
\end{equation}
Form the above formula, the proof of Lemma~\ref{characterization} then follows easily. \end{proof}
Using Lemma~\ref{characterization}, it is shown in \cite{Auvray}: 
\begin{thm}\cite{Auvray} 
 If there exists a CSCK metric of Poincar\'e type on $X\setminus Z$ in the class $c_1(L),$ then the  Sz\'ekelyhidi numerical constraint holds, 
 i.e. \eqref{contraint} is satisfied.
\end{thm}
\begin{rem}  It is plausible to expect a similar numerical expression  for the relative Sz\'ekelyhidi numerical constraint $F_{\chi}''(0)>0$ 
but we failed to see a neat way to compute $F_{\chi}(c)$ in a sufficient generality, especially if the extremal $\C^{\times}$-action is not trivial on $Z$. 
\end{rem}

We shall next turn to the toric case as a model example for the above theory, and where specific computations are manageable.  
We shall show  (see in particular Corollary \ref{Hirzebruch-non-poincare})  that  
there are examples of relatively $K$-stable triples $(X,Z,L)$ satisfying $F_{\chi}''(0) > 0,$ which cannot be of Poincar\'e type. 
We note, however, that these examples do admit a {complete} extremal metric on $X \setminus Z$ -- it just cannot satisfy the Poincar\'e type condition. 
We shall thus propose a straightened version of Conjecture~\ref{c:numerical} 
for when a relatively K-stable triple $(X,Z,L)$ should admit an extremal K\"ahler metric of Poincar\'e type in $c_1(L)$ in the toric setting (see Conjecture \ref{conjecture1}).

\section{Extremal Poincar\'e type K\"ahler metrics on toric varieties}\label{s:toric} 
In this section we consider the case when $(X,L)$ is a (smooth) polarized toric variety. 
We denote by $\T$ the real $n$-dimensional torus and by $\T^c \cong (\C^{\times})^n$ its complexification.  
The material follows \cite{Do2}. 

\subsection{Stability of pairs  and toric test configurations}
Switching from complex to symplectic point of view, Delzant's theorem \cite{Delzant} describes $(X,L)$ in terms of a compact convex polytope $\Delta \subset \mathfrak{t}^*$ 
(where $\mathfrak{t}= {\rm Lie}(\T)$ is the Lie algebra of  $\T$) 
such that $\Delta = \{\mu : L_j(\mu)= \langle e_j, \mu\rangle + \lambda_j \ge 0, j=1, \ldots, d\}$ 
with $e_j$ belonging to the lattice $\Lambda \subset \mathfrak{t}$ of circle subgroups of $\T$. 
The fact that $X$ is smooth corresponds to requiring that at each vertex of $v\in \Delta$ the adjacent normals span the same lattice $\Lambda\subset \tor$ 
(see \cite{Delzant,LT}), while the polarization $L$ forces $\Delta$ to have its vertices in the dual lattice $\Lambda^* \subset \mathfrak{t}^*$. 
Taking any generators of $\Lambda$ as a basis of $\mathfrak{t}$, one identifies $\Lambda$ with $\Z^n$ and we consider the Lebesgue measure $d\mu$ on $\mathfrak{t}^*\cong \R^n$; 
furthermore, one defines a measure $d\nu$ on $\partial \Delta,$ such that on each facet $F_j \subset \Delta$ (i.e. a face of co-dimension one),  we let
\begin{equation}\label{measure}
-dL_j \wedge d\nu_{F_j} =-e_j \wedge d\nu_{F_j} =d\mu. 
\end{equation}
A central fact in this theory (see e.g. \cite[Sect. 6.6]{Cannas da Silva})  is the weight decomposition of $H^0(X, L^k)$ with respect to the (linearized) torus action of $\T$ 
it is isomorphic to
$ \{\mu \in k\Delta \cap\Z^n \}$ with  the weights identified with corresponding elements of $\Z^n$.  
On the other hand, for any smooth function $f$ on $\mathfrak{t}^*$, we have \cite{GS, Zeld}:
\begin{equation*}
 \sum_{\mu \in k\Delta \cap \Z^n} f(\mu) = k^n \int_{\Delta} f d\mu+ \frac{k^{n-1}}{2}\int_{\partial \Delta} f d\nu + O(k^{n-2}),
 \ \text{as }k\to \infty.
\end{equation*}
If $\alpha$ is the $\C^{\times}$-action with Killing potential corresponding to an affine linear  function $f_{\alpha}$ on ${\mathfrak t}^*$ normalized by $f_{\al}(0)=0$, 
the above formula allows us to compute the coefficients $c_0,c_1, a_0, a_1$ in \eqref{dimension} and \eqref{weight} as follows:
\begin{equation}\label{a-c}
c_0= {\rm Vol}(\Delta); \ \ c_1= \frac{{\rm Vol}(\partial \Delta)}{2}; \ \ a_0 = \int_{\Delta} f_{\alpha} d\mu; \ \ a_1= \frac{1}{2}\int_{\partial \Delta} f_{\alpha} d\nu,
\end{equation}
so that the Donaldson--Futaki invariant $\cF(\alpha)$ of $\alpha$ is
\begin{equation}\label{toric-futaki}
\frac{(2\pi)^n{\rm Vol}(\Delta)}{2}\cF(\alpha) =  \int_{\partial \Delta} f_{\alpha} d\nu - \frac{\bf s}{2}\int_{\Delta} f_{\alpha} d\mu,
\end{equation}
where ${\bf s}= 2{\rm Vol}(\partial\Delta)/{\rm Vol}(\Delta)=2n \Big(\int_X c_1(TX) \wedge c_1(L)^{n-1}/\int_X c_1(L)^n\Big)$ 
is the averaged scalar curvature of any compatible K\"ahler metric.

Similarly, if $Z\subset X$ is a divisor corresponding to the preimage of the union $F=F_{i_1}\cup \cdots \cup F_{i_k}$ of some facets $\Delta$  by the momentum map, 
the coefficients $\alpha_0, \alpha_1, \beta_0, \beta_1$ in \eqref{dimension} and \eqref{weight} are given by
\begin{equation}\label{alpha-beta}
\alpha_0 = {\rm Vol}(F); \ \ \alpha_1= \frac{1}{2} {\rm Vol} (\partial F); \ \ \beta_0 = \int_{F} f_{\alpha} d\nu; \ \ \beta_1 = \int_{\partial F} f_{\alpha} d\sigma_F,
\end{equation}
where $d\sigma_{\partial F}$ is the induced measure on the boundary of each $F_i \in F$ (viewed itself as a Delzant polytope in ${\R}^{n-1}$). 
The modified Futaki invariant ${\widetilde \cF_{X,L,Z}}(\alpha)$ is then
\begin{equation}\label{toric-modified-futaki}
\frac{(2\pi)^n{\rm Vol}(\Delta)}{2} {\widetilde \cF_{X,L,Z}}(\alpha) 
     = \int_{\partial \Delta \setminus F} f_{\alpha} d\nu - \frac{{\bf s}_{(\Delta, F)}}{2}\int_{\Delta} f_{\alpha} d\mu,
\end{equation}
with
\begin{equation}\label{average-s}
\begin{split}
{\bf s}_{(\Delta, F)}&=2\frac{{\rm Vol}(\partial\Delta\setminus F)}{{\rm Vol}(\Delta)}\\
 &=2n \Big(\int_X (c_1(TX)+c_1(\cO(-Z)) \wedge c_1(L)^{n-1}/\int_X c_1(L)^n\Big).
 \end{split}
\end{equation}

The extremal $\C^{\times}$-action $\chi$ has Killing potential which is an affine linear function $s_{\Delta}$ determined by requiring
\begin{equation*} \label{extremal-function}
 \int_{\partial \Delta} f d\nu - \frac{1}{2}\int_{\Delta} f s_{\Delta} d\mu =0
 \end{equation*}
for any affine function $f$ (see e.g. \cite{Sz}). Then, as shown in \cite{Do2, Sz}, the relative Donaldson--Futaki invariant is given by
\begin{equation}\label{toric-relative-futaki}
\frac{(2\pi)^n{\rm Vol}(\Delta)}{2}\cF^{\chi}(\alpha) =  \int_{\partial \Delta} f_{\alpha} d\nu - \frac{1}{2}\int_{\Delta} f_{\alpha} s_{\Delta}d\mu
\end{equation}
while its modified version is 
\begin{equation}\label{toric-modified-relative-futaki}
\frac{(2\pi)^n{\rm Vol}(\Delta)}{2} {\widetilde \cF^{\chi}_{X,L,Z}}(\alpha) 
     = \int_{\partial \Delta \setminus F} f_{\alpha} d\nu 
       - \frac{1}{2}\int_{\Delta} f_{\alpha} s_{(\Delta, F)} d\mu
\end{equation}
where $s_{(\Delta, F)}$ again is the unique affine function such that  
\begin{equation*}
 \int_{\partial \Delta \setminus F} f d\nu - \frac{1}{2}\int_{\Delta} f s_{(\Delta, F)} d\mu=0
\end{equation*}
for any affine linear function $f$.

Donaldson generalizes the above expression for $\cF(\alpha)$ by considering convex piecewise affine linear functions $f_{\alpha}$ with integer coefficients. 
He associates to such an $f_{\alpha}$ a test configuration $({\mathcal X}, \cL)$, called {\it toric},  
and identifies the Donaldson--Futaki invariant of the central fibre $(X_0,L_0)$ with \eqref{toric-futaki}. 
Sz\'ekelyhidi~\cite[\S~4.1]{Sz} shows that \eqref{toric-relative-futaki} computes the relative Donaldson--Futaki invariant for such test configurations. 
These computations generalize easily in the case of a pair $(X, Z)$ where  the divisor $Z$ corresponds to the  preimage of a number of facets of $\Delta$ by the moment map. 
In this case,  the toric test configurations come equipped with a divisor ${\mathcal Z}$ which defines a flat family for $Z$; 
furthermore,  \eqref{toric-modified-futaki} and \eqref{toric-modified-relative-futaki} 
compute the modified Donaldson--Futaki and relative Donaldson--Futaki invariant of toric test configurations, respectively. 
We are thus led to the following:

\begin{defn}\label{d:toric-stability} 
 Let $(X, L)$ be a toric polarized variety and $Z\subset X$ a divisor corresponding to the preimage under the moment map of the union $F=F_{i_1}\cup \cdots \cup F_{i_k}$ 
 of some facets of the momentum polytope $\Delta$. We say that $(X, Z, L)$ is {\it relative K-stable} with respect to toric degenerations if 
\begin{equation}\label{toric}
\cL_{(\Delta,F)} (f) := \int_{\partial \Delta \setminus F} f d\nu - \frac{1}{2}\int_{\Delta} f s_{(\Delta, F)} d\mu >0
\end{equation}
for any convex, piecewise affine linear  function $f$ which is not affine linear on $\Delta$. 
Recall that $s_{(\Delta, F)}$ is by definition the unique affine linear function such that \eqref{toric} vanishes for any affine linear function $f$, 
and is called the  {\it extremal affine linear function} of $(\Delta, F)$.  
If \eqref{toric} is satisfied, we shall refer to $(\Delta, F)$ as a {\it stable} pair.
\end{defn}

\begin{lemma}\label{l:toric-numerical} Let $(X, L)$ be a toric polarized variety and $Z\subset X$ a divisor 
corresponding to the preimage under the moment map of  one facet $F$ of  the momentum polytope $\Delta$.  
Then, the relative Sz\'ekelyhidi numerical constraint  is equivalent to
\begin{equation}\label{inequality-gabor}
F_{\chi}''(0) = \frac{1}{2} \int_{F} (s_F - {s}_{(\Delta,F)})d\nu_F >0,
\end{equation}
where $s_{(\Delta, F)}$ is the extremal affine linear function of $(\Delta, F)$ and $s_F$ is the extremal affine linear function of the facet $F$ 
{\rm (}seen as a Delzant polytope of an $(n-1)$ dimensional toric variety{\rm )}.
\end{lemma}
\begin{proof}
In the toric case, Ross--Thomas~\cite[\S~4.3]{RT} link their construction of degenerations to the normal cone to  toric test configurations: 
the degeneration to the normal cone of $Z \subset X$ corresponding to a facet $F \subset \Delta$ defined by the zero set of an affine linear function $L$ (with $L \geq 0$ on $\Delta$) 
is given by Donaldson's construction with $f_c= {\rm max}(0, c-L)$. 
Therefore, the corresponding relative modified Donaldson--Futaki invariant \eqref{toric-modified-relative-futaki} is 
\begin{equation} \label{chi-futaki}
\begin{split}
F_{\chi}(c) =  & \int_{\partial \Delta \setminus F} f_{c} d\nu - \frac{1}{2}\int_{\Delta} f_{c} s_{(\Delta,  F)} d\mu\\
            =  & \int_{\partial \Delta \setminus F} f_{c} d\nu - \frac{1}{2}\int_{\Delta} f_c {\bf s}_{(\Delta, F)} d\mu \\
               & -\frac{1}{2} \int_{\Delta} f_c ({s}_{(\Delta,F)}-{\bf s}_{(\Delta, F)}) d\mu.
\end{split}
\end{equation}
Note that the sum on the second line is $c_0/4$ times the function $F(c)$ introduced in Lemma~\ref{characterization} 
(and computed via \eqref{modified-futaki-slope}) and, for any affine function $\xi$, 
\begin{equation*}
 \frac{\partial^2}{\partial c^2}\Big( \int_{\Delta} f_c \xi  d\mu\Big)\Big|_{c=0} = \int_{F} \xi d\nu_F,
\end{equation*}
where, we recall,  $d\nu_F$ is determined via the defining equation $L=0$ for $F$
by  letting $-dL\wedge d\nu_F = d\mu$. 
Using \eqref{a-c}, \eqref{alpha-beta}, and \eqref{average-s} one then gets 
\begin{equation}\label{relative-constraint}
\begin{split}
F_{\chi}''(0) = & 2\Big(\alpha_1- (c_1-\frac{\alpha_0}{2})\frac{\alpha_0}{c_0}\Big)  -\frac{1}{2} \int_{F} ({s}_{(\Delta,F)}-{\bf s}_{(\Delta,F)}) d\nu_F\\
          = &{\rm Vol}(\partial F)- \frac{1}{2}\int_{F} {s}_{(\Delta,F)}d\nu_F \\
          & + \frac{{\rm Vol}(\partial \Delta \setminus F)}{{\rm Vol}(\Delta)} {\rm Vol}(F)-\frac{{\rm Vol}(\partial \Delta \setminus F)}{{\rm Vol}(\Delta)}{\rm Vol}(F)\\
          =& \int_{\partial F} d{\sigma}_F - \frac{1}{2}\int_{F} {s}_{(\Delta,F)}d\nu_F\\
          =& \frac{1}{2} \int_{F} (s_F - {s}_{(\Delta,F)})d\nu_F,
 \end{split}
 \end{equation}        
where $s_F$ denotes the extremal affine function corresponding to $Z$. \end{proof}

\begin{lemma}\label{l:numerical/stability} Let $(X, Z)$ be as in Lemma~\ref{l:toric-numerical}.  If $(\Delta, F)$ is stable, then 
\begin{equation*}\label{inequality}
\int_{F} (s_{F} - {s}_{(\Delta, F)})d\nu_{F} \ge 0,
\end{equation*}
\end{lemma}
\begin{proof} Using the expression  $F_{\chi}(c)= \frac{c_0}{4}F(c) -\frac{1}{2}\int_{\Delta}f_c(s_{(\Delta,F)}- {\bf s}_{(\Delta,F)}) d\mu$ 
in \eqref{chi-futaki} and  Lemma~\ref{characterization}, one easily computes  that $F_{\chi}(0)= F'_{\chi}(0)=0$. 
It thus follows from \eqref{relative-constraint} that for $c$ sufficiently small, the piecewise affine linear convex function $f_c= {\rm max}(0, c-L)$ will destabilize $(\Delta, F)$, should 
$\int_{F} (s_{F} - {s}_{(\Delta, F)})d\nu_{F} = 2 F_{\chi}''(0) < 0.$ \end{proof}

\section{Labelled polytopes and the Abreu--Guillemin theory for K\"ahler metrics of Poincar\'e type}~\label{s:abreu-guillemin}

\subsection{Donaldson metrics on a labelled polytope}
The discussion in Section~\ref{s:toric} can be put in a broader framework which makes sense 
for {\it any} labelled convex compact simple polytope $(\Delta, {\bf L})$ in $(\R^n)^*$.
\begin{defn} Let $\Delta \subset (\R^n)^*=\tor^*$ be a compact convex polytope defined by a system of $d$  linear  inequalities
$$\Delta =\{ x\in (\R^n)^* : L_j(x)=\langle e_j, x\rangle + \lambda_j \ge 0, \ \ j=1, \ldots, d\}$$
where ${\bf L}=\{L_1(x), \ldots, L_d(x)\}$ are affine linear functions on $(\R^n)^*$  and $dL_j :=e_j \in \R^n$ are inward normals to $\Delta$. 
We suppose that $\Delta$ is {\it simple} in the sense that for each vertex $v$, 
there are precisely $n$ affine linear  functions $L_{v,1}, \ldots, L_{v,n}$ in ${\bf L}$ which vanish at $v$ and the corresponding inward normals 
$\{e_{v,1}, \ldots e_{v,n}\}$ form a basis of $\R^n$.  
We refer to such date $(\Delta, {\bf L})$ as a {\it labelled} (simple, compact, convex) polytope. 
Notice that, by Delzant's theorem~\cite{Delzant}, $(\Delta, \bf L)$ is the momentum image of a compact smooth toric variety 
if the labeling ${\bf L}$ satisfies the integrality condition that at each vertex $v$, 
${\rm span}_{\Z}\{u_{v,1}, \ldots, u_{v,n}\} $ is a fixed lattice $\Lambda \subset \R^n$. 
We shall refer to such labelled polytopes $(\Delta, {\bf L})$ as {\it Delzant polytopes}.
\end{defn}
In the case when $(\Delta, {\bf L})$ is Delzant, the works \cite{Abreu0, Guillemin} give an effective way to parametrize $\T$-invariant, 
$\omega$-compatible K\"ahler metrics $g$ on the toric symplectic manifold $(X, \omega)$ classified by $(\Delta, {\bf L})$ 
in terms of strictly convex smooth functions $u(x)$ defined on the interior $\Delta^0$ of $\Delta\subset (\R^n)^*$ and satisfying certain boundary conditions on $\partial \Delta$. Specifically,  the K\"ahler metric  $g$ is written on $X^0=\mu^{-1}(\Delta^0)$ as
\begin{equation}\label{toric-g}
g =\sum_{i,j=1}^n ( u_{,ij} dx_i \otimes dx_j  + u^{,ij} dt_i \otimes dt_j)
\end{equation}
where $(x_1, \ldots, x_n)$ are the Euclidean coordinates on $(\R^n)^*$, $(u_{,ij}) = {\rm Hess}(u)$ 
(and we tacitly identify smooth functions and tensors on $\Delta^0$ with their pull-backs via $\mu$ on $X^0$) and $(t_1, \ldots, t_n)$ 
are angular ($2\pi$-periodic) coordinates obtained by fixing a point $p_0\in X^0$ 
and identifying $X^0 \cong (\C^{\times})^n$ with the principal orbit of $p_0$ under the complexified action $\T^c$ 
(with respect to the complex structure $J$ determined by $g$ and $\omega$). In this formalism, the symplectic from is
$$\omega = \sum_{i=1}^n dx_i \wedge dt_i.$$
A central fact in this theory (see \cite{Abreu1,Donaldson-estimates}) is that  \eqref{toric-g} extends to a smooth Riemannian metric on $X$ 
if and only if $u$ satisfies the following {\it Guillemin boundary conditions}:
\begin{defn}\label{d:guill} Let $(\Delta, {\bf L})$ be a labelled convex compact simple polytope in $(\R^n)^*=\tor^*$. 
We say that a strictly convex smooth function  $u$ on $\Delta^0$ satisfies the {\it Guillemin boundary conditions} if 
\begin{enumerate}
\item[$\bullet$] $u - \frac{1}{2}\sum_{k=1}^{d} L_k \log L_k$ is smooth on $\Delta$, and
\item[$\bullet$] the restriction of $u$ to the interior $F^0$ of any face $F\subset \Delta$ is smooth and strictly convex.
\end{enumerate}
We denote by $\cS(\Delta, {\bf L})$ the space of such $u$.
\end{defn}

An example of a function in $\cS(\Delta, {\bf L})$ is  (see \cite{Guillemin})
\begin{equation}\label{u0}
u_0:= \frac{1}{2} \sum_{k=1}^d L_k \log L_k, 
\end{equation}
which, in the Delzant case,  characterizes the induced K\"ahler metric on $X$  via the K\"ahler reduction of the flat metric on $\C^d$.

\smallskip

The space $\cS(\Delta, {\bf L})$ can be equivalently characterized in terms of {\it first order boundary conditions}:
\begin{prop}~\cite[Prop.1]{ACGT}\label{first-order-boundary} 
The space $\cS(\Delta, {\bf L})$ consists of all smooth functions $u$ on $\Delta^0$ such that ${\bf H}^u:= ({\rm Hess}(u))^{-1}$ satisfies
\begin{enumerate}
\item[$\bullet$] {\rm [smoothness]} ${\bf H}^u$ extends smoothly on $\Delta$ as an $S^2(\tor^*)$-valued function;
\item[$\bullet$]  {\rm [boundary conditions]}  For any  facet $F_j \subset {\partial \Delta}$ with normal $e_j=dL_j$, and $x \in F_j$
\begin{equation}\label{boundary-divisor}
{\bf H}^u_{x}(e_j, \cdot)=0; \ \ (d{\bf H}^u)_{x} (e_j,e_j) =2e_j;
\end{equation}
\item[$\bullet$] {\rm [positivity]} ${\bf H}^u$ is positive definite on $\Delta^0$, 
 as well as on the interior $\Sigma^0$ of any face $\Sigma \subset \Delta$, 
 viewed there as a smooth function with values in $S^2(\mathfrak{t}/\mathfrak{t}_{\Sigma})^*$ 
 where $\mathfrak{t}_{\Sigma}$ denotes the subspace spanned by normals to facets containing $\Sigma$.
\end{enumerate}
\end{prop}
The extremality of the  K\"ahler  metric \eqref{toric-g} with $u\in \cS(\Delta, {\bf L})$ reduces to solving the  {\it Abreu equation}~\cite{Abreu0}
\begin{equation}\label{abreu}
-\sum_{i,j=1}^n \frac{\partial^2 {\bf H}^u_{ij}}{\partial x_i \partial x_j} = s_{(\Delta, {\bf L})},
\end{equation}
for an affine linear function $s_{(\Delta, {\bf L})}$, pre-determined by the labelled polytope $(\Delta, {\bf L})$ by requiring that
$$\cL_{(\Delta, {\bf L}) }(\varphi) := 2\int_{\partial \Delta} \varphi d\nu_{\bf L} - \int_{\Delta} s_{(\Delta, {\bf L})} \varphi d\mu =0$$
for any affine linear function $\varphi$, where $d\mu$ is a (fixed) Lebesgue measure on $\tor^* =(\R^n)^*$ and $d\nu_{\bf L}$ is obtained from $d\mu$ and ${\bf L}$ 
via \eqref{measure}. 
In this setting, we recall the following
\begin{defn} A labelled compact convex polytope $(\Delta, {\bf L})$ in  a vector space $\tor^*$ is  called {\it stable} if 
$$\cL_{(\Delta, {\bf L}) }(\varphi)  \ge 0$$
for any convex, piecewise affine linear function $\varphi$, and  equality is achieved only when $\varphi$ is affine linear.
\end{defn}
The stability of $(\Delta, {\bf L})$ is a necessary condition for  a solution $u\in \cS(\Delta, {\bf L})$ to \eqref{abreu} to exist~\cite{ZZ}, and  as first observed in \cite{Guan}, in the latter case $u$ must be, up to addition of affine linear functions, the unique critical point  ($=$ the minimum) of the convex relative Mabuchi functional 
\begin{equation}\label{relative-mabuchi}
\cM_{(\Delta, {\bf L})} (u) := \cL_{(\Delta, {\bf L})}(u) - \int_{\Delta} \log \det ({\bf H}^u) d\mu,
\end{equation}
which, as shown in \cite{Do2}, then takes values in $(0, \infty]$.

\bigskip
It is observed in \cite[p. 344]{Do2} that most of the above theory readily generalizes to the case when  for a fixed subset $F=F_1\cup \cdots \cup F_k$ 
which is the union of facets of $(\Delta, {\bf L})$ one modifies the induced measure $d\nu_{\bf L}$ to be zero on $F$. 
By \eqref{measure}, for each facet $F_i \subset F$,  the modified measure   can  be thought of as the limit  $\lim_{t\to \infty} d\nu_{tL_i}$, 
i.e. the measure obtained as in \eqref{measure} when sending the corresponding label $L_i$ to infinity.  
There is a subtle point here, however. 
It is not immediately clear how to extend the Guillemin boundary conditions of Definition~\ref{d:guill} over such limits.  
On the other hand, the equivalent first order boundary conditions given by Proposition~\ref{first-order-boundary}  extend naturally, as observed in \cite{ACG2}:
\begin{defn}\label{d:guillemin-boundary} Let $(\Delta, {\bf L})$ be a labelled convex compact simple polytope in $\tor^*$ and $F=F_1\cup \cdots \cup F_k$  
the union of some of its facets. We denote by $\cS(\Delta,{\bf L},F)$ the functional space of $u\in \cC^{\infty}(\Delta^0)$ verifying the first order boundary conditions
\begin{enumerate}
\item[$\bullet$] {\rm [smoothness]} ${\bf H}^u$ extends smoothly on $\Delta$;
\item[$\bullet$] {\rm [boundary conditions] }  for any facet $F_i \subset F$ and any point $x \in F_i$, 
\begin{equation}\label{boundary-divisor-0}  
{\bf H}^u_{x}(e_i, e)=0; \ \ \big(d{\bf H}^u(e_i,e)\big)_x =0,
\end{equation}
where $e_i=dL_i$ is the inward normal to $F_i$ defined by ${\bf L}$ and $e\in \tor$, and, for any facet $F_r$ which is not in $F$, and  $x \in F_r$, 
\begin{equation}\label{boundary-divisor-1} {\bf H}^u_{x}(e_r, e)=0; \ \ \big(d{\bf H}^u(e_r,e_r)\big)_x =2e_r.
\end{equation}  
\item[$\bullet$] {\rm [positivity]} ${\bf H}^u$ is positive definite on $\Delta^0$, as well as on the interior of any face $\Sigma \subset \Delta$, 
                                    viewed there as a smooth function with values in $S^2(\mathfrak{t}/\mathfrak{t}_{\Sigma})^*$ 
                                    where $\mathfrak{t}_{\Sigma}$ denotes the subspace spanned by normals to facets containing $\Sigma$.
\end{enumerate}
\end{defn}
\begin{rem} Manifestly, the conditions \eqref{boundary-divisor-0} are independent of the choice of labels $L_i$ for the facets $F_i \subset F$, 
and are obtained from \eqref{boundary-divisor-1} by letting $e_r \to \infty$.
\end{rem}
\begin{rem} Symplectic potentials satisfying Definition \ref{d:guillemin-boundary} do not necessarily correspond to Poincar\'e type metrics. 
However, below we shall define subspaces $\cS_{\alpha, \beta}(\Delta, {\bf L}, F)$ of $\cS (\Delta, {\bf L}, F)$ 
depending on a positive parameter $\alpha$ and a real parameter $\beta$ which do induce metrics of Poincar\'e type on $X \setminus Z$.
\end{rem}
We are thus interested to find solutions of \eqref{abreu} in $\cS(\Delta, {\bf L}, F)$, where,  by the integration by parts argument from \cite{Do2}, 
the right hand side again must be the unique  affine linear function  function $s_{(\Delta, {\bf L}, F)}$, called {\it extremal affine function},  satisfying 
\begin{equation}\label{extremal-function-F}
\cL_{(\Delta,{\bf L}, F)}(f)=\int_{\partial \Delta \setminus F} f d\nu_{\bf L} - \frac{1}{2}\int_{\Delta} f s_{(\Delta, {\bf L}, F)} d\mu=0
\end{equation} for any affine linear function $f$.
We also have the following straightforward  extension of arguments in the case of $(\Delta, {\bf L})$:
\begin{prop}\cite{Do2, Guan, ZZ}\label{necessary} 
Suppose there exists a function $u\in \cS(\Delta, {\bf L}, F)$ which solves the Abreu equation
\begin{equation}\label{abreu1}
-\sum_{i,j=1}^n \frac{\partial^2 {\bf H}^u_{ij}}{\partial x_i \partial x_j} = s_{(\Delta, {\bf L}, F)},
\end{equation}
where $s_{(\Delta, {\bf L}, F)}$ is the extremal affine linear function of $(\Delta, {\bf L}, F)$. Then
$\cL_{(\Delta, {\bf L},F)}(\varphi)  \ge 0$  for any convex, piecewise affine linear function $\varphi$, with equality  iff $\varphi$ is affine linear.
\end{prop}
\begin{defn}\label{d:Donaldson-metric} A labelled convex, compact, simple polytope $(\Delta, {\bf L})$ in $(\R^n)^*$  
with a fixed subset $F$ of facets satisfying the conclusion of Proposition~\ref{necessary} will be referred to as {\it stable} triple $(\Delta, {\bf L}, F)$. 
A K\"ahler metric on $g_D$ on $\Delta^0 \times \T$ defined by a solution $u\in \cS(\Delta, {\bf L}, F)$ of \eqref{abreu1}  
(if it exists) will be called the  {\it Donaldson metric} of $(\Delta, {\bf L}, F)$.
\end{defn} 
The geometric interest of Donaldson metrics as above comes from  \cite[Conjectures~ 7.2.3]{Do2}:
\begin{conj}[\bf Donaldson~\cite{Do2}]\label{conjecture2}  Let $(\Delta, {\bf L})$  be the momentum polytope of a smooth compact toric K\"ahler manifold $(X, \omega_0)$ 
and $Z$  the divisor in $X$ corresponding to the momentum  preimage of the union $F$  of facets of $\Delta$.  
If $(\Delta, {\bf L}, F)$ is stable, then there exists a  complete extremal K\"ahler metric $g_D$,  defined on $X\setminus Z$.
\end{conj}
We now use the results from \cite{Auvray1, Auvray2} in order to establish the following
\begin{thm}\label{hugues} Let $(X, L)$ be an  $n$-dimensional polarized toric projective variety under the action of the $n$-dimensional real torus $\T$, 
and $Z\subset X$ a smooth divisor corresponding to the preimage under the moment map of a facet $F$ of  the momentum polytope $\Delta$.  
If $X\setminus Z$ admits a $\T$-invariant extremal K\"ahler metric of Poincar\'e type in $c_1(L)$, 
then  $(\Delta, F)$ is stable and the relative Sz\'ekelyhidi numerical constraint \eqref{inequality-gabor} holds. Furthermore, the Delzant polytope $F$ is stable and 
$$ s_F - ({s}_{(\Delta,F)})_{|_F} = const >0.$$
\end{thm}
\begin{proof}  The main point is to show that a $\T$-invariant  extremal  K\"ahler metric $(g, \omega)$  of Poincar\'e type on $X\setminus Z$ 
gives rise to a Donaldson metric in a slightly weaker sense, 
namely it corresponds to ${\bf H}^u \in C^{\infty}(\Delta^0, S^2(\tor^*))$ which extends smoothly on $\Delta\setminus F$ and $C^0$ on $\Delta$ and,  moreover, 
the conditions \eqref{boundary-divisor-0} and \eqref{boundary-divisor-1} hold where the first order condition at $F$ is taken in the sense of limit, 
i.e. $$\lim_{x\to F, x \in \Delta\setminus F} \big(d{\bf H}^u(e_F, \cdot)\big)_x =0,$$
for  $e_F\in \tor$ the inward normal to $F$. This will be enough  to establish the integration by parts formula
$$\int_{\Delta} \Big(\sum_{i,j=1}^n H^u_{ij,ij}\Big) \varphi d\mu 
= \int_{\Delta}\Big(\sum_{i,j=1}^n H^u_{ij}\varphi_{,ij}\Big) d\mu - 2 \int_{\partial \Delta \setminus F} \varphi d\nu_{\bf L}$$
for any smooth function $\varphi$ on $\Delta$. The latter in turn would imply that:  
\begin{enumerate}
\item[(a)] ${\rm Scal}_g = s_{(\Delta, F)}$ and 
\item[(b)] $(\Delta, F)$ is stable (compare with Proposition~\ref{necessary} above). 
\end{enumerate}

With the conclusions (a) and (b) in place, the result follows easily from \cite{Auvray1,Auvray2}. 
Indeed,  (a) and Lemma~\ref{l:toric-numerical}  together with \cite[Thm.~4 \& Prop.~2.1]{Auvray2} 
show that \eqref{inequality-gabor} is a necessary condition for the existence of an extremal K\"ahler metric of Poincar\'e type on $X\setminus Z$. 
Furthermore, by \cite[Thm. 4]{Auvray1},  $Z$ must admit an extremal K\"ahler metric $\check g$  in the K\"ahler class $c_1\big (L_{|_Z}\big)$, 
so that $F$ must be a stable Delzant polytope by the result in \cite{ZZ}. 
It is also shown in \cite[p.44]{Auvray1} that the extremal vector fields $J{\rm grad}_{g} {\rm Scal}_g$ and $J{\rm grad}_{\check g} {\rm Scal}_{\check g}$ agree on $Z$, 
which in our case translates to say that $s_F-(s_{(\Delta, F)})_{|_F} =const$. 
The constant is positive because of \eqref{inequality-gabor}.

\smallskip
We thus focus for the remainder of the proof to show that an extremal $\T$-invariant Poincar\'e type metric $(g, \omega)$ on $X\setminus Z$ is (weakly) Donaldson. 
To this end, we fix a $\T$-invariant K\"ahler metric $\omega_0 \in c_1(L)$  on $X$ and denote by $(\Delta, {\bf L})$ the corresponding Delzant polytope. 
We shall write, for any basis $\{e_1, \ldots, e_n\}$ of $\tor$, $x^0=(x_1^0, \ldots, x^0_n)$ the corresponding momenta, 
viewed as functions from $X$ to $\tor^*$ defined by $\imath_{K_j}\omega_0 = - dx^0_j$ where $K_j$ is the fundamental vector field of $X$ corresponding to  $e_j \in \tor$; 
thus $\Delta= {\rm Im}(x^0)$ and $Z= (x^0)^{-1}(F)$ for a facet $F\subset \Delta$. 
Let $v\in \Delta$ be a vertex of $\Delta$ and $F$, and $\{e_1, \ldots, e_n\}$ the basis of $\tor$  formed by the inward normals to the facets containing $v$,  with $e_F=e_1$. 
By  Delzant theory (see \cite{Delzant, LT})  there exists a $(\C^{\times})^n$ equivariant chart $\C^n_v$ of $X$ 
(with respect to the complexified $(\C^{\times})^n$-action of $\T$ on $X$ and the standard  $(\C^{\times})^n$-action on $\C_v^n =\C^n$) in which $F$ is given by $z_1=0$. 
Furthermore, in this chart, $|z_j|^2= x_j^0 e^{\phi_j(z)}$ for smooth functions $\phi_j$ on $X$ (see e.g. \cite{Do2}) 
whereas the holomorphic vector fields $\frac12(K_j  - \sqrt{-1}  JK_j)$  become $\sqrt{-1} z_j \frac{\partial}{\partial z_j}$.

According to Definition~\ref{d:Poincare-type}, we can write $\omega= \omega_0 + dd^c \varphi$ for a smooth $\T$-invariant function $\varphi$ on $X\setminus Z$, 
such that $d\varphi$ is bounded at any order  with respect  to the model metric 
$$ \omega_{\rm mod} =\sqrt{-1} \Big(\frac{1}{|z_1|^2(\log|z_1|)^2} dz_1 \wedge d\bar z_1 +  \sum_{j=2}^n dz_j \wedge d\bar z_j \Big)$$
defined on the  chart $\C_v^n$: in particular 
\begin{equation}\label{estimates}
d\varphi(JK_j)= O(|z_j|), j=2, \ldots, n, \ \ \ d\varphi(JK_1) = O\Big(\frac{1}{|\log(|z_1|^2)|}\Big).
\end{equation}
Writing
\begin{equation}\label{poincare-momenta}
x_j = x_j^0 + d\varphi(JK_j), \ \  j=1, \ldots n,
\end{equation}
for the momenta of $(g, \omega)$, we see that the map $x^0 \to x$  sends $\Delta \setminus F$ to itself, preserving  the faces. 
Furthermore, $x_1 : X\setminus Z \to \Delta\setminus F$ extends continuously as zero over $Z$.

We now let ${\bf H}_x(e_i,e_j)=(g_p(K_i, K_j))$ be the smooth 
 $S^2(\tor^*)$-valued function, 
defined on $\Delta^0$ by using the extremal K\"ahler metric $g$ and the momentum map $x$ (with $x=x(p)$ for $p\in X\setminus Z$). 
Clearly, ${\bf H}$ extends smoothly over $\Delta\setminus F$. 
The proof of Proposition~\ref{first-order-boundary} (given in \cite{ACGT}) 
uses local arguments around a point  on a facet $F_r \subset \overline{\Delta\setminus F}$, 
and thus shows that ${\bf H}$ satisfies the boundary conditions \eqref{boundary-divisor-1} on  each $F_r$. 
We now focus on $F$. 
We use the chart $\C^n_v$ as above,  and denote by $\pi_1 : \C_v^n \to \C^{n-1}$  the projection $\pi_1(z_1, z_2, \ldots, z_n) = (z_2, \ldots, z_n)$. 
Then, \cite[Thm.~4]{Auvray1} tells that as $z_1 \to 0$, $\omega$  is written as
\begin{equation}\label{poincare-split}
\omega = a_1\Big(\frac{\sqrt{-1}(dz_1 \wedge d\bar z_1)}{|z_1|^2 \big(\log(|z_1|^2)^2}\Big) + \pi_1
^* \omega_1 + O(|\log|z_1||^{-\delta}),
\end{equation}
where $a_1$ and $\delta$ are positive reals, $\omega_1$ is an extremal K\"ahler metric on $Z\cap \C_v^n=\{(0,z_2, \ldots, z_n)\}$, 
and $O(|\log|z_1||^{-\delta})$ is understood at any order with respect  to the K\"ahler metric 
\begin{equation}   \label{e:split-model}
 a_1\Big(\frac{\sqrt{-1}(dz_1 \wedge d\bar z_1)}{|z_1|^2 \big(\log(|z_1|^2)^2}\Big) + \pi_1^* \omega_1.
\end{equation}
We compute from \eqref{poincare-split}, with respect to the vector fields $\frac12(K_j  - \sqrt{-1}  JK_j)=\sqrt{-1} z_j \frac{\partial}{\partial z_j}$, 
\begin{equation}   \label{e:H}
\begin{split}
{\bf H}(e_1,e_1) &= \omega(K_1, JK_1)=\frac{2a_1}{(\log(|z_1|^2))^2} + O(|\log|z_1||^{-\delta-2}), \\
{\bf H}(e_1, e_j) &= \omega(K_1, JK_j)= O(|z_j||\log|z_1||^{-\delta-1}), \\
{\bf H}(e_i, e_j) & =\omega(K_i, JK_j)= \check{\bf H}^1(e_i,e_j) + O(|z_iz_j||\log|z_1||^{-\delta}), i,j \ge 2,
\end{split}
\end{equation}
where $\check{\bf H}^1(e_i, e_j)= \pi^*\omega_1(K_i, JK_j)$ is the $z_1$ independent smooth function computed from $\omega_1$ 
with respect to the induced vector fields on $Z$. It follows that ${\bf H}$ extends continuously on $F$, verifying ${\bf H}_x(e_1, \cdot)=0$ on $F$.

Taking interior product with $K_1$ and $K_i$, $i\geq 2$, in  \eqref{poincare-split} we obtain 
\begin{equation*}
 \begin{aligned}
  dx_1 &= a_1 \Big(\frac{d|z_1|^2}{|z_1|^2(\log|z_1|^2)^2}\Big) +  \sum_{j=1}^n f_j(z) d|z_j|^2,    \\
  dx_i &= 2\sum_{j=2}^n \check{\bf H}^1(e_i,e_j) d|z_j| + \sum_{j=1}^n f_{ij}(z) d|z_j|^2 = \pi_1^*dx_i^1 + \sum_{j=1}^n f_{ij}(z) d|z_j|^2, 
 \end{aligned}
\end{equation*}
where  $f_1(z)= O\Big(\frac{1}{|z_1|^2(\log|z_1|^2)^{2+\delta}}\Big), \ f_j(z)= O(|\log|z_1||^{-1-\delta}), \ j=2, \ldots, n$, 
and $f_{i1}(z)= O\Big(\frac{|z_i|^2}{|z_1|(\log|z_1|^2)^{1+\delta}}\Big), \ f_{ij}(z)= O(|z_i|^2|\log|z_1||^{-\delta}), \ i,j=2, \ldots, n$; 
here $dx_i^1=-\iota_{K_i}\omega_1$ for $i=2, \ldots, n$, and $dx_i^1 = \sum_{j=2}^n A_{ij} d|z_j|^2$ for some 
invertible matrix $(A_{ij})$, locally uniformly bounded together with its inverse $(A^{ij})$ on $Z\cap \C^n_v$. 
Putting $\sigma_1 = \frac{d|z_1|^2}{|z_1|^2|\log(|z_1|^2)|}$ and $\sigma_i = d|z_i|^2$, 
the relations above can be recapped as 
 \begin{equation*}
  \begin{pmatrix}
   \big|\log|z_1|\big|dx_1 \\ dx_2 \\ \vdots \\ dx_n 
  \end{pmatrix}
   = 
  \Bigg[
  \begin{pmatrix}
    a_1      & 0     \\
   0         & (A_{ij})
  \end{pmatrix}
  + \varepsilon
  \Bigg]
  \begin{pmatrix}
   \sigma_1 \\ \vdots \\ \sigma_n 
  \end{pmatrix},
 \end{equation*}
with $\varepsilon = O(|\log|z_1||^{-\delta})$;  
solving this system provides 
 \begin{equation*}
  \sigma_1 = \frac{|\log|z_1||}{a_1}dx_1 + \sum_{j=1}^n \eta_{1j} dx_j, \ \ \
  \sigma_i = \eta_{i1} dx_1 + \sum_{j=2}^n (A^{ij}+\eta_{ij}) dx_j, \ i=2,\ldots,n,
 \end{equation*}
with $\eta_{i1}= O(|\log|z_1||^{1-\delta})$, $i=1,\ldots,n$, and $\eta_{ij} = O(|\log|z_1||^{-\delta})$, $i=1,\ldots,n, \ j=2,\ldots,n$. 
Differentiating the first two lines of \eqref{e:H} 
with respect to $z_1\log|z_1|\frac{\partial}{\partial z_1}$, the $\frac{\partial}{\partial z_i}$'s ($i\geq2$), 
and their conjugates\footnote{using that we can replace $dz_i$ and $d\overline{z_i}$ by $\sigma_i$ with help of the torus action, 
as in the estimates for $f_j$ and $f_{ij}$ above.},  
implies:
\begin{equation*}
 \begin{aligned}
 d{\bf H}(e_1, e_1) &= - \frac{4}{\log|z_1|}\Big((1+\epsilon_{11})dx_1 + \sum_{j=2}^n \epsilon_{1j} dx_j\Big), \\
 d{\bf H}(e_1, e_i) &= \epsilon_{i1}dx_1 + \sum_{j=2}^n \epsilon_{ij} dx_j, \ \ i=2,\ldots,n, 
 \end{aligned}
\end{equation*}
with $\epsilon_{i1}=O(|\log|z_1||^{-\delta})$, $i=1,\ldots,n$, and 
$\epsilon_{ij}=O(|\log|z_1||^{-\delta-1})$, $i=1,\ldots,n, \ j=2,\ldots,n$. 
Hence in particular: 
$$\lim_{x\to F} (d{\bf H}(e_1, e_i))_x =0, \ \ i=1, \ldots, n, $$
as claimed.
\end{proof}

\subsection{Conjectural picture for the existence of  extremal toric metrics of Poin\-car\'e type} 
Theorem~\ref{hugues}  and Conjecture~\ref{conjecture2}  motivate the  following conjectural picture in the toric case:
\begin{conj}\label{conjecture1}  
A smooth toric variety  $(X,L)$ with momentum polytope $\Delta$ and a divisor $Z \subset X$ 
corresponding to the preimage of the union $F=\cup_{i}F_{i}$ of some facets $F_i$ of $\Delta$ admits an extremal toric K\"ahler metric of Poincar\'e type 
if the following three conditions  are satisfied:
\begin{enumerate}
\item[\rm (i)] $(\Delta, F)$ is stable,   and
\item[\rm (ii)]   for any facet $F_{i} \subset F$, the pair $(F_i, F_{i} \cap (\cup_{j\neq i \in I} F_j))$  is stable,  and 
\item[\rm (iii)]  if $s_{(F_i, F_{i} \cap (\cup_{j\neq i \in I} F_j)} $ is the extremal affine function corresponding to $(F_i, F_{i} \cap (\cup_{j\neq i \in I} F_j))$, then 
\begin{equation}\label{hyperbolicity-condition}
s_{(F_i, F_{i} \cap (\cup_{j\neq i} F_j))} - {s}_{(\Delta,F)} = c_i>0,
\end{equation} 
where  $c_i$ are real constants.
\end{enumerate}
\end{conj}
\begin{rem} Theorem~\ref{hugues} readily generalizes to the case when $Z$ is a smooth toric divisor of $(X, L)$, 
i.e. $Z$ is the preimage under the moment map of the union $F$ of disjoint facets of $\Delta$, 
thus showing the necessity of the conditions (i),(ii),(iii) in this case. 
The situation is not so clear in general, when $Z$ has simple normal crossings. In this case we make the following remarks:
\begin{enumerate}
\item In order to establish (i) we would need to show that {\it any} toric extremal K\"ahler metric of Poincar\'e type on $X\setminus Z$ 
      belongs to the class $\cS(\Delta, {\bf L}, F)$, at least in the weaker sense as in the proof of Theorem~\ref{hugues}.
\item  (ii)  would follow from  (i), noting that  the extremal Poincar\'e type metric on $Z_i\setminus Z'_i$ where  $Z_i$ is the component of $Z$ corresponding to $F_i$ and $Z'_i$  is the divisor of $Z_i$ induced by $Z$ 
(which exists by virtue of \cite[Thm. 4]{Auvray1}) must be {\it toric}. 
      Indeed, this can be derived from \cite{Auvray1} as follows:
       \begin{itemize}
        \item[$\bullet$] using toric-equivariant coordinates $(z_1,\ldots, z_n)\in \C^n_v $ centred at a point in $Z_i$ fixed by the torus action,  
                         and such that $Z_i \cap \C^n_v = \{z_1 = 0\}$ and $Z \cap \C^n_v = \{z_1 \cdots z_s = 0\}$, 
                         the induced metric is a ${\mathcal C}^{\infty}_{\rm loc}$-limit 
                         of $\omega_v^{\epsilon_j} := \omega_{|_{\{z_1=\epsilon_j\}\setminus Z}}$ 
                         (the pull-back of $\omega$ to $\{z_1=\epsilon_j\}\setminus Z$ by inclusion), with $\epsilon_j \to 0$; 
        \item[$\bullet$] the metric $\omega$ and the hypersurfaces $\{z_1=\epsilon_j\}\setminus Z$ are invariant by the action of $\T/\T_{F_i}$;  
                         therefore, the $\omega_v^{\epsilon_j}$ are invariant under this action, 
                         and their ${\mathcal C}^{\infty}_{\rm loc}$-limit is thus toric. 
       \end{itemize}   
\item (iii)  would  follow  by  \cite[Thm. 4 \& Prop. 2.1]{Auvray2}, 
      once we know that the scalar curvature of the extremal K\"ahler Poincar\'e type metric coincides with $s_{(\Delta, F)}$. 
      This in turn would be the case if we establish point (i) above.
\end{enumerate}
Another interesting question is to know whether or not (i) implies (ii).
\end{rem}

\subsection{A class of Poincar\'e type toric K\"ahler metrics} 
To link Conjectures~\ref{conjecture2} and \ref{conjecture1},  
one needs a criterion ensuring that the Donaldson metric is of Poincar\'e type. We address this question in this  section.

We start by introducing a class of toric metrics  in the form \eqref{toric-g} on $(\Delta^0 \times \T)$  via a certain type of {\it Guillemin boundary conditions} 
for the corresponding symplectic potential $u$, depending on the data  $(\Delta, {\bf L}, F)$,  
compare with Definition~\ref{d:guillemin-boundary} in the case $F= \varnothing$. 
For simplicity, we shall assume that $(\Delta, {\bf L})$ is Delzant and $F=F_{1}$ is a single facet defined by the label $L_F(x):=L_{1}(x)=0$.
\begin{defn}\label{toric-poincare} Let $\alpha>0$ and $\beta \in \R$ be fixed real numbers. 
The class $\cS_{\alpha, \beta}(\Delta, {\bf L}, F)$ of symplectic potentials $u$ is defined as the space of smooth and strictly convex functions on $\Delta^0$, 
satisfying the following boundary conditions:
\begin{enumerate}
\item[$\bullet$] $u  + (\alpha +\beta L_{F}) \log L_F - \frac{1}{2}\sum_{j=2}^d L_j \log L_j$ is smooth on $\Delta$;
\item[$\bullet$] if ${\bf f}  \subset F$ is a sub-face  of $F$, 
                 then $u_{\bf f}: = u +  (\alpha +\beta L_{F})\log L_{F}$ restricts to the relative interior of ${\bf f}$ as a smooth strictly convex function;
\item[$\bullet$] if $\Sigma \not\subset F$ is regular face,  then $u$ restricts to the relative interior of $\Sigma$ as a smooth strictly convex function.
\end{enumerate}
\end{defn}

Our first observation is the following result, whose proof  is given  in Appendix~\ref{AppendixA}.
\begin{thm}\label{thm:criterion} Let $(X, \omega)$ be a smooth, compact symplectic toric manifold 
with momentum Delzant polytope  $(\Delta, {\bf L})$ and $F$ be a single facet of $\Delta$. 
Then, for any $u\in \cS_{\alpha, \beta}(\Delta, {\bf L}, F)$,  the metric \eqref{toric-g} defines on $X$ a $\T$-invariant complex structure $J$ 
such that the momentum preimage of $F$ is a smooth divisor $Z$ of $(X, J),$ and \eqref{toric-g} is a K\"ahler metric of Poincar\'e type on $X\setminus Z$. 
\end{thm}
Using arguments similar to those in \cite{ACGT} (see Appendix~\ref{AppendixA} for more details), 
one can relate the spaces $\cS_{\alpha, \beta}(\Delta, {\bf L}, F)$ and $\cS(\Delta, {\bf L}, F)$ as follows:
\begin{prop}\label{complete-first-order} 
The space $\cS_{\alpha, \beta}(\Delta, {\bf L}, F)$  is equivalently defined as the space of smooth functions on $\Delta^0$ 
such that ${\bf H}^u = ({\rm Hess}(u))^{-1}$ satisfies
\begin{enumerate}
\item[$\bullet$] {\rm [smoothness]} ${\bf H}^u$ extends smoothly on $\Delta$;
\item[$\bullet$] {\rm [boundary conditions on $F$] }  for any  $x \in F$ we have 
\begin{equation*}
\begin{split}
{\bf H}^u_{x}(e_F, e)&=0; \ \ (d{\bf H}^u)_{x} (e_F, e) =0, \\  
(d^2{\bf H}^u)_{x} (e_F,e_F) &=\frac{2}{\alpha} e_F \otimes e_F, \  \ (d^3{\bf H}^u)_{x} (e_F,e_F) = -\frac{6\beta}{\alpha^2} e_F^{\otimes 3},
\end{split}
\end{equation*}
where $e_F=dL_F$ is the inward normal to $F$ defined by ${\bf L}$,  $e$ is any vector in $\tor$, and  for a smooth function $f$ on $\tor$,  
$d^k f$ denotes the $k$-th covariant derivative of $f$ with respect to the flat affine structure  on $\tor^*$, so that $(d^k f)_x \in S^k(\tor )$;  
\item[$\bullet$] {\rm [regular boundary conditions]} for any facet $F_r $ with inward normal $e_r$ which is not in $F$, and  $x \in F_r$, 
\begin{equation*}{\bf H}^u_{x}(e_r, e)=0; \ \ (d{\bf H}^u)_{x} (e_r,e_r) =2e_r;
\end{equation*}  
\item[$\bullet$] {\rm [positivity]} ${\bf H}^u$ is positive definite on $\Delta^0$, as well as on the relative interior of any face $\Sigma \subset \Delta$, 
viewed there as a smooth function with values in $S^2(\mathfrak{t}/\mathfrak{t}_{\Sigma})^*$ 
where $\mathfrak{t}_{\Sigma}$ denotes the subspace spanned by normals to facets containing $\Sigma$. 
\end{enumerate}
In particular, $\cS_{\alpha, \beta}(\Delta, {\bf L}, F) \subset \cS(\Delta, {\bf L}, F)$.
\end{prop}

Our next result shows that the extremality assumption in fact determines uniquely the space $\mathcal{S}_{\alpha, \beta}(\Delta, {\bf L}, F)$.

\begin{prop}\label{p:alpha-beta} Suppose $u\in \cS_{\alpha, \beta}(\Delta, {\bf L}, F)$ is a solution of \eqref{abreu1}. 
Then  the real numbers $\alpha, \beta$ are  uniquely and explicitly determined from the data $(\Delta, {\bf L}, F)$.  
Furthermore, the solution  $u$ is unique modulo the addition of an affine linear function.
\end{prop}
\begin{proof} 
The uniqueness part is standard as each $\cS_{\alpha, \beta}(\Delta, {\bf L}, F)$ is a linearly convex space and, choosing a reference point 
$u' \in \cS_{\alpha, \beta}(\Delta, {\bf L}, F)$, we can consider the following modification of relative Mabuchi functional \eqref{relative-mabuchi}:
$$\cM_{(\Delta, {\bf L},F)} (u) := \cL_{(\Delta, {\bf L},F)}(u-u') - \int_{\Delta} \log \det ({\bf H}^u) d\mu.$$
The point is that $\cM_{(\Delta, {\bf L},F)} (u)$ is well-defined with values in $(-\infty, \infty]$, as $u-u'$ is a smooth function over $\Delta$. 
The  argument in \cite{Guan} shows that $\cM_{(\Delta, {\bf L},F)} (u)$ is convex and its minima, 
which  are unique up to the addition of affine linear function, are precisely the solutions of \eqref{abreu1}.

The fact that  $s_{(\Delta, {\bf L}, F)} - s_{(F, {\bf L}_F)}=const$ follows from \cite[p. 44]{Auvray1}  
when it is shown that the extremal vector of $Z$ equals to the vector field induced on $Z$ by the extremal vector field of $X\setminus Z$. 
In the toric case, this condition reads as $d((s_{(\Delta, {\bf L}, F)})_{|_F} - s_{(F, {\bf L}_F)})=0$.

\smallskip
It remains to determine $(\alpha, \beta)$ from $(\Delta, {\bf L}, F)$, which will occupy the remainder of the proof.  

\noindent
{\bf Step 1. Determining $\alpha$.} Let us  choose a basis $\{e_1, e_2, \ldots, e_n\}$ of $\tor$ (and $\{e_1^*, \ldots, e_{n}^*\}$  denote the dual basis of $\tor^*$),  
by fixing a vertex $v\in F$ of $\Delta$ and taking $e_j$ be the inward normals to the facets meeting $v$ with  $e_1 = e_F$ 
(and therefore $e_i^*, i=2, \ldots, n$  are tangent to $F$). 
We assume  furthermore that $v$ is at the origin  
(so that $F\subset \{x_1=0\}$) and we write ${\bf H}^{u}= (H_{ij})$ in the chosen basis, 
where $H_{ij}(x)$ are smooth functions on $\Delta$, see Proposition~\ref{complete-first-order}. 
As $u$ is a solution of \eqref{abreu1}, we have
\begin{equation}\label{solution}
s_{(\Delta, {\bf L}, F)} = -\sum_{i,j=1}^n H_{ij,ij}.
\end{equation}
We denote by $\check{\bf H}^u \in S^2((\tor/\tor_F)^*)$ the induced smooth  positive definite bilinear  form on $F$. 
It is easily seen (by continuity) that $\check{\bf H}^u$ satisfies the boundary conditions  of Proposition~\ref{first-order-boundary} 
with respect to the labeling ${\bf L}_{F}$ of $F$, see \cite[Rem.~1]{ACGT}. 
It thus defines an almost-K\"ahler metric $\check{g}_u$ on $F$ (which can be shown to be K\"ahler). 
With respect to our choice of basis of $\tor$, we can identify $\tor/\tor_F \cong \R^{n-1}= {\rm span}_{\R}\{e_2, \ldots, e_n\}$, 
so that we have $\check{\bf H}^u=(H_{ij})_F, i,j = 2, \ldots, n$.  
It thus follows that on $F$ we have
\begin{equation}\label{trick}
\begin{split}
s_{(\Delta, {\bf L}, F)}  &= -\sum_{i,j=1}^n H_{ij,ij} \\
                          &= -H_{11,11} - 2\sum_{j=2}^n H_{1j,1j} - \sum_{i,j=2}^n H_{ij,ij} \\
                          &= -\frac{2}{\alpha}  - \sum_{i,j=2}^n\check{\bf H}^u_{ij,ij} =  -\frac{2}{\alpha} + {\rm Scal}(\check{g}_u), 
\end{split}
\end{equation}
where we have used the boundary conditions of Proposition~\ref{complete-first-order} 
(or equivalently the form \eqref{eqn_matrixO} in our compatible coordinates) 
in order to see that $H_{1j,1j}=0$ on $F$ for $j>1$. 
It thus follows that $\check{g}_u$ is an extremal  almost-K\"ahler metric on $F$  and
\begin{equation}\label{extremal}
 s_{(\Delta, {\bf L}, F)}  = s_{(F, {\bf L}_F)}  - \frac{2}{\alpha}
\end{equation}
Integrating over $F$, we thus have 
\begin{equation}\label{alpha}
-\frac{2}{\alpha} = \int_F(s_{(\Delta, {\bf L}, F)} - s_{(F, {\bf L}_F)})d\nu_F/{\rm Vol}(F),
\end{equation}
which determines $\alpha$.

\smallskip
\noindent 
{\bf Step 2: Determining $\beta$.} In order to determine $\beta$,  notice that (see \eqref{measure}) at each point $p\in F$, we have $e_1 \wedge d\nu_F = - d\mu$  
where we recall that we have set $e_1=dL_{F}=e_F$. Furthermore,  with our choice of basis  we have  $(e_j \wedge d\nu_F) =0$ for $j=2, \ldots, n$. 
We thus have,   using \eqref{solution}, 
$$(d s_{(\Delta, {\bf L}, F)} \wedge d\nu_F) =  \Big(\sum_{i,j=1}^n H_{ij,ij1}(p)\Big)d\mu = c d\mu$$
for a real constant $c=c(\Delta, {\bf L}, F)$ determined from the polytope $(\Delta, {\bf L}, F)$. In other words, 
\begin{equation}\label{basic}
\begin{split}
c  &=  \Big(\sum_{i,j=1}^n H_{ij,ij1}\Big)_F \\
    &= \big(H_{11,111}\big)_F  + 2\big(\sum_{j=1}^nH_{1j, 11j}\big)_F  +\big(\sum_{i,j=2}^n H_{ij,1ij}\big)_F \\
    &= -\frac{6 \beta}{\alpha^2} + 2\big(\sum_{j=2}^nH_{1j, 11j}\big)_F  +\big(\sum_{i,j=2}^n H_{ij,1ij}\big)_F,
    \end{split}
    \end{equation}
where  in the last line we have used $(H_{11,111})_F = -\frac{6\beta}{\alpha^2}$, see Proposition~\ref{complete-first-order}. 
We are going to integrate \eqref{basic} over $F$, and to this end we are going  to use 
the integration by parts formula  
\begin{equation}\label{byparts}
\int_F \sum_{j=2}^n V_{j, j}d\nu_F = - \sum_{\Sigma \subset \partial F} \int_{\Sigma} \langle V, e_{\Sigma}\rangle d\sigma_{\Sigma},
\end{equation}
where:  $F$ belongs to the hyperplane $x_1=0$ of $\tor^*$,  a smooth function $V$ on $F$  is seen as a smooth function of the variables $(x_2, \ldots, x_n)$, the sum is taken over the facets $\Sigma$ of $F$ with inward normal $e_{\Sigma} \in \tor/\tor_F\cong \R^{n-1}= {\rm span}_{\R}\{e_2, \ldots, e_n\},$ 
and the induced measures $d\sigma_{\Sigma}$  are constructed from the label polytope $(F, {\bf L}_F)$ via \eqref{measure}. Thus,  integrating \eqref{basic} and using \eqref{byparts} gives
\begin{equation}\label{beta0}
\begin{split}
\Big(c+ \frac{6 \beta}{\alpha^2}\Big){\rm Vol}(F)  = &  -2\sum_{\Sigma \in \partial F} \int_{\Sigma}\big(H(e_1, e_{\Sigma})\big)_{,11} d\sigma_{\Sigma}  \\
                                                     &    -  \sum_{\Sigma \in \partial F} \int_{\Sigma} \sum_{i=2}^n \big(H(e_i, e_{\Sigma}))_{,1i} d\sigma_{\Sigma}.
\end{split}
\end{equation}
We recall that in \eqref{beta0},  $\alpha$ and $c$ have been already defined in terms of $(\Delta, {\bf L}, F)$, so in order to define $\beta$  it will be enough to show that  each of  the two sums at the right hand side of \eqref{beta0} can also be defined by  $(\Delta, {\bf L}, F)$.

\smallskip 
We first deal with the term $\int_{\Sigma}\big(H(e_1, e_{\Sigma})\big)_{,11} d\sigma_{\Sigma}$. 
Notice that if $\Sigma$ is a facet of $F$ which meets the chosen vertex ($=$ the origin), 
i.e. if $\Sigma$ belongs to $\{x_1=0, x_j=0\}, j>1$, 
then $\big((H_{1, e_{\Sigma}})_{,11}\big)_{| \Sigma} = (H_{1j,11})_{\{x_1=0, x_j=0\}}=0$  by the expansion \eqref{eqn_matrixO} of $H_{ij}$ near $\Sigma$. 
For a general facet $\Sigma$ of $F$,  we let $P\subset \partial \Delta$ be the unique other facet of $\Delta,$ such that $\Sigma = F\cap P$ 
and denote by $e_{F}, e_{P}$ the corresponding inward normals.  
Thus,  $\tor_{\Sigma} = {\rm span}_{\R}\{e_F, e_P\}$ is the annihilator of $T_p\Sigma \subset \tor^*$ (where $p$ in a interior point for $\Sigma$),  
equipped with a natural basis $\{e_F, e_{P}\}$.  
For any two vectors $e', e'' \in \tor$, the function  ${\bf H}^u(e', e'')$ is smooth on $\Delta$ 
and we denote by ${\rm Hess}\Big({\bf H}^u(e', e'')\Big)_p$ its Hessian at $p\in \Delta$, computed with respect to the affine structure of $\tor^*$. 
Thus, ${\rm Hess}\Big({\bf H}^u(e', e'')\Big)_p \in S^2(\tor)$ and with respect to the chosen basis we have
$$H_{ij,kr} (p) = \Big\langle e_k^*\otimes e_r^*, {\rm Hess}\big({\bf H}^u(e_i, e_j)\big)_p\Big\rangle.$$

Using the boundary conditions of Proposition~\ref{complete-first-order}, 
we notice that  for any $e\in \tor$,  $d{\bf H}^u(e_F, e)=0$ along $F$ and, hence along $\Sigma$. 
It thus follows that  for each interior point $p\in \Sigma$,  the symmetric bilinear form ${\rm Hess}\Big({\bf H}^u(e_F, e)\Big)_p$  degenerates on   $T_p\Sigma$, 
or in other words, for each  $p\in \Sigma$,  ${\rm Hess}\Big({\bf H}^u(e_F, e)\Big)_p$  has values in $\tor_{\Sigma}\otimes \tor_{\Sigma}$.
Using the basis $\{e_{F}, e_{P}\}$ of $\tor_{\Sigma}$, we have a natural decomposition   at each point of $\Sigma$:
\begin{equation*}
\begin{split}
{\rm Hess}\Big({\bf H}^u(e_F, e)\Big)  =  &\big({\bf H}^u(e_F, e)\big)_{, e_Fe_F} e_F\otimes e_F\\
                                          & + \big({\bf H}^u(e_F, e)\big)_{, e_Fe_P} (e_F\otimes e_P + e_P\otimes e_F) \\
                                          &+ \big({\bf H}^u(e_F, e)\big)_{, e_Pe_P} e_P\otimes e_P
\end{split}
\end{equation*} 
By choosing a vertex of $\Delta$ which belongs to $\Sigma$ and a basis as above,  and letting  $e=\sum_{i=1}^n a_i e_i$ the  coefficients above  become
\begin{equation*}
\begin{split}
\big({\bf H}^u(e_F, e)\big)_{, e_Fe_F} &= \sum_{i=1}^n a_i H_{1i, 11}, \\
 \big({\bf H}^u(e_F, e)\big)_{, e_Fe_P}  & = \sum_{i=1}^n a_iH_{1i,1j},  \\
\big({\bf H}^u(e_F, e)\big)_{, e_Pe_P}  &= \sum_{i=1}^n a_iH_{1i,jj}, 
\end{split}
\end{equation*}
where the index $j>1$ is determined by $\Sigma \subset \{x_1=0, x_j=0\}$. Using the boundary conditions of Proposition~\ref{complete-first-order},  
which are equivalently expressed by the form \eqref{eqn_matrixO} of ${\bf H}^u$  near $\Sigma$, we obtain that on $\Sigma$ 
\begin{equation}\label{hessian1}
\begin{split}
{\rm Hess}\big({\bf H}^u(e_F, e_F)\big) &= \frac{2}{\alpha} e_F \otimes e_F= \frac{2}{\alpha} e_1 \otimes e_1 \\
{\rm Hess}\big({\bf H}^u(e_F, e_P)\big) & =  0.
\end{split}
\end{equation}
Similarly, using the constancy of $d {\bf H}^u(e_P, e_P)$ along $\Sigma \subset P$ and \eqref{eqn_matrixO} with respect to a suitable basis, we also conclude that
\begin{equation}\label{hessian2}
{\rm Hess}\big({\bf H}^u(e_P, e_P)\big) =  \big({\bf H}^u(e_P, e_P)\big)_{, e_Pe_P}  e_P\otimes e_P.
\end{equation}
Turning back to the term $\int_{\Sigma}\big(H(e_1, e_{\Sigma})\big)_{,11} d\sigma_{\Sigma}$, 
we notice that the definition of the normal $e_{\Sigma}$ in the expression $H_{1e_{\Sigma}, 11}$ uses the initial basis $\{e_1, \ldots, e_n\}$. 
Indeed, decomposing
$$e_P= \sum_{i=1}^n c_i e_i, $$
we have that $e_{\Sigma}= e_{P} - c_1 e_1=e_{P}-c_1e_{F}$ where $c_1= c_1(\Sigma)=-e_P\wedge d\nu_{F}/d\mu$ is a constant determined 
by the polytope  $(\Delta, {\bf L}, F)$ and the facet $\Sigma$ of $F$. 
It thus follows  from \eqref{hessian1} that  on $\Sigma$
\begin{equation*}
\begin{split}
H(e_1,e_{\Sigma})_{, 11} &= \Big\langle e_1^*\otimes e_1^*,  \big({\bf H}^u(e_F, e_P)- c_1 {\bf H}^u(e_F, e_F)\big)\Big\rangle\\
                         &=- \frac{2c_1}{\alpha},
\end{split}
\end{equation*}
and therefore
\begin{equation}\label{first-term}
\int_{\Sigma}\big(H(e_1, e_{\Sigma})\big)_{,11} d\sigma_{\Sigma} =-\frac{2c_1(\Sigma)}{\alpha}  {\rm Vol}(\Sigma).
\end{equation}
We have thus shown that the first sum in \eqref{beta0} only depends on $(\Delta, F, {\bf L})$.

\smallskip
We now deal with the terms $\int_{\Sigma} \sum_{i=2}^n \big(H(e_i, e_{\Sigma}))_{,1i} d\sigma_{\Sigma}$ in the second sum of \eqref{beta0}.  
First of all, notice that on $F$,  the expression 
$$\sum_{i=2}^n \big(H(e_i, e_{\Sigma}))_{,1i}  = \sum_{i=2}^n \Big\langle  e_i^*\otimes e_1^*, {\rm Hess} \big({\bf H}^u(e_i, e_{\Sigma})\big) \Big\rangle$$
does not change if we replace the initial basis $\{e_1=e_F, e_2, \ldots, e_n\}$ of $\tor$ 
with  a basis of the form $\{\bar e_1=e_F, \bar e_2, \ldots, \bar e_n\}$ with $\bar e_j \in {\rm span}_{\R}\{e_2, \ldots, e_n\} = \R^{n-1}$ for $j =2, \ldots, n$. 
We can thus assume that, on a given $\Sigma$, we have chosen the basis with $e_2=e_{\Sigma} \in \R^{n-1}$, 
and use then the integration by parts formula \eqref{byparts}  to write
\begin{equation}\label{second-term} 
 \int_{\Sigma} \sum_{i=2}^n \big(H(e_i, e_{\Sigma}))_{,1i} d\sigma_{\Sigma}
   = \int_{\Sigma}  \big(H(e_{2}, e_{2}))_{,12} d\sigma_{\Sigma} 
     - \sum_{ {\bf f} \in \partial \Sigma} \int_{\bf f} H(e_{\bf f}, e_{\Sigma})_{,1} d\sigma_{\bf f},
\end{equation}
where the  sum is over the facets ${\bf f}$ of $\Sigma$ (taken to be zero if $n=2$), $e_{\bf f}$ is the induced inward normal of $\bf f$, 
seen as an affine hyperplane of the affine space supporting $\Sigma$, and $d\sigma_{\bf f}$ is the corresponding induced measure on ${\bf f} \subset \Sigma$.  
With these choices, we have
\begin{equation*}
\begin{split}
\big(H(e_2, e_2)\big)_{,12}  &= \Big\langle e_1^*\otimes e_2^*, {\rm Hess}\big({\bf H}^u(e_{\Sigma}, e_{\Sigma})\big)\Big\rangle \\
                             &= \Big\langle e_1^*\otimes e_2^*, {\rm Hess}\big({\bf H}^u(e_{P}- c_1e_F, e_{P}- c_1e_F)\big)\Big\rangle  \\
                             &  = \Big\langle e_1^*\otimes e_2^*, {\rm Hess}\big({\bf H}^u(e_{P}, e_{P})\big)\Big\rangle  \\
                             &= {\bf H}^u(e_P, e_P)_{, e_P,e_P} \big\langle e_1^*\otimes e_2^*, (c_1 e_1 + e_2)\otimes (c_1e_1 + e_2)\big\rangle \\
                             &= c_1  {\bf H}^u(e_P, e_P)_{, e_Pe_P},  
\end{split}
\end{equation*}
for the same  constant $c_1=c_1(\Sigma)$ as above. 
In order to compute the (base independent) quantity $\int_{\Sigma} {\bf H}^u(e_P, e_P)_{, e_Pe_P}  d\sigma_{\Sigma}$, 
we are going to re-introduce a basis $\{e_1=e_F, e_2=e_P, e_3, \ldots, e_n \}$ with respect to a vertex of $v\in \Delta$.  
In this basis, ${\bf H}^u(e_P, e_P)_{, e_Pe_P} =H_{22,22}$. A computation along the lines of  \eqref{trick} yields
\begin{equation*}
\begin{split}
\int_{\Sigma} s_{(F, {\bf L}_F)} d\sigma_{\Sigma} &= \int_{\Sigma} -\Big(\sum_{i,j=2}^n H_{ij,ij}\Big) d\sigma_{\Sigma}\\
                                                  &= \int_{\Sigma} \Big({\rm Scal}\big((g_u)_{|\Sigma}\big) - H_{22,22} - \sum_{j=3}^n H_{2j,2j}\Big) d\sigma_\Sigma \\
                                                  &= \int_{\Sigma} \Big(s_{(\Sigma, {\bf L}_{\Sigma})} -H_{22,22}\Big)d\sigma_{\Sigma}   
                                                     + \sum_{{\bf f} \subset \partial \Sigma} \int_{\bf f} H(e_2,e_{\bf f})_{, 2} d\sigma_{\bf f},
\end{split}
\end{equation*}           
where ${\rm Scal}\big((g_u)_{|\Sigma}\big) := - \sum_{i,j=2}^n H_{ij,ij}$ 
is the scalar curvature of the almost-K\"ahler metric $(g_u)_{|\Sigma}$ induced via ${\bf H}^u$ on the  preimage of $\Sigma$,  
and for passing from the second line to the third we have used that 
$\int_{\Sigma}  {\rm Scal}\big((g_u)_{|\Sigma}\big) d\sigma_{\Sigma}= \int_{\Sigma} s_{(\Sigma, {\bf L}_{\Sigma})} d\sigma_{\Sigma}$ 
(see  \cite[Lem.~3.3.5]{Do2}) and  \eqref{byparts} applied to $(\Sigma, {\bf L}_{\Sigma})$. 
Note that in the last term (which is considered trivially $0$ when $n=2$), 
the  sum is over the facets ${\bf f}$ of $\Sigma$, and $e_{\bf f}$ denotes the corresponding inward normal of ${\bf f}$ 
(when considered as an affine hyperplane of the subspace $\{x_1=0, x_2=0\}$).
We thus have
 $$\int_{\Sigma} {\bf H}^u(e_P, e_P)_{, e_Pe_P}  d\sigma_{\Sigma} 
    = \int_{\Sigma} \big(s_{(\Sigma, {\bf L}_{\Sigma})} - s_{(F, {\bf L}_F)}\big) d\sigma_{\Sigma} 
      + \sum_{{\bf f} \subset \partial \Sigma}\int_{\bf f} H(e_{\bf f}, e_2)_{,2}d\sigma_{\bf f},$$
in an $(F,\Sigma)$-compatible basis with $e_1=e_F, e_2=e_P$. 
Notice that in any such a basis,   we have $e_{\bf f}= e_Q -  c_2 e_F - c_3 e_P= e_Q -c_2e_1-c_3e_2,$ 
where $e_Q$ is the normal of the unique facet $Q \subset \Delta$, such that $F\cap P \cap Q= {\bf f}$. 
Here, the constants $c_2=c_2(\Sigma, {\bf f})= -e_Q\wedge d\nu_{F}/(d\mu)$ and $c_3=c_3(\Sigma, {\bf f})= -(e_Q \wedge d\nu_P)/(d\mu)$ 
are determined in terms of $(\Delta, {\bf L})$. 
In particular, we have on ${\bf f}$
\begin{equation*}
\begin{split}
H(e_{\bf f}, e_2)_{,2}  &= \Big\langle e_2^*, d {\bf H}^u(e_P, e_Q - c_2e_F-c_3 e_P) \Big\rangle \\
                        &=  -c_3 \Big\langle e_2^*,  d {\bf H}^u(e_P,  e_P) \Big\rangle \\
                        &= - 2c_3 \langle e_2^*,  e_P  \rangle  = - 2c_3\langle e_2^*, e-2\rangle = -2c_3,
\end{split}
\end{equation*}
 where we have used the first order boundary conditions  along ${\bf f}$ (see Proposition~\ref{complete-first-order}):
 \begin{equation}\label{first-order}
 \begin{split}
  \big(d{\bf H}^u(e_F, e)\big)_{|\bf f}  &= 0 ,  \  \forall e\in \tor,   \ \  \big(d{\bf H}^u(e_P, e_Q)\big)_{|{\bf f}} =0     \\
 \ \  \  \big(d{\bf H}^u(e_P, e_P)\big)_{|{\bf f}} &= 2e_P,   \ \  \ \ \ \  \ \  \big(d{\bf H}^u(e_Q, e_Q)\big)_{|{\bf f}} = 2 e_Q.           
\end{split}
\end{equation}
To summarize, we have shown that 
\begin{equation}\label{second-term-1}
\begin{split}
 \int_{\Sigma}\big(H(e_2,e_2)\big)_{,12}  d\sigma_{\Sigma} =&  c_1(\Sigma) \Big(\int_{\Sigma}  \big( s_{(\Sigma, {\bf L}_{\Sigma})} 
                                                               -s_{(F, {\bf L}_{F})}\big) d\sigma_{\Sigma} \\
                                                            & \ \ \ \  \ \ \ \  - 2\sum_{{\bf f} \subset \partial \Sigma} c_3(\Sigma, {\bf f}){\rm Vol}({\bf f}) \Big). 
\end{split}
\end{equation}

\smallskip
Finally, we deal with the terms $\int_{\bf f} H(e_{\bf f}, e_{\Sigma})_{,1} d \sigma_{\bf f}$  in \eqref{second-term} 
(where we recall ${\bf f} \subset \Sigma \subset F$ is sequence of co-dimension one sub-faces). 
Using \eqref{first-order} again, we have along ${\bf f}$
\begin{equation*}
\begin{split}
H(e_{\bf f}, e_{\Sigma})_{,1} &= \Big\langle e_1^*, d{\bf H}^u(e_Q-c_2e_F- c_3e_P, e_F - c_1 e_P) \Big\rangle \\
                              &= 2c_1c_3 \langle e_1^*, e_P \rangle \\
                              &= 2c_1c_3 \langle e_1^*,  e_2 + c_1 e_1 \rangle = 2c_1^2c_3,                                    
\end{split}
\end{equation*}
where for passing from the second line to the third 
we have used that we choose in \eqref{second-term} a base with $e_1=e_F, e_2=e_{\Sigma}=e_P - c_1 e_F=e_P - c_1e_1$. 
It follows that
\begin{equation}\label{second-term-2}
\sum_{{\bf f} \subset \partial \Sigma} H(e_{\bf f}, e_{\Sigma})_{,1} d\sigma_{\bf f}
     = \big(c_1(\Sigma)\big)^2\Big( \sum_{{\bf f} \subset \Sigma}c_3(\Sigma, {\bf f})\Big).
\end{equation}
Substituting \eqref{second-term-1} and \eqref{second-term-2} back in \eqref{second-term}, 
and \eqref{first-term}, \eqref{second-term} and \eqref{alpha} back in \eqref{beta0}, we obtain an expression for $\beta$ in terms of $(\Delta, {\bf L}, F)$.  \end{proof}

\begin{rem} \begin{enumerate}
\item Theorem~\ref{thm:criterion} and Proposition~\ref{p:alpha-beta} extend without difficulty to the case when $F=F_{1} \cup \cdots \cup F_{k}$ is a union of {\it non-intersecting} facets,  i.e. $Z$ is a smooth toric divisor.
In general,  it is natural to extend Definition~\ref{toric-poincare} by introducing a pair of real numbers $(\alpha_i, \beta_i)$ for each facet $F_i \subset F$ and,  for each face ${\bf f}\subset F$, one should  require the smoothness and  convexity over the relative interior of ${\bf f}$ of  the function $$u_{\bf f}:=u + \sum_{F_i \in F : {\bf f} \subset F_i}  (\alpha_i + \beta_i L_i)\log L_i.$$  It will be interesting to see wether or not the above statements hold true for a general toric divisor as above,   with respect to such spaces of symplectic potentials  (compare with Conjecture~\ref{conjecture1} above). 

\item The explicit examples in the next section suggest 
that  the complete Donaldson metrics will have, more generally, symplectic potentials $u$  with the asymptotic 
$$u = \frac{1}{2} \Big(\sum_{F_k\notin F}  L_k(x)\log L_k(x) + \sum_{F_i \in F}  f_{i}(x) \log L_{i}(x) \Big)+ {\rm smooth \ terms},$$
where,  for any facet $F_i\in F$, $f_{i}$ is  some affine function.

\item As we noticed in the course of the proof of Proposition~\ref{p:alpha-beta}, the situation simplifies when $n=2$.  
In fact, one can then explicitly determine $(\alpha, \beta)$ as follows: 
suppose  (without loss of generality) that $(\Delta, {\bf L}, F)$  is such that $\Delta \subset \{(x_1, x_2): x_1 \ge 0, x_2 \ge 0,\ \ell- x_2 - \lambda x_1 \ge 0\}$, $F$ corresponds to the affine line $x_1=0$ whereas  the two adjacent facets of $\Delta$  to $F$ are defined by the affine lines $ x_2=0$ and $\ell- x_2 - \lambda x_1 =0$.  Suppose,  furthermore,  that the extremal affine function of $(\Delta, {\bf L}, F)$ is $s_{(\Delta, {\bf L}, F)} = a_0 + a_1 x_1 + a_2 x_2$. Then the real parameters $(\alpha, \beta)$ of Proposition~\ref{p:alpha-beta} are given by
$$\alpha = \frac{2\ell}{4-a_0\ell}, \ \ \ \beta= \frac{\alpha^2}{6}\Big(a_1 + \frac{2\lambda a_0}{\ell} - \frac{12\lambda}{\ell^2}\Big).$$
\end{enumerate}

\end{rem}

With this motivation in mind, we will show in the next section,  by using the explicit constructions of \cite{ACG1,ACG2},  
that Conjecture~\ref{conjecture2} is true for $X=\C P^2, \C P^1\times \C P^1$ 
or the $m$-th Hirzebruch complex surface $\F_{m} = \mathbb \Proj(\cO \oplus \cO(m))\to \C P^1, m\ge 1$.

\section{Explicit Donaldson metrics on quadrilaterals}\label{s:examples}

By \cite[Thm.~1 and Rem.~7]{ACG2},  any {\it stable} compact convex quadrilateral $(\Delta, {\bf L}, F)$  in $\R^2$ admits a (canonical) Donaldson metric, 
which is explicit and ambitoric. K. Dixon~\cite{Dixon}  showed that  when $(\Delta, {\bf L}, F)$ corresponds to a compact toric complex orbi-surface $X$ with a divisor $Z$,  
the metric is complete on $X\setminus Z$.  
In other words,  Conjecture~\ref{conjecture2} holds true for compact toric surfaces whose momentum polytope is a quadrilateral.  
On the other hand, a detailed study of the stability of the triples $(\Delta, {\bf L}, F)$ was carried out by the third named author in \cite{lars}. In the next subsections we shall combine  these results in order to obtain a complete picture 
in the case when $X=\C P^2,  \C P^1 \times \C P^1$ or $\mathbb{P}(\cO \oplus \cO(k))\to \C P^1$, 
i.e. $(\Delta, {\bf L})$ is a Delzant triangle, parallelogram or a  trapezoid.

\subsection{Ambitoric structures} Here we briefly review the explicit construction of extremal toric metrics for $n=2$ via the ambitoric ansatz of \cite{ACG1}.
\begin{defn} An \emph{ambik\"ahler structure} on a real $4$-manifold or
orbifold $M$ consists of a pair of K\"ahler metrics $(g_+, J_+, \omega_+)$ and
$(g_-, J_-, \omega_-)$ such that
\begin{bulletlist}
\item $g_+$ and $g_-$ induce the same conformal structure (i.e., $g_- =
f^2g_+$ for a positive function $f$ on $M$);
\item $J_+$ and $J_-$ have opposite orientations (equivalently the
volume elements $\frac1{2}\omega_+\wedge\omega_+$ and
$\frac1{2}\omega_-\wedge\omega_-$ on $M$ have opposite signs).
\end{bulletlist}
The structure is said to be \emph{ambitoric} if in addition
\begin{bulletlist}
\item there is a $2$-dimensional subspace $\tor$ of vector fields on $M$,
linearly independent on a dense open set, whose elements are hamiltonian and
Poisson-commuting Killing vector fields with respect to both $(g_+, \omega_+)$
and $(g_-, \omega_-)$.
\end{bulletlist}
\end{defn}
Thus $M$ has a pair of conformally equivalent but oppositely oriented K\"ahler
metrics, invariant under a local $2$-torus action, and both locally toric with
respect to that action.
There are three classes of examples of ambitoric structures.

\subsection{Toric products} Let $(\Sigma_1,g_1,J_1,\omega_1)$ and
$(\Sigma_2,g_2,J_2,\omega_2)$ be (locally) toric K\"ahler
manifolds or orbifolds of real dimension 2, with hamiltonian Killing vector fields $K_1$ and
$K_2$. Then $M=\Sigma_1\times \Sigma_2$ is ambitoric, with $g_\pm=g_1\oplus
g_2$, $J_\pm=J_1\oplus (\pm J_2)$, $\omega_\pm=\omega_1\oplus (\pm\omega_2)$
and $\tor$ spanned by $K_1$ and $K_2$.  
The metric $g_+$ is extremal (resp. CSCK) iff $g_{-}$ is extremal (resp. CSCK) iff both $g_1$ and $g_2$ are extremal (resp. CSCK). 
Writing $(\Sigma_1,g_1)$ and  $(\Sigma_2, g_2)$ as toric Riemann surfaces
$$g_1 =  \frac{dx^2}{A(x)} + A(x)dt_1^2; \ \ g_2=\frac{dy^2}{B(y)} + B(y) dt_2^2,$$
for positive functions $A,B$ of one variable, and momentum/angular  coordinates
$$x_1=x, \ x_2=y,  \ t_1,  \ t_2$$
the extremal metrics are given by taking $A$ and $B$ to be polynomials of degrees $\le 3$.  
In this case,  we obtain solutions to Abreu's equation on labelled parallelograms $(\Delta, {\bf L}, F)$ (which are affine equivalent to a square) by taking 
\begin{equation}\label{solution-parallelogram}
{\bf H}^{A,B} ={\rm diag} (A(x), B(y))
\end{equation}
for $A$ and $B$ polynomials of degree $\le 3$ 
and noting that the positivity and boundary conditions of Definition~\ref{d:guillemin-boundary} reduce to $A>0$ on $(\al_\1, \al_\2)$, $B>0$ on $(\be_\1, \be_\2)$  and
\begin{equation}\label{parallelogram-boundary}
A(\al_k)=0=B(\be_k)=0, \quad A'(\al_k)=-2r_{\al,k}, \quad B'(\be_k)
=2r_{\be,k} \quad (k=\1,\2),
\end{equation}
where $r_{\al,\1} \le 0 \le r_{\al,\2}, \  r_{\be,\1} \ge  0 \ge r_{\be,\2}$ are determined 
by the choice of inward normals $e_{\al, k}=\frac{1}{r_{\al, k}}(-1,0)$  (resp. $e_{\be, k}= \frac{1}{r_{\be, k}}(0,1)$) 
if the facet $F_{\al,k}$ (resp. $F_{\be,k}$) defined by $x=\al_k$ (resp. $y=\be_k$) does not belong to $F$,  and $r_{\al,k}=0$ (resp. $r_{\be,k}=0$) otherwise.  
The above boundary conditions can be solved for polynomials of degree $3$, $A(x)$ and $B(y)$,  
if and only if $|r_{\al,\1}| + |r_{\al, \2}|>0$ and $|r_{\beta,\1}| + |r_{\be,\2}|>0$, i.e. iff no two opposite sides of $\Delta$ belong to $F$:  
in this case, the positivity of $A$ and $B$ automatically follows from the boundary conditions. 
On the other hand,  when two opposite sides of $\Delta$ belong to $F$ there is no solution to \eqref{abreu1} verifying the positivity condition. 
Indeed,  if  $r_{\al,\1}=r_{\al,\2}=0$ say,  then  ${\bf H}^{A,B} ={\rm diag} (A(x), B(y))$  
with $A\equiv 0$ and $B(y)$ a polynomial of degree $\le 3$ determined from \eqref{parallelogram-boundary}. 
This  provides a formal solution of \eqref{abreu1}. 
The latter can be used (by using integration by parts, as in \cite{Eveline,ACG2}) to compute that $\cL_{\Delta,u,F}(f_{\al})=0$ 
for any simple crease function $f_{\al}$ with crease at $x=\alpha$ ($\al \in (\al_\1,\al_\2)$), showing that $(\Delta, {\bf L}, F)$ is not stable in this case.

Whenever it exists, the solution $u_{A,B}$ is determined from \eqref{solution-parallelogram} by the formula
$$u_{A,B} = \int^x\Big(\int^s \frac{dt}{A(t)}\Big)ds + \int^y \Big(\int^s \frac{dt}{B(t)}\Big)ds, $$
which leads to the intrinsic expression
\begin{equation}\label{U-parallelogram}
u_{A,B} = \frac{1}{2}\Big(\sum_{F_j \in \partial \Delta} L_j \log L_j  -  \sum_{F_k \in F} a_k \log L^c_k \Big).
\end{equation}
where, for each facet  $F_j \in \partial \Delta \setminus F$,   
$L_j(x)= \langle e_j, x\rangle + \lambda_j$ is the corresponding label from ${\bf L}$ and,  for each $F_k \in F$,   
we define the label $L^c_k(x) := \langle e_k,x \rangle+ \lambda_k$  by requiring that $e_k:=-e_{\tilde k}= -dL_{\tilde k}$  
where $L_{\tilde k}$ is the label form ${\bf L}$ of the {\it opposite} side $F_{\tilde k}$ to $F_k$ 
(by the discussion above, $F_{\tilde k} \in \partial \Delta \setminus F$)  and  let $a_k  := L_k + L_{\tilde k}>0$ be a real constant.

We notice that when the solution exists, the degree $\le 3$ polynomials  $A(x)$ and $B(y)$ must satisfy $A''(x)=A''(\al_k) >0$ (resp. $B''(y)=B''(\be_k)>0$) on facets in $F$. 
The formula for the scalar curvature 
$$s_+ = - (A''(x) + B''(y))$$
then confirms that $s_{(\Delta, {\bf L}, F)} - s_{(F, {\bf L}_F, \check{F})}$ restricts to $F$ as a negative constant.

We conclude that 
\begin{prop}\label{parallelogram}  
Let $(\Delta,{\bf L},F)$ be a labelled parallelogram in $\R^2$. 
Then  the Abreu equation \eqref{abreu1} admits a solution in $\cS(\Delta, {\bf L}, F)$  iff $F$ doesn't contain opposite sides. 
In this case,   there exists a solution explicitly given by \eqref{U-parallelogram}. 
\end{prop}

Turning to the compact smooth case, 
there exists only one compact complex toric surface whose  Delzant polytopes are parallelograms, 
namely $X= \C P^1 \times \C P^1$. 
The result above trivially produces products of a cusp metric on $\C P^1\setminus \{{\rm pt} \}$ with a Fubini-Study metric on another copy of $\C P^1$, 
or the product of two cusp metrics on $(\C P^1\setminus \{{\rm pt} \}) \times (\C P^1\setminus \{{\rm pt} \})$, 
according to  whether $Z \subset \C P^1\times \C P^1$ is a copy of $\C P^1$ (i.e. $F$ is a one facet) 
or is the union of two copies of $\C P^1$ (i.e. $F$ consists of two adjacent facets). 
We thus can conclude that 

\begin{cor}\label{c:product} 
Let  $X= \C P^1 \times \C P^1$ endowed with the product of  circle actions on each factor,  
and $Z$ be either $\C P^1 \times \{p_2\}$ or $(\C P^1 \times \{p_2\}) \cup (\{p_1\} \times \C P^1)$ where $p_1$ and $p_2$ are fixed points for the $S^1$ actions on each factor. 
Then, in each K\"ahler class of $\C P^1\times \C P^1$, there exists a complete extremal Donaldson metric on $\C P^1 \times \C P^1 \setminus Z$ which is of Poincar\'e type. 

If $Z$  contains $(\C P^1 \times \{p_2'\}) \cup (\C P^1 \times \{p_2''\})$ where $p_2'$ are $p_2''$ are the two distinct fixed points for the $S^1$ action, 
then $(X, Z)$ is $K$-unstable, and admits no extremal Donaldson metric at all.
\end{cor}

\subsection{Toric Calabi type metrics}\label{s:calabi} 
The construction in this section is not new, see e.g. \cite{hwang-singer}. 
For the sake of completeness, and to  make the link with toric geometry more explicit, we follow the formalism from \cite{ACG1}.

Let $(\Sigma,g,J,\omega)$ be a toric
real $2$-dimensional K\"ahler manifold with hamiltonian Killing vector
field $V$ (with momentum $y$).  Let $\pi\colon P\to \Sigma$ be a circle bundle with connection
$\theta$ and curvature $d\theta=\pi^*\omega_{\Sigma}$, and  $A(x)$ be a positive
function defined on an open interval $I \subset \R^+$. Then $M=P\times I$ is ambitoric, with
\begin{align*}
g_\pm &= x^{\pm1}
\Bigl(g_{\Sigma}+\frac{dx^2}{A(x)}+ \frac{A(x)}{x^2}\theta^2\Bigr), \ d\theta = \omega_{\Sigma},\\
\omega_\pm &= x^{\pm1}\bigl( \omega_{\Sigma} \pm x^{-1} dx\wedge\theta\bigr),\qquad
J_\pm (x dx) = \pm A(x) \theta,
\end{align*}
and the local torus action spanned by  the generator $K$ of the circle
action on $P$ and the lift $\tilde V = V^H + y K$ of the hamiltonian Killing field of $(\Sigma, g_{\Sigma}, \omega_{\Sigma})$ to $M$. 
Here,  $x\colon M\to \R^+$ is the projection onto $I\subset \R^+$. It is easily seen that $g_+$ is extremal (resp. CSCK) 
iff $g_-$ is extremal iff $(\Sigma,g)$ has constant Gauss curvature $\kappa$ 
and $A(x)$ is a polynomial of degree $\le 4$ with coefficient of $x^2$ equal to $\kappa$. Because of this equivalence, we shall focus on $(g_+,\omega_+)$, say.

Writing the toric metric $(g_{\Sigma}, \omega_{\Sigma})$ in momentum/angle coordinates as
\begin{equation}\label{FS}
g_{\Sigma}= \frac{{dy}^2}{B(y)} + B(y) dt_2^2, \ \omega_{\Sigma}  = dy \wedge dt_2,
\end{equation}
for a positive function $B(y)$,  the K\"ahler metric $(g_+, \omega_+)$ becomes (see \cite{Eveline})
\begin{equation}\label{calabi-type}
\begin{split}
g_+ &= x \frac{dx^2}{A(x)} + x \frac{dy^2}{B(y)} + \frac{A(x)}{x}(dt_1 + y dt_2)^2 + x B(y) dt_2^2 , \\
\omega_+ & = dx \wedge dt_1 + d(xy) \wedge dt_2= dx_1\wedge dt_1 + dx_2 \wedge dt_2,
\end{split}
\end{equation}
with
\begin{equation}\label{calabi-moments}
(x_1, x_2)=(x, xy)
\end{equation} 
being the momentum coordinates and $(t_1,t_2)$ the angular coordinates. The corresponding symplectic potential is then
\begin{equation}\label{potential-calabi}
u_{A,B}(x,y)= x\int^y \Big(\int^s \frac{dt}{B(t)}\Big)ds  + \int^x \Big(\int^s \frac{tdt}{A(t)}\Big) ds.
\end{equation}

In order to obtain  functions in $\cS(\Delta,{\bf L},F)$ for some compact convex labelled polytope $(\Delta, {\bf L}, F)$, we  fix the data of real numbers
\begin{equation}\label{calabi-data}
0 \le \be_\1<\be_\2, \ 0<\al_\1 < \al_\2, \ r_{\al,\1} \le 0 \le r_{\al,\2}, \  r_{\be,\1} \ge  0 \ge r_{\be,\2}
\end{equation}
and impose the following positivity and boundary conditions  on the smooth functions of one variable $A(x)$ and $B(y)$
\begin{equation}\label{positivity-trapez}
A(x)>0 \ {\rm on} \ (\al_\1, \al_\2) \ {\rm and} \ B(y)>0 \ {\rm on} \ (\be_\1, \be_\2),
\end{equation}
\begin{equation}\label{boundary-trapez}
A(\al_k)=0,  A'(\al_k)=-2r_{\al,k}, \  B(\be_k)=0, B'(\be_k)= -2r_{\be,k}, \quad (k=\1,\2), 
\end{equation}
Note that the line $\{x=\alpha\}$ transforms in the $(x_1,x_2)$-coordinates \eqref{calabi-moments} 
to the affine line $\ell_{\alpha} = \{(\alpha, x_2)\}$ with normal $p_{\alpha}=(\alpha,0)$ 
and $y=\beta$ ($\beta>0$) to the affine line $\ell_{\beta}=\{(x_1,\beta x_1) \}$ with normal $p_{\beta}=(-\beta, 1)$. 
Thus, the image of $D=[\al_\1, \al_\2]\times [\be_\1, \be_\2]$ under \eqref{calabi-moments} is  a trapezoid $\Delta$ 
with facets $F_{\al,k}, F_{\be,k}$ determined by the lines $\ell_{\alpha_k}, \ell_{\be_k}$, and the inverse Hessian ${\bf H}^{A,B}$ of $u_{A, B}$, is given by
\begin{equation}\label{solution-trapezoid} 
{\bf H}^{A,B} = \frac{1}{x}\Big(\begin{array}{cc} A(x) & yA(x) \\  yA(x) & x^2B(y) + y^2A(x) \end{array} \Big).
\end{equation}
We write for the normals $e_{\al, k}= p_{\al_k}/ r_{\al,k}, e_{\be,k}= p_{\be_k}/r_{\be,k}$. Then, ${\bf H}^{A,B}$ satisfies the smoothness, 
positivity and boundary conditions \eqref{boundary-divisor-0}-\eqref{boundary-divisor-1} iff  $r_{\al,k}=0$ 
(resp. $r_{\be,k}=0$) on a facet $F_{\al,k}\in F$ (resp. $F_{\be,k} \in F$). 

Conversely, by \cite[Lem.~4.7]{Eveline}, if  $(\Delta,{\bf L}, F)$ is a labelled trapezoid,  there exist real numbers $\al_k, r_{\al,k}, \kappa$ $(k=\1,\2)$, subject
to the inequalities \eqref{calabi-data}, such that  $\Delta$ the image of $D=[\al_\1, \al_\2]\times [\be_\1, \be_\2]$ under $(\mu_1, \mu_2)$,  
and $r_{\al,k}$ are determined from the normals ${\bf L}$ and $F$ as explained above. 
It is easily seen~\cite[Prop.~4.12]{Eveline} that \eqref{solution-trapezoid} satisfies \eqref{abreu1} 
if and only if $A(x)$ is a polynomial of degree $\le 4$, $B(y)$ is a polynomial of degree $\le 2$,  which satisfy
\begin{equation}\label{calabi-relation}
A''(0) + B''(0)=0.
\end{equation}
For such polynomials to satisfy \eqref{boundary-trapez}, one must have $r_{\be,\1}= -r_{\be,\2}= r \ge 0$. 
This is also a sufficient condition to determine the polynomials $A(x)$ and $B(y)$ from \eqref{boundary-trapez},  
subject to the relation \eqref{calabi-relation}. In particular,  $B(y)= \frac{r}{(\be_2-\be_1)}(y-\be_1)(\be_\2-y)$, showing that the positivity
conditions \eqref{positivity-trapez} imply $r>0$, i.e. $r_{\be,k} \neq 0$. 
This also implies positivity for $A(x)$ on $(\al_\1, \al_\2)$:  
otherwise $A(x)$ will  have all of its roots between  $[\al_\1, \al_\2]$ and, by the boundary conditions \eqref{boundary-trapez},  
it must satisfy $\lim_{x\to \infty} A(x)= -\infty$. 
The latter  contradicts $A''(0)= -B''(0)>0$. 
The corresponding K\"ahler metric has scalar curvature
\begin{equation}\label{scal-calabi}
s_+ = -\frac{A''(x) + B''(y)}{x},
\end{equation}
showing that the extremal affine function $s_{(\Delta, {\bf L}, F)}$ determines an affine line parallel to $F_{\alpha, \1}$ and $F_{\al, \2}$. 
Conversely, the proofs  of  \cite[Lem.~4.2, Thm.~1.4]{Eveline} show that if  $(\Delta, {\bf L}, F)$ is a labelled trapezoid $(\Delta,{\bf L}, F)$ 
which is not a parallelogram,  and  the extremal affine function $s_{(\Delta, {\bf L}, F)}$ is  constant on  each of the  pair of  parallel facets of $\Delta$,  
then one can associate to $(\Delta, {\bf L}, F)$ data \eqref{calabi-data} satisfying the relation $r_{\be,\1}=-r_{\be,\2}=r \ge 0$.  
The case $r=0$ (i.e. when two opposite non-parallel facets of $\Delta$ belong to $F$) implies $B(y)\equiv 0$. 
As observed in \cite{Sz}, and similarly to the case of a parallelogram, this contradicts the stability of $(\Delta, {\bf L}, F)$. 
Indeed,   substituting  in \eqref{solution-trapezoid}, 
we still obtain a smooth matrix ${\bf H}^{A,B}$ on $\Delta$ verifying \eqref{abreu1} and the boundary conditions of Definition~\ref{d:guillemin-boundary}. 
This can be used to compute $\cL_{(\Delta,u,F)} (f_{\alpha})$ for a simple crease function $f_{\alpha}$ with crease on the line $\ell_{\alpha}=\{x=\alpha\}$:  
integration by parts  reduces to an integral over the crease of  the quantity ${\bf H}^{A,0}(p_{\al}, p_{\al})=0$, 
showing that  $\cL_{(\Delta,{\bf L}, F)} (f_{\alpha})=0$,  i.e. $(\Delta, {\bf L}, F)$ is not stable. 
We summarize the discussion in the following:
\begin{prop}\label{trapezoid} \cite{Eveline, Sz}
Let $(\Delta,{\bf L}, F)$ be a labelled trapezoid in $\R^2$ which is not a parallelogram. 
Suppose that the corresponding extremal affine function  $s_{(\Delta, {\bf L} ,F)}$ is constant on each of the  pair of parallel facets of $\Delta$. 
Then $(\Delta, {\bf L}, F)$ admits a solution to \eqref{abreu1} in $\cS(\Delta, {\bf L}, F)$ if and only if  $(\Delta, F)$ is stable, 
if and only if $F$ is one or the union of $2$ of the parallel facets of $\Delta$. 
In these cases, the solution is of Calabi-type, i.e. given by \eqref{potential-calabi} for polynomials $A(x)$ and $B(y)$  as described above.
\end{prop}
In order to derive further geometric applications, 
we use  \cite[Cor.~1.6]{Eveline} which  identifies  the choice of labels ${\bf L}$ of a given trapezoid $\Delta$ 
for which $s_{(\Delta,  {\bf L},F)}$ is  constant on each of  the pair of parallel facets with one single linear constraint on the pair of inward normals to non-parallel facets. 
Up to an overall positive rescaling of ${\bf L}$, this fixes the choice of these normals,  
but leaves no constraint on the pair of normals corresponding to the parallel opposite facets. 
In our notation, this corresponds to fixing the boundary condition for $B(y)$ and allowing  $r_{\al,\1} \le 0 \le r_{\al,\2}$ to be arbitrary real numbers. 
It thus follows that if $(\Delta,{\bf L})$ is a labelled trapezoid 
for which the corresponding extremal affine function $s_{(\Delta,{\bf L})}$ is parallel to the pair of parallel facets, 
then, by taking $F$ to be either one or two of the parallel facets of $\Delta$, 
the extremal affine function $s_{(\Delta,{\bf L},F)}$ must also be parallel to the pair of parallel facets. 
We now apply this observation to Delzant trapezoids $(\Delta, {\bf L})$.

The compact toric  complex surfaces $X$ for which the Delzant polytopes are trapezoids  (but not parallelograms) are the Hirzebruch surfaces 
${\mathbb F}_m= \Proj(\cO \oplus \cO(m)) \to \C P^1, \ m \ge 1$. 
Calabi~\cite{calabi} has shown that these surfaces admit extremal K\"ahler metric of Calabi-type in each K\"ahler class. 
In particular,  the extremal affine functions are always  constant on  the pair of parallel facets of the corresponding Delzant polytopes. 
Thus, Proposition~\ref{trapezoid}  yields the following natural extension of Calabi's result.

\begin{cor}\label{Hirzebruch} 
 Let  $X\cong \F_m=\Proj(\cO \oplus \cO(m))\to \C P^1$ be the $m$-th Hirzebruch surface 
 and $Z\subset X$ the  divisor consisting of either  the zero section $S_0$, the infinity section $S_{\infty}$ or the union of both. 
 Then $X\setminus Z$ admits a complete extremal Donaldson K\"ahler metric in each K\"ahler class of $\F_m$. Furthermore, this metric is of Poincar\'e type.
\end{cor}
\begin{proof} 
The only additional clarification we need to supply is whether the explicit extremal Calabi type metrics are of Poincar\'e type. 
This follows  from the expression  \eqref{solution-trapezoid}, noting that  the only the facets $F_{\al, k}$ are in $F$, 
and  on such a facet (having normal vector ($\al_k, 0)$),  
the boundary conditions of Proposition~\ref{complete-first-order}  reduce to $A(\al_k)= A'(\al_k)=0$ and $A'''(\al_k)\neq 0$.  
If these hold,  the extremal K\"ahler metric of Calabi type is manifestly in some class $\cS_{\alpha, \beta}(\Delta, {\bf L}, F)$, 
and thus is of Poincar\'e type according to Theorem~\ref{thm:criterion}.  
The vanishing conditions are always satisfied.  The only condition we need to verify is $A'''(\al_k)\neq 0$. 
To this end, we describe the solutions explicitly. 

Letting 
\begin{equation}
\begin{split}
\al_\1 &=1, \al_\2=a>1, r_{\al,k}=0, \\
\be_\1&=0, \be_\2= m, r_{\be,\1}= - r_{\be,\2} = 1,  
\end{split}
\end{equation}
we obtain in the case  $Z= S_0 \cup S_{\infty}$  an  extremal K\"ahler metric on $X\setminus Z$ given by \eqref{calabi-type} with  
\begin{equation}\label{double-cusp-calabi}
B(y) = -\frac{2}{m}y(y - m), \ A(x)=-\frac{2}{m(a^2 + 4a +1)}(x-1)^2(x-a)^2
\end{equation}
where $a>1$ parametrizes (up to a scale) the K\"ahler cone of $\F_m$. 
This is a complete extremal K\"ahler metric  defined on the total space of the principal $\C^{\times}$-bundle over $\C P^1$  
classified by  $c_1(\cO(m)) \in H^2(\C P^1, \Z)$,  with cusp singularities at $0$ and $\infty$. 
The conditions $A'''(1) \neq 0 \neq A'''(a)$ obviously hold.

Similarly, when $Z=S_0$ say, for the same choice of $\al_k, \be_k$  the extremal solution is given by \eqref{calabi-type}  with
\begin{equation}\label{zero-cusp-calabi}
B(y)  = -\frac{2}{m}y(y - m), \ A(x) =-(px+q)(x-1)^2(x-a),
\end{equation}
where the constants $p,q$ are given by
\begin{equation}\label{AB-0}
\begin{split}
p & = \frac{2( \frac{r_{\al,\2}(a+2)}{(a-1)^2} -\frac{1}{m})}{(a^2 + 4a +1)}, \\
q &= \frac {2( \frac{r_{\al,\2}(2a+1)}{(a-1)^2} +\frac{a}{m})}{(a^2+4a+1)}.
\end{split}
\end{equation}
Such a metric compactifies smoothly at $S_{\infty}$ precisely when the real parameter $r_{\al,\2}=1$, 
which gives  the  complete extremal K\"ahler metrics in Corollary~\ref{Hirzebruch};  
for other values of $r_{\al,\2}>0$,  one gets a complete metric on $X\setminus S_0$ with a cone singularity of angle $2\pi r_{\al,\2}$ along $S_{\infty}$.   Now, the condition $A'''(1)=0$  is equivalent to $p=-q$. With $r_{\al, \2}=1$, this reads as 
$$3m(a+1) + (a-1)^3 =0,$$
which is impossible for $a>1$.

The case of $Z=S_{\infty}$ can be treated similarly. \end{proof}

\begin{rem}
As a special case of the ansatz \eqref{zero-cusp-calabi},  one can construct  CSCK metrics by putting $r_{\al,\2}= \frac{(a-1)^2}{m(a+2)}$ (i.e. setting the coefficient $p$ in \eqref{AB-0} to be zero). For each $m\ge 1$, this defines a CSCK metric on $\F_m\setminus S_0$, in each K\"ahler class of $\F_m$  (parametrized by $a>1$) with a cusp singularity along $S_0$ and a cone singularity of angle $2\pi \Big(\frac{(a-1)^2}{ma(a+2)}\Big)<2\pi$ along  $S_{\infty}$.   
\end{rem}

\subsection{Regular ambitoric structures} Let $q(z)=q_0z^2+2q_1 z+ q_2$ be a
quadratic polynomial,  $M$ a $4$-dimensional manifold 
with real-valued functions $(x,y,\tau_0,\tau_1,\tau_2),$ such that $x>y$,  
$2q_1\tau_1=q_0\tau_2+q_2\tau_0$, and  at each point of $M$, the $1$-forms  $dx, dy, d\tau_0, d\tau_1, d\tau_2$ span the cotangent space. 
Let $\tor$ be the $2$-dimensional space of vector fields $K$
on $M$ satisfying $dx(K)=0=dy(K)$ and $d\tau_j(K)$ constant. 
Then, for any  smooth  and positive functions  of one variable,  $A(x)$ and $B(y)$,  defined  on  the images of $x$ and $y$ in $\R$, respectively, 
$M$ is ambitoric with respect to $\tor$ and the K\"ahler structures
\begin{align} \label{g-pm}
g_\pm =  \biggl(\frac{x-y}{q(x,y)}\biggr)^{\pm1}
\biggl(& \frac{dx^2}{A(x)} + \frac{dy^2}{B(y)}  + A(x) \Bigl(\frac{y^2 d\tau_0 + 2y d\tau_1 + d\tau_2}{(x-y)q(x,y)}\Bigr)^2 \\ \nonumber
& + B(y) \Bigl(\frac{x^2 d\tau_0 + 2x d\tau_1 + d\tau_2}{(x-y)q(x,y)}\Bigr)^2
\biggr),\\
\label{omega-pm}
\omega_\pm = \biggl(\frac{x-y}{q(x,y)}\biggr)^{\pm 1}
& \frac{dx\wedge (y^2 d\tau_0 + 2y d\tau_1 + d\tau_2)
\pm dy \wedge (x^2 d\tau_0 + 2x d\tau_1 + d\tau_2)}{(x-y)q(x,y)},\\ \nonumber
J_\pm dx= A(x)&\frac{y^2 d\tau_0 + 2y d\tau_1 + d\tau_2}{(x-y)q(x,y)}, \quad 
J_\pm dy= \pm B(y)\frac{x^2 d\tau_0 + 2x d\tau_1 + d\tau_2}{(x-y)q(x,y)},
\end{align}
where $q(x,y)=q_0xy+q_1(x+y)+q_2$.  The metric $g_+$ is extremal iff $g_-$ is extremal  iff 
\begin{equation}\label{regular-extremal}\begin{split}
A(z)&=q(z)\pi(z)+P(z),\\
B(z)&=q(z)\pi(z)-P(z),\\
\end{split}\end{equation}
where $\pi(z)=\pi_0z^2 + 2\pi_1z + \pi_2$ is a polynomial of degree at most two satisfying  $2\pi_1 q_1 -(q_2\pi_0 + q_0\pi_2)=0$,  and $P(z)$ is
polynomial of degree at most four.

The space of Killing fields of $g_{\pm}$  for  the torus is naturally isomorphic to the space of $S^2_{0,q}$ of polynomials $p(z)$ of degree $\le 2$ 
which are orthogonal to $q$ with respect to the inner product $\langle\cdot, \cdot \rangle$ defined by the discriminant, i.e. $p(z)= p_0z^2 + 2p_1z + p_2$ satisfying
$$\langle p, q \rangle =2p_1 q_1 -(q_2p_0 + q_0p_2)=0.$$
The space $S^2_{0,q}$ is in turn isomorphic to the quotient space $S^2 /\langle q \rangle$ of all polynomials of degree $\le 2$ by the subspace generated by $q$, 
by using $\frac{1}{2} {\rm ad}_q$ with respect to the Poisson bracket $${\rm ad}_q(w)=\{q,w\}= q'w - w'q$$ on $S^2$. 
Thus, if $\{p_1,p_2\}$ is a basis of $S^2_{q,0}$ and $\{w_1,w_2\}$ the corresponding basis of $S^2 /\langle q \rangle$ 
(with $p_i = \frac{1}{2}\{q, w_i\}$) momentum/angular coordinates for $g_{\pm}$ are given by
\begin{equation}\label{moments}
\begin{split}
x_i^+ &= w_i(x,y)/q(x,y), \ t_i, \  (i=1,2) \\
x_i^- & = p_i(x,y)/(x-y), \ t_i, \  (i=1,2).
\end{split}
\end{equation}
It follows that the lines $x=\alpha$ (resp. $y=\beta$) transform to lines $\ell^+_{\alpha}= \frac{(x-\alpha)(y-\alpha)}{q(x,y)}=0$ 
(resp. $\ell^+_{\beta}= \frac{(x-\beta)(y-\beta)}{q(x,y)}$) in the $(x_1^+, x_2^+)$ plane,  
which are tangent to the non-degenerate conic $C^*_+ \subset  \tor^*$ corresponding to $\Big(\frac{x-y}{q(x,y)}\Big)^2=0$; 
similarly, $\ell^-_{\alpha}= \frac{(x-\alpha)q(y, \alpha)}{(x-y)}$ 
(resp. $\ell_{\beta}^-=\frac{(y-\beta)q(x,\beta)}{x-y}$) are lines in the $(x_1^-,x_2^-)$-plane (corresponding to $x=\alpha$ and $y=\beta$ in the $(x,y)$-plane) 
which are  tangent to the (possibly degenerate) conic $C^*_-\subset \tor^*$  defined by $\Big(\frac{q(x,y)}{x-y}\Big)^2=0$.  
In both cases, the corresponding normals are
\begin{equation}\label{normals}
p_{\alpha}(z) = q(\alpha, z)(z-\alpha) ; \ \ p_{\be} (z) = q(z, \be)(z-\be),
\end{equation}
viewed as elements of $S^2_{0,q}$.

It is straightforward to compute the  matrix  ${\mathbf H}^{A,B}_{\pm}$ of
$g_\pm$:
\begin{equation}\label{solution-regular}
\begin{split}
{\mathbf H}^{A,B}_{-}(p_i,p_j)&=\frac{A(x) p_i(y)  p_j(y) + B(y) p_i(x) p_j(x)}{(x-y)^3\, q(x, y) },\\
{\mathbf H}^{A,B}_{+}(p_i,p_j)&=\frac{A(x) p_i(y) p_j(y) + B(y) p_i(x) p_j(x)}{(x-y)\, q(x, y)^3 }.
\end{split}
\end{equation}
whose inverses are the Hessians in momenta of the symplectic potentials
\begin{equation}\label{U-AB}
\begin{split}
u^+_{A,B}(x,y) &= - \int^x \frac{(t-x)(t-y)dt}{q(x,y)A(t)}  + \int^y \frac{2(t-x)(t-y)dt}{q(x,y)B(t)},\\
u^-_{A,B}(x,y)  & =  \int^x \frac{2(x-t)q(y,t)dt}{(x-y)A(t)}  + \int^y \frac{2(y-t)q(x,t)dt}{(x-y)B(t)}.
\end{split}
\end{equation}

In order for $u^{\pm}_{A,B}$ be in  $\cS(\Delta, {\bf L}, F)$ for some labelled compact convex quadrilateral $(\Delta,{\bf L},F)$, 
one has to choose real numbers $\al_k, \be_k, r_{\al,k}, r_{\be,k}$ $(k=\1,\2)$ satisfying the inequalities
\begin{equation*}
\be_\1 < \be_\2 <  \al_\1 < \al_\2, \qquad
r_{\al,\1} \le 0 \le r_{\al,\2}, \quad r_{\be,\1} \ge  0 \ge r_{\be,\2},
\end{equation*}
and such that $q(x,y)>0$ on $D=[\al_\1,\al_\2]\times [\be_\1,\be_\2]$, and then impose on the smooth functions $A(x)$, $B(y)$ the positivity conditions 
\begin{equation}\label{positivity-AB}
A(x)>0 \ {\rm on} \  (\al_\1,\al_\2) \  {\rm and} \ B(y)>0 \ {\rm on} \  (\be_\1, \be_\2), 
\end{equation}
and the boundary conditions 
\begin{equation}\label{boundary-AB}
A(\al_k)=0=B(\be_k),\  A'(\al_k)=-2r_{\al,k}, \ B'(\be_k)
=2r_{\be,k} \ (k=\1,\2).\end{equation}

Considering ${\bf H}^{A,B}_+$ for simplicity (and dropping the $+$ script), the data as above gives rise to a convex compact quadrilateral $\Delta$ 
(determined by the affine lines $\ell_{\al_k}$ and $\ell_{\be_k}$ introduced above via \eqref{moments}) 
which is  endowed with the canonical set  $\{p_{\al_k}, p_{\be_k}, k=\1,\2\}$ of normals \eqref{normals}. 
We take $F$ be the union of all facets $\ell_{\al_k}=0$ and $\ell_{\be,k}=0$ for which  $r_{\al,k}=0$ and $r_{\be,k}=0$, and normalize the remaining normals by 
$$e_{\al,k}: = p_{\al_k}/r_{\al,k}, \ e_{\be,k}:= p_{\be_k}/r_{\be,k}.$$
One can easily check that these become inward normals to $\Delta$ and that ${\bf H}^{A,B}$ verifies the boundary conditions \eqref{boundary-divisor-0}-\eqref{boundary-divisor-1} 
on $(\Delta, {\bf L}, F)$ if and only if \eqref{boundary-AB} holds. 
Furthermore, as it is shown in \cite{ACG1},  ${\bf H}^{A,B}$ gives rise to a solution of the Abreu equation \eqref{abreu1} on $(\Delta,{\bf L},F)$  
iff $A, B$ are polynomials of degree $\le 4$ which satisfy \eqref{regular-extremal} and the  positivity and boundary conditions \eqref{positivity-AB}-\eqref{boundary-AB}.

Conversely, the following is established in \cite{ACG2}.
\begin{prop}\label{main} \cite{ACG2} Let $(\Delta, {\bf L})$ be a compact convex labelled quadrilateral in $\R^2$, and $F$ the union of some of its facets. 
Suppose that $\Delta$  is neither a parallelogram nor a trapezoid whose extremal affine function $s_{(\Delta, {\bf L}, F)}$ is constant on the  parallel facets $\Delta$. 
Then there exist real numbers $\al_k, \be_k, r_{\al,k}, r_{\be,k}$ $(k=\1,\2)$, subject
to the inequalities
\begin{equation*}
\be_\1 < \be_\2 <  \al_\1 < \al_\2, \qquad
r_{\al,\1} \le 0 \le r_{\al,\2}, \quad r_{\be,\1} \ge  0 \ge r_{\be,\2},
\end{equation*}
and a quadratic $q(z)$  satisfying $q(x,y)>0$ on $D=[\al_\1,\al_\2]\times [\be_\1,\be_\2]$,  such that 
\begin{bulletlist}
\item $\Delta$ is the image of $D$ either under $(x_1^+, x_2^+)$  or  $(x_1^-, x_2^-)$ in \eqref{moments};
\item For each facet $F_{\al,k}$ of $\Delta$ obtained as the image of $x=\al_k$ under \eqref{moments} {\rm (}resp. $F_{\be,k}$ obtained as the image of $y=\beta_k${\rm )}, 
      which does not belong to $F$, $r_{\al,k}\neq 0$ {\rm (}resp. $r_{\be,k} \neq 0${\rm )} 
      and the corresponding inward normal is  $e_{\al,k} = \frac{p_{\alpha_k}}{r_{\al,k}}$ 
      {\rm (}resp. $e_{\be,k} = \frac{p_{\beta_k}}{r_{\al,k}}${\rm )}, where $p_{\al_k}$ and $p_{\be_k}$ are the the normals defined by  \eqref{normals}; 
\item For each facet $F_{\al,k}$ of $\Delta$ {\rm (}resp. $F_{\be,k}${\rm )} which belongs to $F$, the corresponding $r_{\al,k}= 0$ {\rm (}resp. $r_{\be,k} =0${\rm )};
\item There exist polynomials $P(z)$ of degree $\le 4$ and $\pi(z)$ of degree $\le 2$, satisfying $\langle q, \pi \rangle =0$,  
      such that the $A(z)$ and $B(z)$ defined by \eqref{regular-extremal} satisfy the boundary conditions \eqref{boundary-AB} 
      {\rm (}but not necessarily the positivity condition \eqref{positivity-AB}{\rm )}.
\end{bulletlist}
Furthermore, the corresponding ${\bf H}_{+}^{A,B}$ or ${\bf H}_-^{A,B}$ defined by \eqref{solution-regular} satisfies \eqref{abreu1} 
and defines a solution $u_{A,B}\in \cS(\Delta, {\bf L}, F)$  if and only if $(\Delta, {\bf L}, F)$ is stable.  
The latter condition is  equivalent to \eqref{positivity-AB}.
\end{prop}

As an illustration of the theory, let us again take $(\Delta, {\bf L})$ to be a trapezoid but not a parallelogram,  
and $F$ to be either a facet which is not parallel to another facet, or  the union of two adjacent facets. 
We have shown in Section ~\ref{s:calabi} that 
in this case the extremal affine linear function $s_{(\Delta,{\bf L},F)}$ is {\it not} constant on the parallel facets of $\Delta$ and, 
therefore, the solution of \eqref{abreu1} (if it exists)  must be given by Proposition~\ref{main} above. 
On the other hand, we have the following
\begin{prop}\label{stable-Hirzebruch} 
Let $(\Delta, {\bf L})$ be a labelled trapezoid corresponding to a Hirzebruch surface, and $F$ be one facet, or the union of $2$ adjacent facets. 
Then $(\Delta, {\bf L}, F)$ is stable.
\end{prop}
Putting Propositions~\ref{main} and \ref{stable-Hirzebruch} together, we obtain
\begin{cor}\label{Hirzebruch-non-poincare} Let $X= \F_m$ be the $m$-th Hirzebruch surface and $Z$ be the divisor consisting of a single fibre fixed by the $\T$ action, 
or the union of such a fibre with either the zero section or the infinity section. 
Then,  $X\setminus Z$ admits a complete extremal Donaldson metric  in each K\"ahler class of $X,$ which is \emph{not} of Poincar\'e type.
\end{cor}
The proofs of Proposition \ref{stable-Hirzebruch}  and Corollary~\ref{Hirzebruch-non-poincare}  are  presented in Appendix \ref{AppendixB}. 

\begin{ex}\label{ex:1}  In the light of Corollary~~\ref{Hirzebruch-non-poincare}, 
we  use  the explicit description of the extremal Donaldson metrics in order to determine their  asymptotic behaviour in normal direction to $Z$.

\smallskip
The parametrization of a regular ambitoric metric by the data $$\al_k, \be_k, r_{\al,k}, r_{\be,k}, q(z), A(z), B(z)$$ as above is not effective: 
there is a natural ${\rm SL}(2, \R)$ action on the space of degree $2$ polynomials $q(z)$,  as well as a homothety freedom for the metric. 
This can be normalized by taking $q(z)$ to be either $1, 2z$ or $z^2+1$ 
(see \cite[Sec.~5.4]{ACG1}), thus referring to the corresponding ambitoric metric as being of {\it parabolic}, {\it hyperbolic} or {\it elliptic type}, respectively.   
Moreover, it is observed  in  \cite[Sec.~5.4]{ACG2}  that the solution corresponding to a trapezoid is given by a (positive) hyperbolic ambitoric metric, i.e. 
\begin{equation}\label{hyperbolic}
\begin{split}
g =& \frac{(x-y)}{(x+y)}\Big(\frac{dx^2}{A(x)} + \frac{dy^2}{B(y)}\Big) \\ &+ \frac{1}{(x-y)(x+y)^3}\Big(A(x)(dt_1 + y^2 dt_2)^2 + B(y)(dt_1 + x^2 dt_2)^2\Big),\\
\omega = & \frac{dx\wedge (dt_1 + y^2 dt_2)}{(x+y)^2} + \frac{dy\wedge (dt_1 + x^2 dt_2)}{(x+y)^2},
\end{split}
\end{equation}
for $(x, y) \in [\al_\1, \al_\2]\times [\be_\1, \be_\2]$ with
$$\be_{\1} <\be_{\2} < \al_{\1} < \al_{\2}, \ \  \ \be_{\1} + \al_{\1} >0$$
and polynomials $A(z)= \sum_{i=0}^4 a_i z^{4-i}$ and $B(z)=\sum_{i=0}^4 b_i z^{4-i}$   satisfying
\begin{equation}\label{extremal-hyperbolic-0}
a_0+b_0=a_2+b_2 = a_4+b_4=0,
\end{equation}
and the positivity and boundary conditions  \eqref{positivity-AB}-\eqref{boundary-AB}. 
The momentum coordinates of \eqref{hyperbolic}  then become
\begin{equation}\label{momentum-hyperbolic}
x_1= -\frac{1}{x+y}, \ x_2 = \frac{xy}{x+y}
\end{equation}
so that the image of the interval $[\al_\1, \al_\2]\times [\be_\1, \be_\2]$ under \eqref{momentum-hyperbolic} is a quadrilateral $\Delta$ determined by the affine lines 
$$\ell_{\alpha,k}= -\al_k^2x_1 +x_2 - \al_k=0, \ \  \ell_{\beta,k}= -\be_k^2x_1 + x_2 - \be_k =0, \ k=\1, \2, $$
whose normals are $p_{\al,k}=(-\al_k^2, 1)$ and $p_{\be,k}=(-\be_k^2, 1),$ respectively. It follows that $\Delta$ is a trapezoid iff $\be_\2=-\be_\1=b>0$, see Figure 2 below.

\begin{figure}[h!]
   \centering
   \includegraphics[scale=0.50]{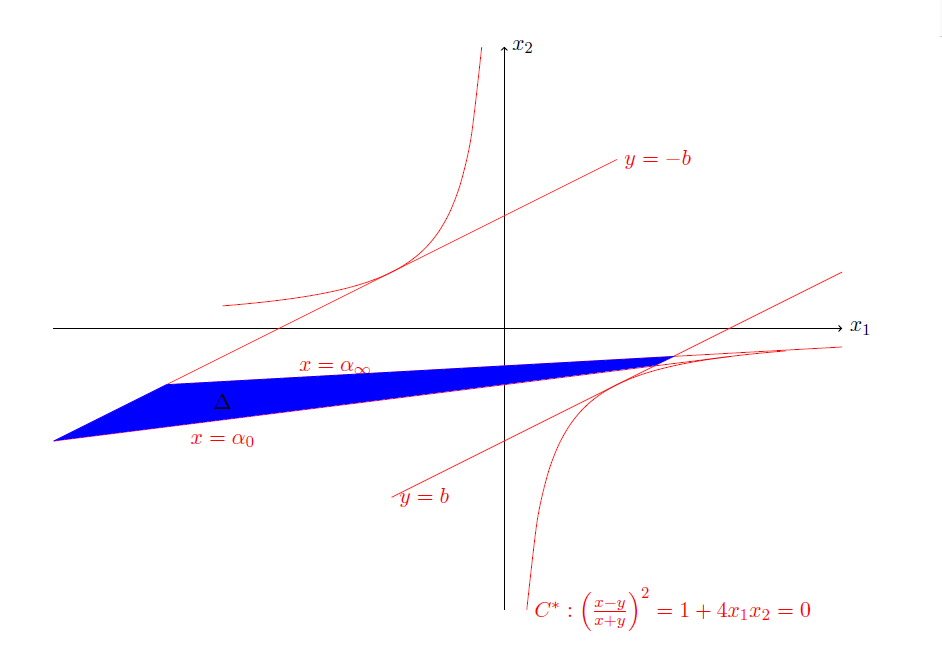}
   \caption{The Delzant polytope of a Hirzebruch surfaces (in blue)  obtained by the hyperbolic ambitoric construction.}
\end{figure}

As observed in \cite{Dixon}, each Hirzebruch surface $\F_m$  can be obtained from a labelled trapezoid $(\Delta, {\bf L})$ as above,  
by taking  inward normals $e_{\al,k}= p_{\al_k}/r_{\al,k}$ and $e_{\be,k}=p_{\be_k}/r_{\be,k}$
satisfying
$$e_{\be,\1}= -e_{\be,\2},  \  e_{\al,\2} + m e_{\be, \1}= -e_{\al, \1}.$$
Equivalently, the labelling ${\bf L} =\{L_{\al,k}=\frac{1}{r_{\al,k}} \ell_{\al,k}, L_{\be, k} = \frac{1}{r_{\be, k}} \ell_{\be, k}, k=\1, \2\}$   satisfies
\begin{equation}\label{hirzebruch}
r_{\be,\1}=-r_{\be, \2}=r>0, \ r_{\al,\1}= \frac{r}{m} \Big(\frac{\al_\1^2-\al_\2^2}{\al_\2^2-b^2}\Big), \  
r_{\al, \2}= \frac{r}{m} \Big( \frac{\al_\2^2 -\al_\1^2}{\al_\1^2-b^2}\Big).
\end{equation}
The positive constant $r$ is just a scale factor for the K\"ahler class and can be taken $r=1$. 
Thus, by considering the lattice generated by $e_{\al,k}, e_{\be,k}$ as above, the corresponding labelled trapezoid corresponds to a toric Hirzebruch surface $\F_m$.

We now take  $F_{\alpha, \2}$ (defined by $\ell_{\al,{\2}}=0$) be the facet of $\Delta$ corresponding to a fibre of $\F_m$ fixed by the torus action.  
We are thus looking for extremal metrics given by \eqref{hyperbolic} for polynomials 
\begin{equation}\label{hyperbolicAB}
A(x)= -c(x-\al_\1)(x-\al_\2)^2(x-\al_3), \ B(y)= c(y-b)(y+b)(y^2 + py + q)
\end{equation}
where $0<b<\al_\1 < \al_\2$  and $\al_3, c, p, q$ are real parameters (which we are going to express as functions of $(b, \al_{\1}, \al_{\2})$). 

The extremality conditions \eqref{extremal-hyperbolic-0} then  read as 
\begin{equation}\label{extremal-hyperbolic}
\begin{split}
& \al_3(2\al_\2 + \al_\1) + \al_\2^2 + 2\al_\1\al_\2 = q- b^2\\
& \al_3\al_\2^2\al_\1= -q b^2
\end{split}
\end{equation}
from which we get
\begin{equation}\label{al-be}
\begin{split}
\al_3 &=- \Big(\frac{b^2 + \al_\2^2 + 2\al_\1\al_\2}{2\al_\2 + \al_\1 + \frac{\al_\2^2\al_\1}{b^2}}\Big), \\
q &=  \frac{\al_\2^2\al_\1}{b^2}\Big(\frac{b^2 + \al_\2^2 + 2\al_\1\al_\2}{2\al_\2 + \al_\1 + \frac{\al_\2^2\al_\1}{b^2}}\Big).
\end{split}
\end{equation}
From the boundary condition  \eqref{boundary-AB} at $\al_{\1}$ we obtain
$$A'(\al_\1)=-c(\al_\1-\al_3)(\al_\1-\al_\2)^2 =-2r_{\al, \1}= \frac{2}{m} \Big( \frac{\al_\2^2 -\al_\1^2}{\al_\2^2-b^2}\Big)$$
so that we determine
\begin{equation}\label{c}
\begin{split}
c &= -\frac{2}{m} \Big(\frac{(\al_{\2} + \al_{\1})}{(\al_{\1}-\al_3)(\al_{\2}-\al_{\1})(\al_{\2}^2 - b^2)}\Big) \\
  &= -\frac{2}{m} \Big(\frac{(\al_{\2} + \al_{\1})(2\al_{\2} + \al_{\1} + \frac{\al_{\2}^2\al_{\1}}{b^2})}{\big(\al_{\1}^2 + \al_{\2}^2 + 4\al_{\1}\al_{\2} + b^2 + \frac{\al_{\1}^2\al_{\2}^2}{b^2}\big)(\al_{\2}-\al_{\1})(\al_{\2}^2 - b^2)}\Big) .
  \end{split}
\end{equation}

\smallskip
Consider first the case  when  $F=F_{\al, {\2}}$ consists of only one facet. 
The boundary conditions \eqref{boundary-AB} at $\pm b$ read as
$$B'(b) =2cb(b^2 + pb + q) =2r_{\be, \2}=-2, \  B'(-b)=-2cb(b^2 - pb + q) = 2r_{\be, \1}= 2.$$
We then have $p=0$ and the additional relation  $-1 = cb(b^2 + q)$  
which can be used in order to express $\al_{\1}$ as a function of $(\al_{\2}, b)$. 
This last step, however, is not obvious (and is implicit) as $\al_{\1}$ appears to be a real root of a polynomial of degree $4$, 
which also needs to satisfy $0<b  < \al_{\1} < \al_{\2}$.  
The existence of such a root is thus  guaranteed by Propositions~\ref{main} and \ref{stable-Hirzebruch},  so we shall not develop this step any further. 
We also notice that homotheties in $(x,y)$ preserve the form \eqref{hyperbolic} (but change $A$ and $B$ by scale) so we can assume $b=1$. 
Thus, on a fixed Hirzebruch surface $\F_m$ we obtain a one-dimensional family of complete extremal K\"ahler metrics  (defined on $\F_m \setminus Z$) 
parametrized by $a=\al_{\infty}>1$, which is precisely the dimension of the K\"ahler cone of $\F_m$ modulo scales. 

Notice that, by \eqref{al-be},  the third root $\alpha_3$ of $A$ is negative, thus $\al_{\2}$ has aways multiplicity $2$. 
Using  \eqref{U-AB} with $q(x,y)=x+y$ and $A,B$ given by \eqref{hyperbolicAB}, we observe that  up to smooth terms  on $\Delta$, the symplectic potential  is of the form
\begin{equation*}
\begin{split}
u  &= \Big(A \frac{(2\al_{\2} -(x+y))}{(x+y)} + B L_{\al, {\2}}\Big) \log L_{\al, {\2}} + \frac{1}{2} \sum_{k=\1, \2}L_{\al,k} \log L_{\al,k} + L_{\be,k} \log L_{\be, k}  \\
    & = \big(A (-2\al_{\2}x_1 -1)  + B L_{\al, {\2}}\big) \log L_{\al, {\2}} + \frac{1}{2} \sum_{k=\1, \2}L_{\al, k} \log L_{\al, k} + L_{\be,k} \log L_{\be, k} 
    \end{split}
\end{equation*}
for some real constants $A \neq 0, B$.
As $\al_{\2}>0$, the affine function $\big(A (-2\al_{\2} x_1 -1)  + B L_{\al, {\2}}\big)$ is {\it not} constant when restricted to  the facet $F_{\al, \2}$  
(on this facet $x = \al_{\2}$ and $y \in [\be_{\1}, \be_{\2}]$).

\smallskip
Similarly, when $F= F_{\al, \2}\cup F_{\be, \2}$ say,  we must  have $y^2+py + q= (y-b)(y-\be_3)$, so that $q= b\beta_3$ and $p= -(b + \beta_3)$,  
and from \eqref{al-be} we determine
$$\beta_3 = \frac{\al_\2^2\al_\1}{b^3}\Big(\frac{b^2 + \al_\2^2 + 2\al_\1\al_\2}{2\al_\2 + \al_\1 + \frac{\al_\2^2\al_\1}{b^2}}\Big).$$
The above formula together with the inequalities $0<b<\al_{\1} < \al_{\2}$ show that $\beta_3 >b$, thus $b$ is double root of $B(y)$. 
Similarly to the previous case, the symplectic potential of the extremal metric then takes the form
$$u = f_{\al, \2} \log L_{\al, \2} + f_{\be, \2}\log L_{\be, \2} + \frac{1}{2}\Big(L_{\al, \1} \log L_{\al, \1} +  L_{\be, \1} \log L_{\be, \1}\Big) + smooth$$
where $ f_{\al, \2}$ and $f_{\be, \2}$ are affine functions in momenta which are not constant on the corresponding facets in $F$. 
One can also check that in this case too the condition \eqref{hyperbolicity-condition} fails.

 \end{ex}

We notice also
\begin{prop}\label{p:unstable} 
Let $X$  be a Hirzebruch surface $\F_m$ or $\C P^1 \times \C P^1$, 
viewed as a toric variety endowed with a K\"ahler class $[\omega]$ corresponding to a Delzant polytope $(\Delta, {\bf L})$.  
Let $Z\subset X$ be the toric divisor corresponding  to the union $F$ of $3$ facets of $\Delta$.  
Then $(\Delta, F)$ is unstable and $X\setminus Z$ admits neither  a Donaldson extremal K\"ahler metric nor an extremal K\"ahler metric of Poincar\'e type in $[\omega]$.
\end{prop}
\begin{proof} The proof of instability of $(\Delta, F)$ follows from the arguments in Appendix~\ref{AppendixB}, see in particular Remark \ref{3facetsproof}. 
Thus $(\Delta, F)$ cannot admit a Donaldson metric by Proposition~\ref{necessary}. 

To rule out the existence of a (non-toric) complete extremal metric of Poincar\'e type, 
we can use \cite[Thm.~5]{Auvray1} which asserts that each rational curve corresponding to the  preimage of a facet in $F$  
must admit a complete Poincar\'e type extremal K\"ahler metric.   
Taking the $\C P^1$ corresponding to the  facet in $F$ which intersects the other two facets in $F$, we conclude that $\C P^1 \setminus (\{p\} \cup \{q\})$ admits 
a complete extremal metric of Poincar\'e type. 
But if it did, it would have to be scalar-flat, as the Poincar\'e-Futaki invariant vanishes and the average scalar curvature is $0$. 
This would then violate the numerical constraint in \cite[Thm. 1.2]{Auvray} for Poincar\'e type metrics of constant scalar curvature. 
So no such metric can exist. 
\end{proof}

\begin{cor}
Conjecture~\ref{conjecture1} holds true  if $X$ is a Hirzebruch surface $\F_m$ or $\C P^1 \times \C P^1$.
\end{cor}
\begin{proof} Using Corollaries \ref{c:product}, \ref{Hirzebruch} and \ref{Hirzebruch-non-poincare} together with Proposition~\ref{necessary} at one hand, 
and Propositions~\ref{trapezoid} and \ref{p:unstable} at the other hand, 
we conclude that the conditions (i) and (ii) of Conjecture~\ref{conjecture1} limit the possibilities  as follows:
\begin{enumerate}
\item[\rm (a)] $X= \C P^1 \times \C P^1$ and $Z$ is the union of  the preimage of one or two adjacent facets;
\item[\rm (b)] $X = \F_m$ and $Z$ consist of either the zero section, the infinity section, or the the union of both,;
\item[\rm (c)] $X= \F_m$ and $Z$ consist of a single fibre or the union of such a fibre and either the zero section or the infinity section.
\end{enumerate}
In the cases (a) and (b), there  exists an explicit extremal Poincar\'e type metric by the Corollaries \ref{c:product} and \ref{Hirzebruch}. 
In the case (c), there exists a Donaldson complete extremal metric  which is not of Poincar\'e type,  
but in this case the condition (iii) of Conjecture~\ref{conjecture1} fails, as shown in Example~\ref{ex:1}. \end{proof}

\subsection{Triangles as a limiting case} 
This case is already treated in \cite{Bryant} (see also \cite{Abreu1}), but it can also be viewed as a limiting case of the ambitoric ansatz with $\pi =0$ (i.e. $A=-B$). 
The corresponding extremal metrics  $(g_+, J_+, \omega_+)$ provide solutions of \eqref{abreu1} on labelled triangles, 
and compactify on weighted projective planes as extremal Bochner--flat (i.e. self-dual) orbifold metrics, see \cite{Abreu1,Bryant}. 

Indeed, putting $\pi=0$ and $P(z)= - \prod_{j=0}^{3} (z-\beta_j)$ with $\beta_0\le \beta_1<\beta_2 <  \beta_3$  in \eqref{g-pm} and \eqref{omega-pm}, 
the degree $4$ polynomial $B(y)=-P(y)$ is positive on $(\beta_1, \beta_2)$ while $A(x)=P(x)$ on $(\beta_2, \beta_3)$. 
When $\beta_0< \beta_1$, the K\"ahler metric $(g_+,\omega_+)$ defines an extremal Bochner-flat K\"ahler metric on a labelled  simplex $(\Delta,{\bf L})$, 
while taking $\beta_0=\beta_1$ gives rise to a  solution to the Abreu equation \eqref{abreu1} on a labelled simplex minus one facet 
(corresponding to the image of $y=\beta_1$ under the momentum map \eqref{moments}). 
One can always take the two normals to form  a basis of a lattice, so that the metric extends smoothly over the corresponding faces and has a complete end towards the third, 
see \cite{Dixon}.
We get, in fact,  one of the complete Bochner-flat metrics described in \cite[Thm. 4.2.7]{Bryant} (see also \cite{david-gauduchon}). 

To see this explicitly, let us identify (by an affine map) the simplex $\Delta$ with the standard simplex of $\R^2$ (with vertices at $(0,0), (1,0)$ and $(0,1)$)  
and assume (without loss) that $F$ corresponds to the facet defined by the equation $L_3(x)= a_3(1-x_1-x_2)=0$ 
whereas the other labels are $L_1(x)=a_1x_1$ and $L_2(x)=a_2x_2$  with $a_i>0$. 
The Bryant complete extremal Bochner-flat metric has symplectic potential in $\cS(\Delta, {\bf L}, F)$, given by
\begin{equation}\label{bryant}
u_{B} = \frac{1}{2} \Big(a_1x_1\log(x_1) + a_2 x_2\log(x_2) - (a_1x_1 + a_2x_2)\log(1-x_1-x_2)\Big).
\end{equation}

If we take $(\Delta, {\bf L})$ be a labelled simplex and $F= F_1 \cup F_2$ the union of  two facets, 
then by identifying $\Delta$ with the standard simplex of $\R^2$ and  $F_i$ with  the affine line $x_i=0, i=1,2$, respectively,    
one sees that  the reflection along the line  $x_1=x_2$ is a symmetry of $(\Delta, {\bf L}, F)$. 
By uniqueness,  $s_{(\Delta,{\bf L},F)}$ must be invariant under this reflection, i.e. $s_{(\Delta, {\bf L}, F)}= r(x_1+x_2) + c$ for some real numbers $r,c$. 
Using the definition \eqref{extremal-function} with $f= 1-(x_1 +x_2)$ (which vanishes on $F_3$), one gets $r=-2c$ for a real number $c$ 
(which must be inverse proportional to the normal $e_3$).   
It follows that $s_{(\Delta, {\bf L}, F)}$ vanishes at the affine line parallel  to $F_3$ and passing though the midpoint $m=(1/4,1/4)$ of its median $d$. 
Let $f_d$ be a simple crease function with crease along $d$ and non-zero on the sub-triangle $\Delta' \subset \Delta$ (cut from $\Delta$ by $d$).  
Two of the facets of $\Delta'$ inherit the measures of the facets of $\Delta$ and  we put measure zero to the facet along $d$. 
Thus, $(\Delta', d\nu_{\partial \Delta'})$ and $(\Delta, d\nu_{\partial \Delta})$ are equivalent under an affine transformation of $\R^2$. 
From the affine characterization of $s_{(\Delta, {\bf L}, F)}$, it follows that the extremal affine linear function of $\Delta'$ is a multiple of $s_{(\Delta, {\bf L}, F)}$; 
it is not hard to see (e.g. by using the definition  \eqref{extremal-function} with $f= 1-(x_1 +x_2)$ and $f\equiv 1$) 
that the extremal affine linear function of $\Delta'$  equals to $s_{(\Delta, {\bf L}, F)}$. 
Thus, $\cL_{(\Delta, {\bf L}, F)}(f_d)$ also computes the Donaldson--Futaki invariant of the affine linear function  $f_d$ over $\Delta'$, and hence is zero.  
It follows that $(\Delta, {\bf L}, F)$ is unstable. We thus conclude
\begin{thm}\label{thm:triangle} 
Let $(\Delta,{\bf L},F)$ be a labelled simplex in $\R^2$. Then $(\Delta, {\bf L}, F)$ is stable if and only if $F$ consist of a single facet. 
In this case  \eqref{abreu1} admits an explicit solution $u_B$ in $\cS(\Delta, {\bf L}, F)$ given by
\begin{equation}
u_{B} = \frac{1}{2}\Big( L_1 \log L_1 + L_2\log L_2 - (L_1 + L_2)\log L_3\Big),
\end{equation}
where  $L_3$ vanishes on $F$. The corresponding metric \eqref{toric-g} extends to the complete Bochner-flat metric on $\C^2$ found in \cite{Bryant}. 
In particular,  $\C P^2\setminus \C P^1$ admits a complete extremal Donaldson metric,  which is of Poincar\'e type 
(and conformal to the Taub-NUT metric).
\end{thm}

\appendix
\section{Proof of Theorem~\ref{thm:criterion} and Proposition~\ref{complete-first-order}}\label{AppendixA} 
\subsection{Proof of Theorem~\ref{thm:criterion}} We follow the notation of Sections~\ref{s:toric} and~\ref{s:abreu-guillemin}. 
Thus, $X$ is a smooth compact toric variety  classified by the labelled Delzant polytope $(\Delta, {\bf L})$. 
We fix once and for all a $\T$-invariant K\"ahler metric $\omega_0$  on $X$  with momentum map $\mu_0 : X \to \Delta \subset \tor^*$.  
To simplify the discussion, we can take $\omega_0$ to be the K\"ahler quotient metric on $X$ obtained from the flat K\"ahler structure on $\C^d$ via the Delzant construction, 
see equation \eqref{u0}.  We denote by $X^\circ=\mu_0^{-1}(\Delta^0)$ the  preimage of the interior of $\Delta$, which is also the subspace of regular points for the action of $\T$. 
Complexifying the $\T$-action, we obtain a holomorphic action of the complex $n$-torus $\T^c =(\C^{\times})^n$ with $X^{\circ}$ being  the principal orbit for the $\T^c$-action. 
Choosing (once for all) a reference point $z^{\circ}\in X^{\circ}$, for each fixed point for  the $\T$-action,  
corresponding to a vertex $v\in \Delta$, we introduce a $(\C^{\times})^n$-equivariant chart $\C_v^n \cong \C^n$ as follows.  
Using a basis of $\tor$ obtained by the inward normals of the facets of $\Delta$  meeting at $v$, 
we identify $\T^c $ with $(\C^{\times})^n$ and  consider the equivariant map $\Phi_v : (\C^{\times})^n \to X^{\circ}$ 
$$\Phi_{v}(r_1e^{\sqrt{-1}t_1}, \ldots, r_n e^{\sqrt{-1}t_n}) : =(r_1e^{\sqrt{-1}t_1}, \ldots, r_n e^{\sqrt{-1}t_n})\cdot z^{\circ},$$
where $r_ie^{\sqrt{-1}t_i} \in \C^{\times}$ stand for the polar coordinates on each factor.
It follows by the holomorphic slice theorem that $\Phi_v : X^{\circ} \to (\C^{\times})^n$ extends equivariantly to a holomorphic embedding of $\C^n_v$ to $X$, 
thus defining an equivariant atlas of affine charts $\C^n_v$ of $X$ (where $v$ runs among the vertices of $\Delta$). 

The theory of toric varieties (see e.g.~\cite{guillemin-book}) yields that in such a chart, 
the (smooth) divisor $Z$ corresponding to the  preimage under $\mu_0$ of a facet $F$ meeting $v$  
has the equation $z_j=0$ where $(z_1, \ldots, z_n)=(r_1 e^{\sqrt{-1}t_1}, \ldots, r_n e^{\sqrt{-1}t_n})$ are the affine coordinates on $\C_v^n$. 
In what follows,  we will suppose without loss that 
\begin{equation}\label{chart-divisor}
Z\cap \C_v^n =\{(z_1, \ldots, z_n) \in \C^n : z_1 =0\}.
\end{equation}

To connect with the description \eqref{toric-g} of the K\"ahler metric $\omega_0$, 
one needs to apply the Legendre transform to the strictly convex smooth function $u_0$ on $\Delta^0$,  given by \eqref{u0}. 
More precisely, if $x^0 = \mu_0(z^{\circ})\in \Delta^0$ and $u\in \mathcal{S}(\Delta, {\bf L})$ is any symplectic potential, we let 
$$y(x):=du(x) - du(x^0) = (u_{,1}(x), \ldots, u_{,n}(x)) - (u_{,1}(x^0), \ldots, u_{,n}(x^0))$$
and define a smooth function $\varphi_u(y)$ by
\begin{equation}\label{kahler-potential}
\varphi_{u}(y) + u(x) = \sum_{i=1}^n y_ix_i.
\end{equation}
We notice the following elementary  
\begin{lemma}   \label{lem_diffeo}
   Let $\Omega$ be a non-empty bounded convex open subset of a finite-dimensional affine space $\tor^*$, 
   and  $u:\Omega\to\R$ a smooth strictly convex function such that $|du|$ tends to $\infty$ near $\partial \Omega$, 
   where $|\cdot|$ is any Euclidean  norm on $\tor$.  
   Then $du$ is a diffeomorphism from $\Omega$ onto $\tor$. 
  \end{lemma}
 \begin{proof} Clearly, $du(x)=\big(u_{,1}(x), \ldots, u_{,n}(x)\big)$ is a local diffeomorphism  
  as  its differential at $x$ (represented by  the matrix  $(u_{,ij}(x))$ is invertible for  all $x\in \Omega$ by the strict convexity assumption. 
  
  Moreover $du$ is injective by using the convexity of $\Omega$ and the fact that $du(p(t))$ is strictly monotone  in $t$ on each linear segment $p(t) \in \Omega$ 
  (which again follows from the strict convexity of $u$).

  Finally, $du$ is surjective: This can be proven by checking that $du(\Omega)$ is closed, and  hence coincides with $\R^n$ as it is also non-empty and open. 
  Indeed, let $(q_k)$ be a sequence in $\R^n\cong \tor$, with limit $q$, such that $q_k = du(p_k)$ for $p_k\in\Omega$ for all $k\geq0$. 
  There is a subsequence, still denoted by $(p_k)$, converging to some $p\in\overline{\Omega}$. 
  Now if $p\in\partial\Omega$, then by assumption, $|q_k|=|du(p_k)|\to \infty$, a contradiction, and thus $p\in\Omega$, and $q=du(p)$. \end{proof}

It follows that for each $u\in \mathcal{S}(\Delta, {\bf L})$, we have  a $\T^c$-equivariant biholomorphism  
$\Phi_{u} : (\Delta^0 \times \T, J_u)  \to (\C^{\times})_v^n \cong X^{\circ},$ given by
\begin{equation}\label{complex-chart}
\Phi_{u}(x_1, \ldots, x_n, t_1, \ldots, t_n) := (e^{y_1 + \sqrt{-1} t_1}, \ldots, e^{y_n + \sqrt{-1} t_n}),
\end{equation}
where, we recall,  $y(x) = du(x) - du(x^0)$ is a diffeomorphism from $\Delta^0$ to $\R^n$ by virtue of Lemma~\ref{lem_diffeo},  
and $J_u$ is the $\T$-invariant $\omega_0$-compatible complex structure corresponding to $u\in \mathcal{S}(\Delta, {\bf L})$. 
The central fact in this theory is the following  identity on $(\C^{\times})^n_v$ (see \cite{guillemin-book}):
\begin{equation}\label{complex}
\omega_u:=(\Phi_{u}^{-1})^*(\omega_0) = dd^c \Big(\big(\varphi_u(\log |z_1|, \ldots, \log |z_n|)\big)\Big), 
\end{equation}
where  $(z_1, \ldots, z_n)$ are the complex coordinates associated to the chart $\C_v^n,$ and $d^c$ is taken with respect to the standard complex structure.  
The fact that $u\in \mathcal{S}(\Delta, {\bf L})$ guarantees the smooth extension of the right hand side to a positive definite $(1,1)$-form on $\C^n_v$.

 \smallskip
 Let us now suppose $u\in \mathcal{S}_{\alpha, \beta}(\Delta, {\bf L}, F)$ (instead of being in $\mathcal{S}(\Delta, {\bf L})$).  
 It is easily checked that such a $u$ still verifies the condition that $|du|$ tends to $\infty$ near $\partial\Delta$, 
 so that, by using Lemma~\ref{lem_diffeo},  \eqref{complex-chart}  and \eqref{complex},  we obtain a K\"ahler metric $\omega_u$ on $(\C^{\times})^n_v,$ which can be written as
 \begin{equation}\label{potential}
 \omega_u = \omega_{0} + dd^c  \varphi,
 \end{equation}
 where $\omega_0$ is the (globally defined on $\C^n_v$) K\"ahler metric corresponding to \eqref{u0} and 
  \begin{equation}\label{phi}
 \varphi(z_1, \ldots, z_n) :=(\varphi_u-\varphi_{u_0})(\log |z_1|, \ldots, \log |z_n|)
  \end{equation}
 is a smooth function on $(\C^{\times})^n_v \subset \C^n_v$. 
 We notice that through this identification (which depends upon $u$!), $Z\cap \C^n_v$ still corresponds to a hyperplane in the $\C^n_v$ affine chart, 
 as it follows from the following
   \begin{lemma}  \label{lem_diffeo2}
    Let $u_0$ be the symplectic potential in $\mathcal{S}(\Delta, {\bf L})$ given by \eqref{u0} and let $u \in \mathcal{S}_{\alpha, \beta}(\Delta, {\bf L}, F)$. 
    Then,  
     \begin{equation*}
      (du_0)^{-1} \circ du :\Delta^0 \longrightarrow \Delta^0
     \end{equation*}
    extends continuously to a homeomorphism of $\Delta$, 
    inducing a diffeomorphism on every open face of the polytope $\Delta$ and preserving its vertices. 
   \end{lemma}
   \begin{proof} The proof is elementary. Let $p\in \partial\Delta$, and consider some sequence $(p_k)$ in $\Delta^0$ 
   converging to $p$. 
   We want to see that:
    \begin{enumerate}
     \item[\rm (a)] $\lim _{k\to \infty} ((d{u_0})^{-1} \circ du) (p_k)$ lies in the same open facet as $p$ (or equals $p$ if $p$ is a vertex of $\Delta$);
     \item[\rm (b)] this limit does not depend on the sequence $(p_k)$. 
    \end{enumerate}
The claimed  regularity of $(d{u_0})^{-1} \circ du$ will  follow from the proof (a) and (b) above. 
  
  Let us assume that $p \in \mathbf{f} := (F_{i_1}\cap \cdots \cap F_{i_{\ell}})\backslash (F_{i_{\ell+1}}\cup \cdots \cap F_{i_{d}}$), $\ell\leq n$, 
  where $i_1<\cdots <i_{\ell}$ and $\{i_1, \ldots, i_{d} \}= \{1, \dots ,d\}$,  
  and also that $F_{i_1}\cap \cdots \cap F_{i_n}$ is not empty. 
  
  As already noticed, $\big(du(p_k)\big)$ tends to $\infty$ in $\mathfrak{t}$; 
  thus, as $du_0$ is a diffeomorphism $\Delta\to \mathfrak{t}$ whose norm tends to $\infty$ near $\partial\Delta$, 
   \begin{equation*}
    (\underline{p}{}_k) := \big((d{u_0})^{-1}\circ du (p_k) \big)
   \end{equation*}
  tends to $\partial \Delta$. 
  We now prove that any limit point of $(\underline{p}{}_k)$ must belong to $\mathbf{f}$. 
  For this,  observe 
  that if $(\overline{p}_k)$ tends to $\partial \Delta$, 
  then $(\overline{p}_k)$ tends to $\mathbf{f}$ if and only if 
  $u_{0,1}(\overline{p}_k),\ldots, u_{0,\ell}(\overline{p}_k)$ tend to $-\infty$ 
  while $u_{0,\ell+1}(\overline{p}_k), \ldots, u_{0,n}(\overline{p}_k)$ remain bounded (we have set $u_{0,i}(x) = \frac{\partial u_0}{\partial x_i}(x)$). We thus want to prove that, as $k\to \infty$, 
   \begin{equation*}
    \left\{\begin{aligned}
     {u}_{0,1}(\underline{p}{}_k),    & \ldots, {u}_{0,\ell}(\underline{p}{}_k)  \longrightarrow -\infty,  \\
       {u}_{0,\ell+1}(\underline{p}{}_k),& \ldots, {u}_{0,n}(\underline{p}{}_k)   \,   = {O}(1), 
           \end{aligned}
    \right.
   \end{equation*}
  or equivalently
   \begin{equation}    \label{eqn_estmtes}
    \left\{\begin{aligned}
             u_{,1}(p_k),  &  \dots, u_{,\ell}(p_k)      \longrightarrow -\infty,  \\
             u_{,\ell+1}(p_k),&  \dots, u_{, n}(p_k)     \,   = {O}(1).
           \end{aligned}
    \right.
   \end{equation}
  Since $u\in \mathcal{S}_{\alpha, \beta}(\Delta, {\bf L}, F)$,  
  up to a smooth $\mathfrak{t}$-valued function  near $p$  we have
  \begin{equation*}
   u_{, j} =\left\{\begin{aligned}
              \Big(-\frac{\alpha}{L_1}+\beta\log L_1 \Big) &  &\text{ if }j=1,    \\
                                                 \frac{1}{2}\log L_{j} &    &\text{ if } j\geq 2,
             \end{aligned}
      \right.
  \end{equation*}
  which yield the estimates \eqref{eqn_estmtes}  and concludes the point (a) of our proof. 
  
 We now  address (b).  Let $p_k \to p$ and  ${\underline p}{}_k\to {\underline p}$.    
 Our task is to prove that $\underline{p}$ does not depend on $(p_k)$; if $p$ is a vertex, this already follows from (a), 
 so  we assume that $\mathbf{f}$ is an open face of $\Delta$.   
 Letting $u_{0, {\bf f}}: = (u_0)_{|_{\bf f}},$  $u_{\mathbf{f}}:=( u + \alpha \log L_1)_{|_{\mathbf{f}}}$ if $\mathbf{f}\subset F=F_1$  and 
 $u_{\mathbf{f}}=u_{|_{\mathbf{f}}}$ if $\mathbf{f}\not\subset F$, 
 the definitions of the spaces $\mathcal{S}(\Delta, {\bf L})$ and $\mathcal{S}_{\alpha, \beta}(\Delta, {\bf L}, F)$ ensure that   
 ${u}_{0,\mathbf{f}}$ and $u_{\mathbf{f}}$ are strictly convex on $\mathbf{f}$.   
 With the notations above, observe that the ${u}_{0,j}$, $j=\ell+1, \ldots, n$,  are smooth  in a neighbourhood of $\mathbf{f}$, 
 and ${u}_{0, j}= ({u}_{0,\mathbf{f}})_{,j}$ along $\mathbf{f}$;
  similarly,  the $u_{,j}$  are smooth around $\mathbf{f}$ and $u_{,j} = (u_{\mathbf{f}})_j$ along
 $\mathbf{f}$.    In this way, letting $k \to \infty$ in the equality $d{u_0}(\underline{p}{}_k) = du(p_k),$ we obtain
   \begin{equation*}
    (d{u}_{0,\mathbf{f}})(\underline{p}) = (du_{\mathbf{f}})(p),  
   \end{equation*}
  Using the strict convexity of   ${u}_{0,\mathbf{f}}$ and $u_{\mathbf{f}}$ on $\mathbf{f}$, and that 
  the norms of their differentials tend to $\infty$ near $\partial\mathbf{f}$,  
  we conclude that $d{u}_{0,\mathbf{f}}$ and $du_{\mathbf{f}}$ are diffeomorphisms $\mathbf{f}\to (\mathfrak{t}/\mathfrak{t}_{\mathbf{f}})^*$ 
  (see Lemma~\ref{lem_diffeo} above), and thus: 
   \begin{equation*}
    \underline{p}{} = (d{u}_{0,\mathbf{f}})^{-1} \circ du_{\mathbf{f}} (p), 
   \end{equation*}
  does not depend on $(p_k)$. 
 
 This completes the proof of Lemma \ref{lem_diffeo2}.  \end{proof}

The general theory~\cite{guillemin-book} 
(which uses local arguments around the  preimage of each face)  ensures that $\varphi(z)$ extends smoothly over $\C^n_v \setminus (\C^n_v \cap Z)$, 
and that $\omega_u$ defines a K\"ahler metric on $X \setminus Z$. 
We shall thus focus our analysis on $\C^n_v$ in order to understand the behaviour of $\omega_u$ near $Z\cap \C^n_v$, see \eqref{chart-divisor}.
 
 \smallskip
  Let $p\in Z \cap \C^n_v$. We shall consider the following limiting cases:
   \begin{enumerate}
    \item[\rm (a)]  $p=(0,\ldots, 0)$  corresponds to the vertex $v$ of $\Delta$;
    \item[\rm (b)]   $\mu_0(p)$ belongs to the relative interior of $F$, 
                     i.e. in the chart $\C^n_v$, $p$ has coordinates $(0, z_2, \ldots, z_n)$ with $z_j \neq 0, j=2, \ldots, n$. 
   \end{enumerate}
  The  case when $\mu_0(p)$ belongs to  an $\ell$-codimensional face of $F$ with $1\leq \ell\leq n-1$, 
  can be dealt with by combining the arguments for the cases (a) and (b).  
  We shall also assume at first that $\beta=0$. 
     
     ~
     
  \noindent
  \textit{Case (a): $p=(0, \ldots, 0)$.}  We can assume without loss that  the vertex  $\mu_0(p)=v$ of $\Delta$ is at the origin of $\tor^*\cong \R^n$, i.e. $L_j (x)= x_j, j=1, \ldots, n$.
   We thus have, near $0 \in \Delta$,   
       \begin{equation}\label{w}
      u(x)=-\alpha \log(x_1)+\frac{1}{2}\Big(\sum_{j=2}^m x_j\log(x_j) - x_j\Big) + w(x)
   \end{equation} 
 with $w(x)$ smooth on $\R^n$. We can further modify $u$ by adding an affine linear function (which does not change neither the induced K\"ahler metric nor the belonging of $u$ to $\mathcal{S}_{\alpha, \beta}(\Delta, {\bf L}, F)$) so that $du(x^0)=0$.  It then follows that in the chart $\C^n_v$ the  functions $y_j = u_{,j } = \log |z_j|$ are given by
 \begin{equation*}
 \begin{split}
 y_1 &= \log|z_1|=-\frac{\alpha}{x_1} + w_{,1}(x),\\
 y_j &= \log|z_j| = \frac{1}{2}\log x_j  + w_{,j}(x), \ j=2, \ldots, n,
 \end{split}
 \end{equation*}
 or equivalently,
 \begin{equation}    \label{eqn_Fx}
     \left\{\begin{aligned}
             -\frac{1}{\log|z_1|} &= f_1(x)  \quad,  \\
                    |z_j|^2 \quad   &= f_j (x)  \quad \text{for }j=2,\dots, n, 
            \end{aligned}
     \right.
    \end{equation}
with $f_1(x) =\frac{x_1}{\alpha - x_1w_{,1}(x)}$ and $f_j(x)=x_je^{2w_{,j}(x)}, j=2, \ldots, n.$ We want to use \eqref{eqn_Fx} in order to express the 
momentum coordinates $x=(x_1, \ldots, x_n)$ of $\omega_u$ in terms of $|z_j|'s$.  
To this end, we notice that $f(x)=(f_1(x), \ldots, f_n(x))$ extends smoothly  near the origin $v=0\in \Delta$. 
(Similarly, the LHS of \eqref{eqn_Fx} extends continuously on $\C^n_v$ by letting $\frac{-1}{\log|z_1|}$ be 0 along $z_1=0$.)  
Computing the Jacobian of $f(x)$ at $0$, we conclude that the local inverse $h(\zeta)$ of $f(x)$ is defined on a small neighbourhood of $0\in \R^n$  and has the form
\begin{equation}\label{h}
      h(\zeta)=\big(\alpha\zeta_1\tilde{h}_1(\zeta), \zeta_2\tilde{h}_2(\zeta), 
                     \ldots, \zeta_n\tilde{h}_n(\zeta)\big),
    \end{equation}
with $\tilde{h}_1(\zeta),\dots, \tilde{h}_n(\zeta)$ smooth and non-vanishing near $0$.   It thus follows from \eqref{eqn_Fx} that
    \begin{equation}   \label{eqn_Gxi}
     \left\{\begin{aligned}
             x_1 &= \Big(\frac{-\alpha}{\log|z_1|}\Big)\tilde{h}_1\Big(\frac{-1}{\log|z_1|},|z_2|^2,\dots,|z_n|^2\Big),            \\
             x_j &= |z_j|^2\tilde{h}_j\Big(\frac{-1}{\log|z_1|},|z_2|^2,\dots,|z_m|^2\Big)  \qquad \text{for }j=2,\dots, n. 
            \end{aligned}
     \right.
    \end{equation}
   By \eqref{kahler-potential},  and using that $y_j = \log |z_j|$ together with \eqref{eqn_Gxi}, we find that
    \begin{equation}\label{u-potential}
    \begin{split}
    \varphi_u(z)   =& \frac{1}{2}\sum_{j=1}^n \log(|z_j|^2)x_j - u(x) \\
                       =& - \alpha \log\big(-\log|z_1|\big)    +  W\Big(\frac{-1}{\log|z_1|},|z_2|^2,\dots,|z_n|^2\Big)
    \end{split}                                 
    \end{equation}
  where 
     \begin{equation}    \label{eqn_W}
    W(\zeta):= \alpha\log\big[\alpha\tilde{h}_1(\zeta)\big]-\alpha\tilde{h}_1(\zeta)
                          -\frac{1}{2}\Big(\sum_{j=2}^m \zeta_j\tilde{h}_j(\zeta)\big(\log \tilde{h}_j(\zeta) -1\big)\Big)-w\big[h(\zeta)\big]
    \end{equation}
for $w(x)$ defined in \eqref{w}. Thus, $W(\zeta)$  is a smooth function near $0$. 

A similar (and well-established) argument  using $u_0$ instead of $u$ shows that $\varphi_{u_0}$ can be viewed as a smooth function  $\psi_0(z)$ on $\C^n_v$, 
so that the relative potential $\varphi(z)$ in \eqref{phi} is written as
$$\varphi(z) = - \alpha \log\big(-\log|z_1|\big) 
              +\psi_0(z) + W\big(\frac{-1}{\log |z_1|},|z_2|^2,\dots,|z_n|^2\big)$$ 
 and thus has the required behaviour of a Poincar\'e type potential near $Z=\{z_1=0\}$, see Definition~\ref{d:Poincare-type}.  
   
  We now examine the asymptotic behaviour  of the K\"ahler form $\omega_u=dd^c \varphi_u$ near the origin $0 \in Z\cap \C^n_v$.  
  Using \eqref{u-potential}, we find that 
  \begin{equation}  \label{eqn_decomp1}
     \omega_u = dd^c\varphi_u = 2\alpha \frac{\sqrt{-1}dz_1\wedge d\overline{z_1}}{|z_1|^2\log^2(|z_1|^2)} +\Omega +\eta, 
    \end{equation}
   where $\Omega+\eta= dd^c\big[W\big(\frac{-1}{\log(|z_1|)},|z_2|^2,\dots,|z_n|^2\big)\big]$ are given by
    \begin{equation}   \label{eqn_OmegaEta}
     \left\{\begin{aligned}
             \Omega = & \sum_{j=2}^n W_{j}\Big(\frac{-1}{\log(|z_1|},|z_2|^2,\dots,|z_n|^2\Big) \, \sqrt{-1}dz_j\wedge d\overline{z_j}                        \\
                      & +\sum_{j,k=2}^n W_{jk}\Big(\frac{-1}{\log(|z_1|)},|z_2|^2,\dots,|z_n|^2\Big) \, \overline{z_j}z_k \sqrt{-1}dz_j\wedge d\overline{z_k}, \\
             \eta   = & W_1\Big(\frac{-1}{\log(|z_1|)},|z_2|^2,\dots,|z_n|^2\Big)\frac{ \sqrt{-1}dz_1\wedge d\overline{z_1}}{|z_1|^2\log^2(|z_1|)}          \\
                      & +\sum_{j=2}^n W_{1j}\Big(\frac{-1}{\log(|z_1|)},|z_2|^2,\dots,|z_n|^2\Big)
                                            \Big(\frac{z_j\sqrt{-1}dz_1\wedge  d\overline{z_j}}{z_1\log(|z_1|)}
                                                 +\frac{\overline{z_j} \sqrt{-1}dz_j\wedge d\overline{z_1}}{\overline{z_1}\log(|z_1|)}\Big),           
            \end{aligned}
     \right. 
    \end{equation} 
   and we have put
   \begin{equation*}
   \begin{split}
   W_1(\zeta) &=\zeta_1^2\Big(\frac{\partial^2W}{\partial \zeta_1 \partial \zeta_1}\Big)(\zeta) +\frac{1}{2}\zeta_1\Big(\frac{\partial W}{\partial \zeta_1}\Big)(\zeta), \ \ W_j(\zeta) = \Big(\frac{\partial W}{\partial \zeta_j}\Big)(\zeta), \ j\ge 2;\\
    W_{1j}(\zeta) &=-\zeta_1\Big(\frac{\partial^2 W}{\partial \zeta_1 \partial \zeta_j}\Big)(\zeta), \ j\ge 2; \ \   W_{jk}(\zeta) = \Big(\frac{\partial^2 W}{\partial \zeta_j \partial \zeta_k}\Big)(\zeta), \  j,k\ge  2. 
   \end{split}
   \end{equation*}
  Since $W_k(\zeta), W_{pq}(\zeta)$ are smooth,  $||\nabla^{s}\eta||={O}\big(\frac{1}{\log(|z_1|)}\big)$ near $Z$ for all $s\geq0$,  where $\nabla$ is the Levi--Civita connection of 
   the model Poincar\'e type metric
   \begin{equation*} 
 \omega_{\rm mod}=\frac{\sqrt{-1}dz_1\wedge d\overline{z_1}}{|z_1|^2\log^2(|z_1|^2)}+\sum_{j=2}^n \sqrt{-1}dz_j\wedge d\overline{z_j}
 \end{equation*} on $\C^n_v\backslash Z,$ 
  and the norms are computed with help of  $\omega_{\rm mod}$. 
   
  It follows from  \eqref{eqn_decomp1}--\eqref{eqn_OmegaEta} that $\omega_u$ has the Poincar\'e type behaviour in the normal $z_1$-direction, 
   as well as in the $(z_1,z_j)$-directions for $j\geq 2$.

   We are thus left  to examine the metric over the hyperplane $z_1=0$. 
   Letting $\tilde{W}_j(\zeta)$ and $\tilde{W}_{jk}(\zeta)$ be  the smooth functions  determined near $0$ by
      \begin{equation*}
  W_j(\zeta)=W_j(0,\zeta_2,\dots,\zeta_n)+\zeta_1\tilde{W}_j(\zeta); \ \  W_{jk}(\zeta)=W_{jk}(0,\zeta_2,\dots,\zeta_n)+\zeta_1\tilde{W}_{jk}(\zeta)
  \end{equation*}
   one has the decomposition $\Omega=\Omega_0+\varepsilon$, 
   where 
    \begin{equation}   \label{eqn_Omega}
     \left\{\begin{aligned}
             \Omega_0= & \sum_{j=2}^n W_{j}(0,|z_2|^2,\dots,|z_n|^2) \,\sqrt{-1}dz_j\wedge d\overline{z_j}                                                      
                                                                                                                           \\
                       & +\sum_{j,k=2}^n W_{jk}(0,|z_2|^2,\dots,|z_n|^2) \,\overline{z_j}z_k \sqrt{-1}dz_j\wedge d\overline{z_k},                               
            \\
             \varepsilon = & \frac{-1}{\log(|z_1|)}\sum_{j=2}^n \tilde{W}_{j}\Big(\frac{-1}{\log(|z_1|)},|z_2|^2,\dots,|z_n|^2\Big) 
                          \,\sqrt{-1}dz_j\wedge d\overline{z_j}       \\
                       & -\frac{1}{\log(|z_1|)}\sum_{j,k=2}^n 
                                                \tilde{W}_{jk}\Big(\frac{-1}{\log(|z_1|)},|z_2|^2,\dots,|z_n|^2\Big) 
                                                \,\overline{z_j}z_k \sqrt{-1}dz_j\wedge d\overline{z_k}.  
            \end{aligned}
     \right. 
    \end{equation}
We notice that $\Omega_0$ is smooth around the origin   whereas $\varepsilon$ satisfies,  
for all $s\geq0$, $||\nabla^{s}\varepsilon||={O}\big(\frac{1}{\log(|z_1|)}\big)$ near $Z$ 
(with covariant derivatives and norms taken with respect to the model Poincar\'e metric $\omega_{\rm mod}$). Computing the value of $\Omega_0$ at $z=0$, we get
     \begin{equation*}
     (\Omega_0)_{|_{z=0}} = \sum_{j=2}^m W_j(0) \, \sqrt{-1}dz_j\wedge d\overline{z_j}.
    \end{equation*}
 Using \eqref{eqn_W}, we have 
  \begin{equation}\label{comput}
    \begin{split}
     W_j(0) &= \Big(\frac{\partial W}{\partial \zeta_j} \Big)(0) \\
                 &= -\frac{1}{2}\tilde{h_j}(0)\log\big(\tilde{h_j}(0)\big)- \Big(\frac{\partial h_j}{\partial \zeta_j}\Big) (0)\Big(w_{,j}(0) -\frac{1}{2}\Big)   \\
                 &= -\frac{1}{2}\tilde{h}_j(0)\log\big(\tilde{h}_j(0)\big)- \tilde{h}_j(0)\Big(w_{,j}(0)-\frac{1}{2}\Big). 
    \end{split}
    \end{equation}
By the definition \eqref{h} (and inverting the diagonal Jacobian of $f$ at $0$) 
we have that for any $j=2, \ldots, n$, ${\tilde h}_j (0)= e^{-c_j}$ where $c_j=2w_{,j}(0)$.  
Substituting back to \eqref{comput}, we conclude
  \begin{equation*}
W_j(0)= e^{-c_j}\Big(-\frac{1}{2}(-c_j) - \frac{c_j-1}{2} \Big)  = \frac{e^{-c_j}}{2} >0,
\end{equation*} 
 thus showing  the positivity of $\omega_u$ in the directions parallel to $Z$.

    ~
    
  \noindent
  \textit{Case 2: $p=(0, z_2, \ldots, z_n)$ with $ z_j \neq 0$ for $j=2, \ldots, n$. } 
Now
$$u(x)=-\alpha\log x_1 +w(x)$$ with $w(x)$ smooth in a neighbourhood of $x_p=\mu(p) \in F^0$.  By assumption,    $u$ is strictly convex on $\Delta^0$, and $u_{F}= w_{|_F}$   is strictly convex on $F^0$. 
The relations \eqref{eqn_Fx}  now become
    \begin{equation}   \label{eqn_Fx2}
     \left\{\begin{aligned}
             \frac{-1}{\log(|z_1|)} &= f_1(x)  \quad,   \\
                    |z_j|^2 \quad   &= f_j (x)\quad \text{for }j=2,\dots, n, 
            \end{aligned}
     \right.
    \end{equation}
 with $f_1(x)= \frac{x_1}{\alpha-x_1w_{,1}(x)}$,  
and  $f_j(x)=e^{2w_{,j}(x)}$, $j=2,\dots,n$.  Using the strict convexity of $w$ along $F^0=\{x_1=0\}\cap \Delta$, 
we see that $f(x)=(f_1(x), \ldots, f_n(x))$ is smooth and locally invertible around $x_p=(0, b_2, \ldots, b_n)$. 
We denote by $h(\zeta)=(h_1(\zeta), \ldots, h_n(\zeta))$ the local inverse of $f$ around $x_p$, which must be of the form
 \begin{equation*}
    h(\zeta)= \big(\alpha\zeta_1\tilde{h}_1(\zeta), \tilde{h}_2(\zeta), \ldots, \tilde{h}_{n}(\zeta)\big)  
    \end{equation*}
  with ${\tilde h}_1(0) >0$.
 Thus,
 \begin{equation*}
   x = h\Big(\frac{-1}{\log|z_1|},|z_2|^2,\dots,|z_n|^2 \Big), 
    \end{equation*}
(which extends along $z_1=0$ near $p$). We obtain again
    \begin{equation}  \label{eqn_PhiuW}
     \varphi_u(z)= -\alpha\log\big(-\log|z_1|\big) +W\Big(\frac{-1}{\log|z_1|},|z_2|^2,\dots,|z_m|^2\Big), 
    \end{equation}
   with $W$ smooth  and given by
       \begin{align*}
     W(\zeta) =& \alpha\log\big(\alpha\tilde{h_1}(\zeta)\big)-\alpha\tilde{h_1}(\zeta) 
                               +\frac{1}{2}\sum_{j=2}^m \tilde{h_j}(\zeta)\log \zeta_j
                                -w\big[h(\zeta)\big], 
    \end{align*}
   (notice that $\zeta_j(p)=|z_j(p)|^2>0$ for $j=2,\ldots, n$). Thus,
   $$\varphi(z)= \varphi_u(z) - \varphi_{u_0}(z)= -\alpha\log(-\log|z_1|)- \psi(z)+ W\big(\frac{-1}{\log|z_1|},|z_2|^2,\dots,|z_n|^2\big), $$
   with $\psi$ and $W$ are smooth near $p$. 
   Consequently, $\varphi$  has the right asymptotics near $p$.

   \smallskip
   We  now address  the positivity  of $\omega_u$  near $p$.  
   Writing
    \begin{equation}  \label{eqn_decomp2}
     \omega_u = dd^c\varphi_u =\frac{2\alpha \sqrt{-1}dz_1\wedge d\overline{z_1}}{|z_1|^2\log^2(|z_1|^2)}+ \Omega+\eta,
    \end{equation}
   with $\Omega+\eta=dd^c\big(W\big(\frac{-1}{\log|z_1|},|z_2|^2,\dots,|z_n|^2\big)\big)$ given by  \eqref{eqn_OmegaEta}, we have that 
   $\eta={O}\big(\frac{1}{\log|z_1|}\big)$ at any order;  it is thus enough to show the positivity of $\Omega$ in \eqref{eqn_decomp2}. 
   Decomposing $\Omega=\Omega_0+\varepsilon$ with $\varepsilon =O\big(\frac{1}{\log|z_1|}\big)$ as in \eqref{eqn_Omega}, 
   we need to establish  the positivity of $(\Omega_0)_p$.  By its very definition, 
   $\Omega_0|_{Z}=dd^c\Big(W(0,|z_2|^2,\ldots,|z_n|^2)\Big)$. A careful examination of the definition of $W$ reveals that, 
   up to additive pluriharmonic terms $\log |z_j|^2$,  $W(0,|z_2|^2,\ldots,|z_n|^2)$, seen as a function on the hypersurface $\{z_1=0\}$,  coincides with 
   the K\"ahler potential corresponding to the Legendre transform \eqref{kahler-potential}  of the strictly convex function 
   $u_F(x) = \alpha \log x_1 + u(x)= w(x)$ restricted to the relative interior of $F$. Thus, $(\Omega_0)_p>0$.

   \smallskip
   We finally comment on the case when  $u\in \mathcal{S}_{\alpha, \beta}(\Delta, {\bf L}, F)$ with $\beta\neq 0.$ 
   The main difficulty is that the equations \eqref{eqn_Fx} hold with
    \begin{equation}  \label{eqn_F1}
     f_1(x)= \frac{x_1}{\alpha-\beta x_1\log x_1 - x_1(\beta+w_{,1}(x))}
    \end{equation}
which  is no longer smooth (nor even $C^2$)  around $x_1=0$.
 
  One way to bypass this difficulty is to use a suitable change of variables. 
  We detail below the case $n=1$, for the general case is treated similarly by considering the change of variables with respect to $|z_1|^2$ 
  and leaving the variables $|z_j|^2, j =2, \ldots, n$ unchanged.
  
  We  set 
  $$s:=\frac{-1}{\log x_1}, \  \ t_1:=\log(-\log|z_1|),$$
  so that  the first equation in \eqref{eqn_Fx} becomes
    \begin{equation} \label{eqn_Fs}
     \frac{1}{t_1}
         = \frac{s}{1+s\log\big(\alpha+\beta e^{-1/s}/s + e^{-1/s}(\beta+w_{,1}(e^{-1/s}))\big)}
         =: f(s). 
    \end{equation}
    As $w_{,1}$ is smooth in a neighbourhood of $0$ and the functions 
    \begin{equation*}
     s\longmapsto \left\{\begin{aligned}
                          & 0        &\text{ if }s\leq 0,\\
                          & e^{-1/s} &\text{ if }s > 0,
                         \end{aligned}
                  \right.
     \qquad \text{and} \qquad 
     s\longmapsto \left\{\begin{aligned}
                          & 0                    &\text{ if }s\leq 0,\\
                          & \frac{1}{s} e^{-1/s} &\text{ if }s > 0,
                         \end{aligned}
                  \right.
    \end{equation*}
   are smooth on $\R$,  we can extend $f(s)$ as a smooth function in a neighbourhood of $0$ by  letting $f(s)=\frac{s}{1-s\log\alpha}$ for $s\leq 0$. 
   
   As $\partial_sf(0)= 1$, we get that $s=h(\frac{1}{t_1})$ for some $h$ smooth around 0,  satisfying $h(0)=0$, $h'(0)=1$.  Thus, 
    $s=\frac{1}{t_1}\big(1+\frac{\gamma}{t_1}+{\tilde h}(\frac{1}{t_1})\big)$ for some constant $\gamma$, 
   and $\tilde h$ a smooth function vanishing at order $2$ at $0$. In fact, one  must have  $\gamma=\log\alpha$, 
   and thus $$\frac{1}{s}=t_1-\log\alpha+H\Big(\frac{1}{t_1}\Big), $$
   with $H$ a smooth function vanishing at order (at least) $1$  at $0$. 
  We can be more precise, and rewrite \eqref{eqn_Fs} as
    \begin{equation*}
     \frac{1}{t_1}=\frac{1}{1/s+\log\alpha+P(s)}, \quad\text{i.e.} \quad \frac{1}{s}=t_1-\log\alpha-P(s)
    \end{equation*}
   with $P(s)=\log\big(1+(\beta/\alpha)e^{-1/s}/s+e^{-1/s}(\beta+w_{,1}(e^{-1/s}))/\alpha\big)$. 
   Replacing $1/s$ with  $t_1-\log\alpha+H\big(\frac{1}{t_1}\big)$ in  the explicit expression of $P(s)$, 
   we  see by induction that $H(\frac{1}{t_1})$, 
   as well as  its derivatives with respect to $t_1$ at any order,  are ${O}(t_1e^{-t_1})$ when $t_1\to +\infty$.  
    Therefore, 
    \begin{align*}
     x= e^{-1/s} = \exp\big(-t_1 +\log\alpha + P(s)\big)
         = \frac{-\alpha}{\log|z_1|} \big(1+O(t_1e^{-t_1})\big), 
    \end{align*}
   at any order with respect to differentiation in $t_1$. 
   In particular,  
    $x=\frac{-\alpha}{\log|z_1|} + {O}(t_1e^{-2t_1})$ (instead of ${O}(e^{-2t_1})$ in the case $\beta=0$)
   at any order with respect to $\frac{|dz_1|^2}{|z_1|^2\log^2(|z_1|^2)}$.  We thus  get  for the relative potential \eqref{phi}
   $$\varphi = -\log\big(-\log|z_1|\big) + \log \alpha + O(t_1e^{-t_1}),$$
   and recover the asymptotic behaviour $\omega_u$ 
   near $Z$ with arguments identical with the ones in the case $\beta=0$.     
   This ends the proof of Theorem~\ref{thm:criterion}. 
    
 \subsection{Proof of Proposition~\ref{complete-first-order}}\label{Ealpha} 
 The arguments are local,  near a point $p\in \partial\Delta$, and not materially different from ones in \cite[\S 1.3]{ACGT}. 
 In fact, we need to only consider the case when $p\in F=F_1$ (otherwise the result follows from \cite{ACGT}). 
 To this end, we fix a vertex $v\in F$ of $\Delta$, which without loss can be taken to be at the origin of $\tor^*$, 
 and consider a basis of $\tor$ corresponding to the inward normals of the $n$ facets $F_1, \ldots, F_n$ meeting at $v$. 
 We can also assume that $F=F_1$ is defined by the equation $x_1=0$, i.e. $p=(0, x_2, \ldots, x_n)$ with $x_j \ge 0,$  and $L_j(x)=x_j$.
  
In one direction,  we want to show that  if $u \in \mathcal{S}_{\alpha, \beta}(\Delta, {\bf L}, F)$ 
then $u$ satisfies the four conditions of Proposition~\ref{complete-first-order}. 

Let us  define the mutually inverse matrices
   \begin{equation*}
  \mathbf{G}^0_{\alpha, \beta}=  \begin{pmatrix}
     \frac{\alpha + \beta x_1}{x_1^2} & 0              &  \cdots   &   0             \\
                       0                      & \frac{1}{2x_2} &  \ddots   &   \vdots        \\
                     \vdots                   &  \ddots        &  \ddots   &   0             \\
                       0                      &  \cdots        &    0      &   \frac{1}{2x_n} 
    \end{pmatrix}
    \quad\text{and}\quad
   \mathbf{H}^0_{\alpha, \beta} = \begin{pmatrix}
     \frac{(x_1)^2}{\alpha+\beta x_1}         &    0           &  \cdots   &   0             \\
                       0                      &   2x_2         &  \ddots   &   \vdots        \\
                     \vdots                   &  \ddots        &  \ddots   &   0             \\
                       0                      &  \cdots        &    0      &   2x_n 
    \end{pmatrix}.
   \end{equation*}
 Given $u\in \mathcal{S}_{\alpha, \beta}(\Delta, {\bf L}, F_1)$, 
 we first prove that  $\mathbf{G}^u ={\rm Hess}(u)$ and its inverse $\mathbf{H}^u$ satisfy the property that $\mathbf{G}^u-\mathbf{G}^0_{\alpha, \beta}$ and 
 $\mathbf{G}^0_{\alpha, \beta}\mathbf{H}^u\mathbf{G}^0_{\alpha, \beta}-\mathbf{G}^0_{\alpha, \beta}$ 
 extend smoothly through $(F_1\cup \cdots \cup F_n)\backslash(F_{n+1}\cup \cdots \cup F_d)$ 
 in the region $\{\alpha+\beta x_1>0\}$, see  Figure \ref{fig_1}.  
 Then we will show that if $\mathbf{G}^u-\mathbf{G}^0_{\alpha, \beta}$ and 
 $\mathbf{G}^0_{\alpha, \beta}\mathbf{H}_u\mathbf{G}^0_{\alpha, \beta}-\mathbf{G}^0_{\alpha, \beta}$ 
 extend smoothly, then the conditions of Propostion~\ref{complete-first-order} must be satisfied.
 
  \selectlanguage{french}
 
   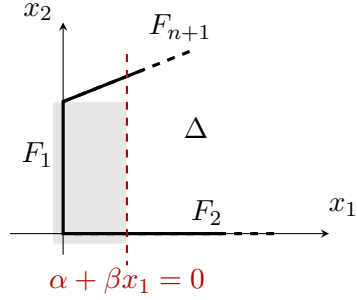
\begin{figure}[!ht]  
    \begin{center}
    \begin{tikzpicture}[scale=.7] 
             \draw [color=white, fill=gray!20, rounded corners=1pt] (-.2,1) -- (-.2,-.2) -- (1.2,-.2) -- (1.2, 2.5) -- (-.2,2.5) -- (-.2,1);
             (1.2,3.18) -- (-.2,2.62) -- (-.2,1);
             \draw (4.8,0.5)  node [right] {$x_1$};
             \draw[>=stealth, ->] (-1,0) -- (5,0); 
             \draw[>=stealth, ->] (0,-.4) -- (0,4); 
             \draw (0,4.2)  node [left] {$x_2$};
             \draw [very thick] (1.5,3.1) -- (0,2.5) -- (0,0) -- (3,0);
             \draw [very thick, dashed] (.2,0) -- (4,0);
             \draw [very thick, dashed] (0,2.5) -- (2.5,3.5);
             \draw [dashed, thick, color=black!30!red] (1.2,3.4) -- (1.2,-.6);
             \draw (0,1.5)  node [left] {$F_1$};
             \draw (2.7,.4)  node {$F_2$};
             \draw (2.2,3.9)  node {$F_{n+1}$};
             \draw [color=black!30!red] (1.2,-.9) node  {$\alpha +\beta x_1=0$};
             \draw (2.5,2)  node {$\Delta$};
       \end{tikzpicture}
      \caption{The polytope ${\Delta}$, and the domain near the face $F_1$ (grey)}
      \label{fig_1}
      \end{center}
     \end{figure}
 
\selectlanguage{english}
 
We notice that  $\mathbf{G}^0_{\alpha,\beta}={\rm Hess}(\vzero)$ with $$\vzero=-(\alpha-\beta x_1)\log x_1+\frac{1}{2}\sum_{j=2}^n x_j\log x_j.$$ 
Thus,  $\mathbf{G}^u-\mathbf{G}^0_{\alpha,\beta}$  is the Hessian 
 of a smooth  function  through the wall $(F_1\cup \cdots \cup F_n)\backslash(F_{m+1}\cup \cdots \cup F_d)$.  
 Writing $u=-(\alpha-\beta x_1)\log x_1+\frac{1}{2}\sum_{j=2}^n x_j\log x_j +w(x)$, 
with $w$ smooth,  we get:
  \begin{equation*}
   \Hzero \mathbf{G}^u = I_n +\begin{pmatrix}
                               \frac{x_1^2}{\alpha x_1+\beta}w_{,11} & \cdots & \frac{(x_1)^2}{\alpha x_1+\beta}w_{,1n} \\
                                2x_2w_{,21}                          & \cdots &  2x_2w_{,2n}                            \\
                                  \vdots                             &        & \vdots                                  \\
                                2x_nw_{,2n}                          & \cdots & 2x_nw_{,2n}
                              \end{pmatrix}
  \end{equation*}
 which clearly extends smoothly (with positive determinant over the origin). 
 Moreover, on $\mathbf{f}:=(F_1\cap\cdots\cap F_{\ell})\backslash (F_{\ell+1}\cup\cdots\cup F_m)$, for some $\ell\in\{1,\ldots,  n\}$
  \begin{equation*}
   \Hzero \mathbf{G}^u = I_n +\begin{pmatrix}
                                       0                              & \cdots &          0                             \\
                                  \vdots                              &        & \vdots                                 \\
                                              0                       & \cdots &          0                             \\
                              2x_{\ell+1}w_{,(\ell+1)1}               & \cdots &  2x_{\ell+1}w_{,(\ell+1) m}            \\
                                  \vdots                              &        & \vdots                                 \\
                                2x_nw_{,2n}                           & \cdots & 2x_nw_{,2n}
                              \end{pmatrix}
                       = \begin{pmatrix}
                           I_{\ell}  &                                     0                                             \\
                              *      & \mathbf{H}^0_{\mathbf{f}}\mathbf{G}{u_{\mathbf{f}}}
                         \end{pmatrix}
  \end{equation*} 
 where $\mathbf{H}^0_{\mathbf{f}}={\rm diag}(2x_{\ell+1}, \ldots, 2x_n)$  and, we recall,  $u_{\mathbf{f}}:= \big( u +(\alpha-\beta x_1)\log x_1)_{|_{\mathbf{f}}}$. 
 Hence $\det[\Hzero \mathbf{G}^u]=2^{m-\ell}x_{\ell+1}\cdots x_n \det[\mathbf{G}^{u_{\mathbf{f}}}]>0$ along $\mathbf{f}$, 
 as $\mathbf{G}^{u_\mathbf{f}}$ is positive definite by assumption.  Since this holds for all $\mathbf{f}$ and on the vertex $F_1\cap\cdots\cap F_n$, 
we conclude that  $\Hzero \mathbf{G}^u$ admits a smooth inverse, i.e.  $\mathbf{H}^u\Gzero =(\Hzero \mathbf{G}^u)^{-1}$ can be extended smoothly 
 through the wall too. Therefore, $(\Gzero-\mathbf{G}^u)\mathbf{H}^u\Gzero= \Gzero\mathbf{H}^u\Gzero-\Gzero$ extends smoothly as well. 
 
 Set $\mathbf{Q}^u:=\Gzero\mathbf{H}^u\Gzero-\Gzero$.   
 It follows that $\mathbf{H}^u = \Hzero\mathbf{Q}^u\Hzero+\Hzero$, hence ${\mathbf H}^u$ is smooth through the wall.  
 In fact,  a direct computation tells us that 
   \begin{equation}  \label{eqn_matrixO}
   \mathbf{H}^u - \Hzero =\begin{pmatrix}
    x_1^4 f_{11}      & x_1^2 x_2   f_{12}  & \cdots         & \cdots            & x_1^2x_n f_{1n}            \\ 
    x_1^2 x_2 f_{21}  &       x_2^2  f_{22} & x_2x_3  f_{23} & \cdots            & x_2x_n f_{2m}              \\
            \vdots    &      x_2x_3  f_{32} &                &                   & \vdots                     \\
            \vdots    &         \vdots      &                &                   & x_{n-1}x_n f_{2n}          \\
   x_1^2 x_n f_{n1}   &      x_2x_n  f_{n2} &   \cdots       & x_{n-1}x_n f_{2n} & x_n^2  f_{nn}              \\
   \end{pmatrix}
  \end{equation}
 with smooth $f_{ab}$, 
 which allows us to see that all the boundary conditions of  $\mathbf{H}^u$ are satisfied along a face ${\bf f} \subset F_1$.

 It remains to check the positivity assertion on the faces ${\mathbf f} \subset F_1$; 
 as above, suppose $\mathbf{f} =(F_1\cap\cdots\cap F_{\ell})\backslash (F_{\ell+1}\cup\cdots\cup F_d)$. 
 Along $\mathbf{f}$, 
 $\Hzero\Gv$ can be written as 
 $\big(\begin{smallmatrix} I_{\ell} & 0 \\ * & \mathbf{H}_{\mathbf{f}}^0\mathbf{G}^{u_{\mathbf{f}}} \end{smallmatrix}\big)$.  
 It thus follows that  along $\mathbf{f}$ 
 $(\Hzero\Gv)^{-1}$ has the shape 
 $\big(\begin{smallmatrix} I_{\ell} & 0 \\ * &  \mathbf{H}^{u_{\mathbf{f}}}\mathbf{G}_{\mathbf{f}}^0 \end{smallmatrix}\big)$,  
with  $\mathbf{G}^0_{\mathbf{f}}=(\mathbf{H}_{\mathbf{f}}^0)^{-1}= {\rm diag}( \frac{1}{2x_{\ell+1}}, \ldots, \frac{1}{2x_n})$, 
where ${\bf H}^{u_{\bf f}} = ({\bf G}^{u_{\bf f}})^{-1}$.
Thus, along $\bf f$, $\Hv=(\Hzero\Gv)^{-1}\Hzero=\big(\begin{smallmatrix}  0 & 0 \\ * & \mathbf{H}^{u_{\mathbf{f}}} \end{smallmatrix}\big)$ 
 (since $\Hzero=\big(\begin{smallmatrix}  0 & 0 \\ 0 & \mathbf{H}_{\mathbf{f}}^0 \end{smallmatrix}\big)$ along $\mathbf{f}$).
 The desired positivity now readily follows from that of $\mathbf{H}^{u_{\mathbf{f}}}$ along $\mathbf{f}$, 
 which in turn is a direct consequence of the convexity assumption of $u_{\bf f}$.

 \smallskip
 
We now deal with the converse direction of Proposition~\ref{complete-first-order},  
i.e.  given a strictly convex $u\in C^{\infty}(\Delta^0)$ such that the associated $\Hv$ verifies the conditions of  Proposition~\ref{complete-first-order}, 
 we have to show that $u\in \mathcal{S}_{\alpha, \beta}(\Delta, {\bf L}, F_1)$.  
 Again, as this is local and already known far from the Poincar\'e face $F_1$, 
 we focus on the same region as above. Arguments analogous to those in \cite[pp. 290-291]{ACGT} allow one to show that the boundary conditions for ${\bf H}^u$ yield that 
 $\mathbf{G}^u-\mathbf{G}^0_{\alpha, \beta}$  and 
 $\mathbf{H}^u\mathbf{G}^0_{\alpha, \beta}$ 
 extend smoothly through $(F_1\cup \cdots \cup F_n)\backslash(F_{n+1}\cup \cdots \cup F_d)$, the later having positive determinant on $F_1$; 
 the smooth extension of $\mathbf{G}^u-\mathbf{G}^0_{\alpha, \beta}$ ensures that $u$ can be written as 
  \begin{equation*} 
   -(\alpha-\beta L_1)\log\ L_1+\frac{1}{2}\sum_{j=2}^d L_j\log L_j + w(x)
  \end{equation*}
 for some $w \in C^{\infty}\big({\Delta}, \R \big)$.

 The boundary conditions  for ${\bf H}^u$ also tell us that 
 $\Hv=\Hzero+\mathbf{R}^u$, with $\mathbf{R}^u$ a smooth matrix of shape given by \eqref{eqn_matrixO} 
 (the {third order boundary condition} on the Poincar\'e face gives precisely the ${O}\big(x_1^4\big)$  estimate for the upper left coefficient of $\mathbf{R}^{u}$). 
 Hence,  $\Hv\Gzero=I_n+\mathbf{R}^u\Gzero$ with
  \begin{equation}  \label{eqn_matrixRG}
  \mathbf{R}^u\Gzero= \begin{pmatrix}
    (x_1)^2 g_{11}  & (x_1)^2    g_{12} & \cdots         & \cdots        & (x_1)^2 g_{1n}    \\ 
     x_2 g_{21}     &  x_2  g_{22}      &  x_2   g _{23} & \cdots        & x_2     g_{2n}    \\
            \vdots  &      \vdots       &                &               & \vdots            \\
            \vdots  &         \vdots    &                &               & x_{n-1} g_{(n-1)n}\\
     x_n g_{n1}     &  x_n  g_{n2}      &   \cdots       & x_n g_{n,(n-1)} &  x_n    g_{nn}    
   \end{pmatrix}, 
  \end{equation}
 for smooth $g_{ab}$. It follows that
 along $\mathbf{f}=(F_1\cap\cdots\cap F_{\ell})\backslash(F_{\ell+1}\cup\cdots\cup F_d)$ ($1\leq \ell\leq n$), 
 $\Hzero\Gv=(\Hv\Gzero)^{-1}=\big(\begin{smallmatrix} I_{\ell} & 0 \\ * & \mathbf{H}_{\mathbf{f}}^0\mathbf{G}^{u_{\mathbf{f}}} \end{smallmatrix}\big)$. 
As $\Hv\Gzero$ extends as a smooth and positive definite matrix over ${\bf f}$, 
 we conclude that $\mathbf{G}^{u_{\mathbf{f}}}>0$, i.e. $u_{\bf f}$ is strictly convex in the relative interior of $\bf f$.

\section{Proof of Proposition~\ref{stable-Hirzebruch} and Corollary~\ref{Hirzebruch-non-poincare}}\label{AppendixB}

For the proof of Proposition~\ref{stable-Hirzebruch} and Corollary~\ref{Hirzebruch-non-poincare} 
we will change point of view slightly and use explicit computations for quadrilaterals in $\mathbb{R}^2$ with the standard lattice $\mathbb{Z}^2$. 
The moment polytope corresponding to a Hirzebruch surface has at least one pair of parallel edges. 
Therefore after possibly scaling the polytope, we can take it to be given as the intersection $\cap_{i=1}^4 L_{i}^{-1}([0, \infty))$ of 
\begin{align}\label{hirzebrucheqns} L_1(x,y) &= y,\nonumber \\
 L_2(x,y) &= (1-x),\\
 L_3(x,y) &= (q-k)x- y + k,\nonumber \\
L_4(x,y) &= x \nonumber,
\end{align}
for some positive real numbers $q$ and $k$. This will then correspond to a Hirzebruch surface exactly when $q - k \in \mathbb{Z}$.

We begin with the proof of Proposition~\ref{stable-Hirzebruch} in the case of two edges, 
then show that the case of one edge is a corollary of this, using the convexity of the set of stable weights. 
We end by proving Corollary~\ref{Hirzebruch-non-poincare}.
\subsection{The case of two edges} For this we will use a criterion for stability found in \cite{lars} 
which used the ambitoric framework of \cite{ACG1} and \cite{ACG2} described above. We begin by recalling this result.

For a pair of edges $F_1,F_2$ of a general $2$-dimensional convex polytope $\Delta$, one can parametrize the lines that meet both $F_1$ and $F_2$ by $[0,1] \times [0,1]$. 
Let the vertices of $F_1$ be $v_0$ and $v_1$ and let the vertices of $F_2$ be $w_0$ and $w_1$. Then let $v_s = (1-s) v_0 + s v_1$ and $w_t = (1-t)w_0 + t w_1$. 
Picking an affine linear function $h_{s,t}$ whose zero set is the line containing $v_s$ and $w_t$, 
one then obtains a corresponding simple piecewise linear function $f_{s,t} = \max \{ 0, h_{s,t} \}$. 
We can parametrize the Donaldson-Futaki invariant of these functions as a map 
\begin{align*} \phi : [0,1]\times[0,1] &\longrightarrow \mathbb{R}, \\
(s,t) \ \ \ \ \                        &\longmapsto \mathcal{L} ( f_{s,t} ).
\end{align*}
The positivity of this is independent of the scaling of $h_{s,t}$ chosen. 
Since all lines meeting $F_1$ and $F_2$ are traced out as $(s,t)$ takes all values in $[0,1] \times [0,1]$, 
it suffices to check the positivity of $\phi$ in order to check whether or not there are any simple piecewise linear functions 
with crease meeting $F_1$ and $F_2$ violating stability. 
By choosing an appropriate scaling of $h_{s,t}$, $\phi$ can be taken to be polynomial in $(s,t)$ of bidegree $(3,3)$. 
This was essentially shown in \cite{Do2}, see also \cite[Lem. 2.10]{lars}.

If $F_1$ and $F_2$ are adjacent to a common edge $\tilde{F}$ then each of the $F_i$ has a vertex lying on $\tilde{F}$. 
Thus, up to reordering the $v_i$ and $w_i$, we have that $v_0$ and $w_0$ lie on $\tilde{F}$. 
In the domain $[0,1] \times [0,1]$ of $\phi$ the point $(0,0)$ will then correspond to 
the simple piecewise linear function $f_{0,0}$ whose crease is $\tilde{F}$; $f_{0,0}$ is then actually affine linear on $P$ and so $\phi \big(( 0,0) \big) =0$. 
Moreover, if $\tilde{F}$ is one of the components of $F$, the facets along which we let the boundary measure vanish, then the point will be a critical point of $\phi$.

In the case when $\Delta = Q$ is a quadrilateral, there are exactly two pairs of such edges that have common adjacent edges, 
namely the two pairs of opposite edges. 
As above, we then have functions $\phi_1, \phi_2$ parametrizing the Donaldson-Futaki invariant of simple piecewise linear functions meeting opposite edges of $Q$. 
In the case when $F$ consists of two edges of $Q$, there will then be exactly two critical points corresponding to two of the vertices of $\phi_1$ and/or $\phi_2$.

\begin{prop}\label{detcond} Let $Q$ be a quadrilateral and pick two edges $F_1, F_2$ of $Q$. 
Then $(Q, {\bf L}, F_1 \cup F_2)$ is stable if and only if 
the determinant of the Hessian of the functions $\phi_1,\phi_2$ at the points corresponding to an affine linear function is
\begin{itemize}
\item non-negative if $F_1, F_2$ are adjacent,
\item positive if $F_1, F_2$ are opposite.
\end{itemize}
Moreover, the positivity of the determinant implies that the relative Sz\'ekelyhidi numerical constraint is satisfied.
\end{prop}
Note that the converse of the final statement is not quite true. 
If the determinant vanishes, the Hessian is positive semi-definite but not positive definite at the critical point of $\phi_i$. 
This means that there is a family $f_c$ of simple piecewise linear functions, 
with $f_0$ corresponding to the critical point of the domain of $\phi_i$, such that $\frac{d^2}{dc^2}\big|_{c=0} \big( \mathcal{L}(f_c) \big) = 0$. 
However, it is not necessarily the case that the crease of this family can be taken to be parallel to relevant edge of $Q$. 
This would have to be the case if the positivity of the determinant was equivalent to the relative Sz\'ekelyhidi numerical constraint.

For $Q$ being the moment polytope of a Hirzebruch surface given by equation \eqref{hirzebrucheqns} and ${\bf L}$ the canonical scaling of the normals to $Q$, 
one can then compute the functions $\phi_1$ and $\phi_2$ of Proposition \ref{detcond} and hence their determinants directly in terms of $q$ and $k$. 
The result of this computation is given in Lemmas \ref{adjcomp} and \ref{oppcomp} below.
\begin{lemma}\label{adjcomp} Suppose $F$ consists of two adjacent edges, which without loss of generality can be assumed to be the two edges not lying on the coordinate axes. 
Then the determinants of the Hessians of $\phi_1, \phi_2$ at the two critical points are up to a positive constant given by 
\begin{align}\label{1stdet} k^4 + 2 k^2 q^2 + q^4 - k^3 + 3k^2 q + 3 kq^2 - q^3
\end{align}
and 
\begin{equation}\label{2nddet} 
\begin{aligned}3k^6 q + 3k^5 q^2 + 6k^4q^3 + 6 k^3 q^4 + 3k^2 q^5 + 3k q^6 + 2k^6 + 2k^5q + 6k^4 q^2 \\ 
               + 4 k^3 q^3 + 6 k^2 q^4 + 2 k q^5 + 2 q^6 - 2 k^5 - 2 k^4 q + 4 k^3 q^2 + 4k^2 q^3 - 2 q^4 k - 2 q^5.
\end{aligned}
\end{equation}
\end{lemma}
Thus to complete the proof of Proposition~\ref{stable-Hirzebruch} in the case when $F$ consists of two adjacent edges, 
we have to show that both these numbers are always positive. This is not true for arbitrary positive $q$ and $k$, 
but we will use that $q-k \in \mathbb{Z}$. 
In fact, having $|q-k| \geq 1$ ensures that both \eqref{1stdet} and \eqref{2nddet} are positive. 
Note that we do not have to consider the case $q=k$ as this corresponds to a product.

So assume first that $q>k$, so that $q \geq k+1$. We then use the substitution $q = k+1+\lambda$, where $\lambda \geq 0$ by assumption. 
The expression \eqref{1stdet} is then given by 
\begin{align*}\lambda^4 + 4 k \lambda^3 + 8 k^2 \lambda^2 + 8 k^3 \lambda + 4 k^4 + 3 \lambda^3 + 12 k \lambda^2 \\ 
+ 22 k^2 \lambda + 12 k^3 + 3 \lambda^2 + 12 k \lambda +14 k^2 + \lambda + 4 k .
\end{align*}
Since $k>0$ and $\lambda \geq 0$ it therefore follows that the term in equation \eqref{1stdet} is always positive.

Similarly, using the substitution $k=q+1+\lambda$ instead, one can show that \eqref{1stdet} is always positive when $q<k$. 
The same technique also works to show that \eqref{2nddet} is positive whenever $q-k \in \mathbb{Z}$. 
This completes the proof of Proposition \ref{stable-Hirzebruch} for the case of adjacent edges.

Though we have already proved this by different means, the above technique also works when $F$ consists of two opposite edges. 
Sz\'ekelyhidi showed in \cite[Prop. 15]{szekelyhidi08} that if the two edges that are not in $F$ are parallel, then $(Q,{\bf L} , F)$ is always strictly semistable. 
It can also be verified directly that the determinant of the Hessian as in Proposition \ref{detcond} vanishes in this situation. 
With $Q$ determined by positive numbers $k,q$ as above, we can therefore assume that $F$ consists of the two edges of $Q$ contained in $L_2 =0$ and $L_4=0$, respectively, 
since the edges lying in $L_1=0$ and $L_3=0$  are the only opposite edges that may not be parallel. 

In the case of opposite edges, the two determinant conditions turn out to be equivalent. 
Thus we need to determine that this single number is non-negative and vanishes precisely if $q=k$. This is a consequence of the Lemma below.

\begin{lemma}\label{oppcomp} Suppose $F$ consists of the two opposite edges lying on $L_2=0$ and $L_4=0$. 
Then the determinant of the Hessian of the function corresponding to the Donaldson-Futaki invariant 
of simple piecewise linear functions with crease meeting $L_1=0$ and $L_3=0$ at the critical point corresponding to $L_2=0$ is given by 
\begin{align*} \frac{(k-q)^2 (k+q)^2k^2}{2(k^2+4kq+q^2)^2}.
\end{align*}
\end{lemma}
The determinant is thus always non-negative and since $k$ and $q$ are positive it vanishes if and only if $k=q$, i.e. if and only if $Q$ is a rectangle, as expected. 

\begin{rem}\label{3facetsproof} The determinant condition Proposition \ref{detcond} holds regardless of the normals we use for the remaining two edges in $F$. 
In particular, it applies when we have a third facet in $F$. In this case similar formulae to the ones given above show that the determinant condition is violated, 
and so $(Q,F)$ is always unstable when $F$ consists of three edges of $Q$.
\end{rem}

\subsection{The case of one edge}

To prove that $(Q,{\bf L} , F)$ is stable when $Q$ corresponds to a Hirzebruch surface and $F$ is a single edge of $Q$, 
we will use the notion of weighted stability and the convexity of the set of stable weights. 
In general, for a Delzant polytope $\Delta$, 
we let $r_i \in \mathbb{R}_{\geq 0}$ be the reciprocal of the scaling of the $i^{\textnormal{th}}$ defining function of $\Delta$ 
given by the data ${\bf L}$ if this facet is not in $F$ and $r_i =0$ if it is. 
Then we can identify the triple with $(\Delta, {\bf L}, F)$ with $(\Delta, \underline{r})$ where $\underline{r} \in \mathbb{R}^d_{\geq 0}$ 
is the \textit{weight} of $(\Delta, {\bf L}, F)$. The weight $\underline{r}$ is \textit{stable} if the corresponding triple $(\Delta, {\bf L}, F)$ is. 

For us the key property of weighted stability is that the set of non-zero stable weights, thought of as a subset of $\mathbb{R}_{\geq 0}^d \setminus \{ 0 \}$, 
is a \textit{convex cone}. Thus a positive linear combination of semistable weights is semistable, 
and moreover, if at least one of the weights is stable, then the linear combination is stable too.

Going back to the case when $\Delta = Q$ corresponds to a Hirzebruch surface, 
the notion of the stability of $(Q,{\bf L} , F)$ where $F$ is a single edge of $Q$ is exactly the same as the stability of the weight $(0,1,1,1)$, 
where $F_1$ is the edge in $F$ and $F_2, F_3, F_4$ are the remaining three edges of $P$.

We now note that this weight can be written as 
\begin{align*} (0,1,1,1) = \frac{1}{2} (0,0,1,1) +  \frac{1}{2} (0,1,0,1) + \frac{1}{2} (0,1,1,0) .
\end{align*} 
The stability of the weights $ (0,0,1,1) , (0,1,0,1)$ and $(0,1,1,0)$ each correspond to the stability of some $(Q,{\bf L} , F')$ 
where $F'$ consists of exactly two edges of $Q$. 
Thus $(0,1,1,1)$ is a positive combination of semistable weights, hence is semistable. 
Moreover, at least two of the weights are then in fact stable, since two of these weights correspond to when $F'$ consists of two adjacent sides. 
Hence $(0,1,1,1)$ must be stable as well. 
This means that $(Q,{\bf L} , F)$ where $Q$ is a Hirzebruch surface, ${\bf L}$ are the canonical defining functions of $Q$ and $F$ is a single edge of $Q$, is always stable, 
and this completes the proof of Proposition~\ref{stable-Hirzebruch}.

\subsection{The proof of Corollary~\ref{Hirzebruch-non-poincare}}

The missing component in the proof is to show that the extremal metrics obtained cannot be of Poincar\'e type. 
For this we will show that in the above situation, the necessary condition in equation~\eqref{hyperbolicity-condition} 
for the existence of a toric extremal metric of Poincar\'e type is violated. 
In the Hirzebruch surface cases we are interested in, this condition becomes the following: 
\begin{lemma}\label{PTrestriction} Let $Q$ be the moment polytope of a Hirzebruch surface $X$ and K\"ahler class $\Omega$ given by the data in equations~\eqref{hirzebrucheqns}. 
If $X$ admits an extremal Poincar\'e type metric on the complement of a divisor $Z$ corresponding to the a union $F$ of facets of $Q$, 
then the associated affine linear function $A$ satisfies that
\begin{itemize}
\item $A$ is constant along $F$, and so in particular at its vertices, if $F$ is a single edge,
\item $A$ is constant along the line with vertices $(0,k)$ and $(\frac{1}{k},0)$, if $F$ consists of the two edges lying in the $x$ and $y$-axes.
\end{itemize}
\end{lemma}
\begin{proof} The only part that needs clarification is the last statement regarding the case of two adjacent edges. 
Let $A$ be the associated affine linear function to $Q$ given by the equations~\eqref{hirzebrucheqns}. We are taking the edges in $F$ to be $F_1$ and $F_4$. 

Recall that for the moment polytope $I = [0,\lambda]$ with coordinate $z$ and $F$ being the end-point $0$, the associated affine linear function is
\begin{align*} B_{\lambda} (z) = \frac{6}{\lambda^2} z - \frac{2}{\lambda}.
\end{align*}
Since $F$ is $F_1 \cup F_4$ for $Q$, the condition of equation~\eqref{hyperbolicity-condition} then becomes that 
\begin{align*} A_{| y = 0} &= B_1 (x) + c_1 \\
A_{| x = 0} &= B_k (y ) + c_2 ,
\end{align*}
where the $c_i$ are constants. Thus if $A = ax + b y +c $, we have that
\begin{align*} a&= 6 \\
b &= \frac{6}{k^2}.
\end{align*}
Thus $A (0,k) = c+ \frac{6}{k}$ and $A(\frac{1}{k},0) = c + \frac{6}{k}$, too.  
\end{proof}

We will change our parametrization of $Q$ slightly from the beginning of this Appendix and the above proof, and instead take $Q$ to have vertices 
\begin{align*} v_1 =& (-d, 0), \\
v_2 =&(k,0) ,\\
v_3 =&(0,1) ,\\
v_4 =&(-d, 1).
\end{align*}
Here $d>0$ and $k \in \mathbb{Z}_{\geq 0}$. We can always take $Q$ to be of this form up to scaling. 
We will let $F_1$ be the edge connecting $v_1$ and $v_4$, $F_2$ to be the edge connecting $v_1$ and $v_2$, $F_3$ to be edge connecting $v_2$ 
and $v_3$ and $F_4$ to be the edge connecting $v_3$ and $v_4$.

First, we let $A_i$ be the associated affine linear function the case when $F$ consists of all edges but $F_i$. Then
\begin{align*} A_1  =& - \frac{12}{2d^2 + 2dk + k^2} x - \frac{24 kd (k+d)}{(2d^2 + 2dk + k^2)(6 d^2 + 6 dk + k^2)} y \\ &-  \frac{6(4d^3 -2d^2k -6k^2d - k^3)}{(2d^2 + 2dk + k^2)(6 d^2 + 6 dk + k^2)} , \\
A_2 =& - \frac{12 (3d^2 + 4dk + k^2)}{6 d^2 + 6 dk + k^2} y -  \frac{6(4d^2 + 5 dk + k^2)}{6 d^2 + 6 dk + k^2} , \\
A_3  =& \frac{12}{2d^2 + 2dk + k^2} x + \frac{12 k  (4d^2 + 4dk + k^2)}{(2d^2 + 2dk + k^2)(6 d^2 + 6 dk + k^2)} y \\ &+ \frac{6(8d^3 + 2d^2k -4k^2d - k^3)}{(2d^2 + 2dk + k^2)(6 d^2 + 6 dk + k^2)} , \\
A_4 =& - \frac{12 (3d + 2k)}{6 d^2 + 6 dk + k^2} y -  \frac{6d (k+2d)}{6 d^2 + 6 dk + k^2} .
\end{align*}

Since the associated linear functions depend linearly on the inverse normals, the associated linear function $B_i$ to when $F$ consists of a single edge $F_i$ is therefore given by $B_i = \sum_{j \neq i} A_i$, which is 
\begin{align*} B_1  =& \frac{12}{2d^2 + 2dk + k^2} x - \frac{12 k (4 d^3 + 6d^2 k + 4 dk^2 + k^3 - 4d^2 -4 dk - k^2)}{(2d^2 + 2dk + k^2)(6 d^2 + 6 dk + k^2)} y \\ & +  \frac{6(4d^4 +12d^3 k + 12 d^2 k^2 + 6 dk^3 + k^4 + 8 d^3 + 2d^2k - 4dk^2 - k^3)}{(2d^2 + 2 dk + k^2)(6 d^2 + 6 dk + k^2)} , \\
B_2 =& \frac{12 (3d^2 + 2 dk + k)}{6 d^2 + 6 dk + k^2} y -  \frac{6d (2d + k- 2)}{6 d^2 + 6 dk + k^2} , \\
B_3  =& - \frac{12}{2d^2 + 2dk + k^2} x - \frac{12 k  (4d^3 + 6d^2 k+ 4dk^2 + k^3 + 2d^2 + 2dk)}{(2d^2 + 2dk + k^2)(6 d^2 + 6 dk + k^2)} y \\ &+ \frac{6 ( 4d^4 + 12 d^3 k + 12 d^2 k^2 + 6dk^3 + k^4 - 4d^3 + 2d^2 k + 6dk^2 + k^3)}{(2d^2 + 2dk + k^2)(6 d^2 + 6 dk + k^2)} , \\
B_4 =& - \frac{12 (3d^2 + 4 dk + k^2 - k)}{6 d^2 + 6 dk + k^2} y + \frac{6(4 d^2 + 5dk + k^2 + 2d)}{6 d^2 + 6 dk + k^2} .
\end{align*}

Using Lemma~\ref{PTrestriction}, it suffices to verify whether or not the linear part of $B_i$ is equal at the two vertices of $F_i$. Let $K_i$ be the difference of these two numbers. Then 
\begin{align*} K_1 &= \frac{12k (4 d^3 + 6 d^2 k + 4 dk^2 + k^3 -4 d^2 - 4dk - k^2)}{(2d^2 + 2dk + k^2)(6 d^2 + 6 dk + k^2)}, \\
K_2 &= 0, \\
K_3 &= - \frac{12k (4 d^3 + 6 d^2 k + 4 dk^2 + k^3  - 4 d^2 - 4 dk - k^2)}{(2d^2 + 2dk + k^2)(6 d^2 + 6 dk + k^2)} ,\\
K_4 &= 0.
\end{align*}

The cases of $K_2$ and $K_4$ are the cases when $Z$ is the zero or infinity section in $X$. This is already treated in Corollary~\ref{Hirzebruch} where we know that there exists Poincar\'e type extremal metrics. The above computations then confirm that the condition of equation~\eqref{hyperbolicity-condition} holds, which we also know by general theory regarding extremal Poincar\'e type metrics. 

To prove Corollary~\ref{Hirzebruch-non-poincare} in the case when $F$ consists of a single edge using Lemma~\ref{PTrestriction}, we need to show that $K_1$ and $K_3$ can never be $0$ if $d>0$ and $k \geq 1$. The requirement $k \geq 1$ comes from the fact that if $k=0$, then $X = \mathbb{CP}^1 \times \mathbb{CP}^1$, which is the case we are not considering.

First note that denominator of $K_1$ is always positive, so it will have the same sign as
\begin{align*}4 d^3 + 6 d^2 k + 4 dk^2 + k^3 -4 d^2 - 4dk - k^2  
\end{align*}
which equals
\begin{align*} 4 d^3 + 2 d^2 k + 4 dk (k-1) + k^2(k-1)  + 4 d^2 (k-1) .
\end{align*}
This is always positive as $d>0$ and $k \geq 1$. For $K_3$, note that it equals $-K_1$, hence is always negative.

The remaining case is that of when $F$ consists of two adjacent edges of $Q$, which we take to be $F_1$ and $F_2$. The associated affine linear function $A$ is then $A_3 + A_4$ and by Lemma~\ref{PTrestriction} we need to verify that $K = A(k,0) - A(-d, \frac{1}{k+d})$ can never be $0$. A computation shows that this quantity is given by

\begin{align*} K  &= \frac{12(6d^4+18d^3k+19d^2k^2+8dk^3+k^4+6d^3+6d^2k+3dk^2+k^3)}{(k+d)(2d^2 + 2dk + k^2)(6 d^2 + 6 dk + k^2)}  
\end{align*} 
which is clearly positive. In particular, it can never be $0$.

\end{document}